\documentclass[a4paper,leqno,11pt]{amsart}

\usepackage[english]{babel}

\usepackage{amsfonts}
\usepackage{amsmath}
\usepackage{amsthm}
\usepackage{amssymb}
\usepackage{indentfirst}
\usepackage{color}
\usepackage{tabularx}
\usepackage{graphicx}
\usepackage{esint}
\allowdisplaybreaks

\usepackage[a4paper,top=3.5cm,bottom=1cm,left=2.5cm,right=2.5cm]{geometry}

\usepackage{subcaption}

\usepackage{mathrsfs}
\usepackage{mathtools}
\usepackage{enumitem}
\usepackage[normalem]{ulem}

\definecolor{blue_links}{RGB}{13,0,180} 

\usepackage{hyperref}
\hypersetup{
    colorlinks=true, 
    linktoc=all,     
    linkcolor=blue_links,  
    citecolor=blue_links,
    urlcolor=blue_links,
}

\usepackage[utf8]{inputenc}
\usepackage[T1]{fontenc}
\usepackage[english]{babel}

\numberwithin{equation}{section}

\theoremstyle{plain}
\begingroup
\theoremstyle{plain}
\newtheorem{theorem}{Theorem}[section]
\newtheorem{corollary}[theorem]{Corollary}
\newtheorem{proposition}[theorem]{Proposition}
\newtheorem{lemma}[theorem]{Lemma}
\theoremstyle{definition}
\newtheorem{definition}[theorem]{Definition}

\theoremstyle{remark}
\newtheorem{remark}[theorem]{Remark}
\endgroup

\theoremstyle{definition}
\theoremstyle{remark}

\mathchardef\emptyset="001F

\newcommand{\stress}{\boldsymbol{\sigma}} 
\newcommand{\strain}{\boldsymbol{\varepsilon}} 

\newcommand{\R}{\mathbb{R}}
\newcommand{\E}{\mathcal{E}}

\newcommand{\C}{\mathbb{C}}
\newcommand{\e}{\mathrm{E}}

\newcommand{\di}{\mathrm{d}}

\newcommand{\A}{\mathcal{A}}
\newcommand{\HH}{\mathcal{H}}
\newcommand{\Om}{\Omega}

\newcommand{\M}{\mathbb{M}}

\newcommand{\F}{\mathcal{F}}

\newcommand{\J}{\mathcal{J}}
\newcommand{\p}{\boldsymbol{p}}
\newcommand{\D}{\mathcal{D}}
\newcommand{\VV}{\mathcal{V}}

\newcommand{\argmin}{\mathrm{argmin}}
\newcommand{\pp}{\boldsymbol{P}}
\newcommand{\brho}{\boldsymbol{\rho}}
\newcommand{\bpi}{\boldsymbol{\pi}}
\newcommand{\tu}{\tilde{u}}
\newcommand{\tstrain}{\tilde{\strain}}
\newcommand{\tp}{\tilde{\p}}
\newcommand{\bxi}{\boldsymbol{\xi}}
\newcommand{\btheta}{\boldsymbol{\theta}}

\definecolor{dred}{rgb}{.8,0,0}
\definecolor{ddmagenta}{rgb}{0.7,0,0.9}
\definecolor{ddcyan}{rgb}{0,0.2,1.0}
\definecolor{Orchid}{rgb}{0.7,0.4,0}

\newcommand{\BV}{\mathrm{BV}}

\newcommand{\coloneq }{\hspace{1pt}\raisebox{0.74pt}{\scalebox{0.8}{:}}\hspace{-2.2pt}=}

\newcommand{\om}{\omega}

\newcommand{\ha}{\mathbb{H}}
\newcommand{\q}{\boldsymbol{\mathrm{Q}}}
\newcommand{\qq}{\boldsymbol{q}}
\newcommand{\eeta}{\boldsymbol{\eta}}
\newcommand{\trho}{\tilde{\brho}}
\newcommand{\tpi}{\tilde{\bpi}}

\title{Topology optimization for quasistatic elastoplasticity}
\author[S. Almi]{Stefano Almi}
\address[Stefano Almi]{Faculty of Mathematics, University of Vienna, 
Oskar-Morgenstern-Platz 1, 1090 Wien, Austria.}
\email{stefano.almi@univie.ac.at}
\urladdr{http://www.mat.univie.ac.at/$\sim$almi}

\author[U. Stefanelli]{Ulisse Stefanelli} 
\address[Ulisse Stefanelli]{Faculty of Mathematics, University of
  Vienna, Oskar-Morgenstern-Platz 1, A-1090 Vienna, Austria,
Vienna Research Platform on Accelerating
  Photoreaction Discovery, University of Vienna, W\"ahringerstra\ss e 17, 1090 Wien, Austria,
 \& Istituto di
  Matematica Applicata e Tecnologie Informatiche {\it E. Magenes}, via
  Ferrata 1, I-27100 Pavia, Italy
}
\email{ulisse.stefanelli@univie.ac.at}
\urladdr{http://www.mat.univie.ac.at/$\sim$stefanelli}

\date{\today}
 \subjclass[2010]{74C05, 	
 			   74P10,  
			   49Q10,  	
			   49J20, 	
			   49K20,  	
			   }

\begin{document}
\maketitle

\begin{abstract}
 
Topology optimization is concerned with the identification of optimal 
shapes of deformable bodies with respect to given target functionals. 
The focus of this paper is on a topology optimization problem for a
time-evolving elastoplastic medium under kinematic hardening. We adopt a phase-field approach and argue by subsequent
approximations, first by discretizing time and then by regularizing
the flow rule. Existence of optimal shapes is proved both at the
time-discrete and time-continuous level, independently of the
regularization. First order optimality conditions are firstly obtained
in the regularized time-discrete setting and then proved to pass to
the nonregularized time-continuous limit.  
The phase-field approximation is shown to pass to its sharp-interface
limit via an evolutive variational convergence argument.
\end{abstract}

\section{Introduction}
\label{s.introduction} 

The design of a mechanical piece is often driven by an optimization
process. The mechanical response of a given shape is tested against a number
of criteria, possibly including weight, material and manufacturing
costs, topological, and geometrical features. The tenet of Topology
Optimization (TO in the following) is that of identifying the optimal
shape of a body $E\subset \Omega$ within a given design region $\Omega \subset \R^n$ with
respect to a given target functional. This optimality depends on the
mechanical response of the body with respect to  the imposed actions
(boundary displacements, forces, tractions) and is hence a function of
$E$ itself. As such, the target functional is minimized with respect
to the shape $E$. This general setting is common to most TO problems
and arises ubiquitously, from mechanical engineering, to
aerospace and automotive, to architectural engineering, to
biomechanics \cite{Bensdsoe}.

In this paper, we investigate a TO problem for a linearized elastoplastic
medium showing kinematic hardening. The mechanical state of the  the system is described by its
time-dependent 
{\it displacement} $u(x,t) \in
 \R^n$    and its {\it plastic strain} $\p(x,t) \in \R^{n \times n}_{\rm
  dev}$ (symmetric deviatoric tensors). We assume that the {\it total strain} $ \e u
= (\nabla u + \nabla u^T)/2$ of the body can be additively decomposed into an
elastic part $\strain \in \R^{n \times n}_{\rm
  sym}$ (symmetric tensors) related to the stress state of the
material and the plastic part $\p$,
namely,
\begin{equation}\label{response}
  \e u = {\strain} + \p.
  \end{equation}

The actual position of the body within the design domain $\Omega$ is identified
by means of the scalar function $z \colon \Omega \to [0,1]$. In particular,
the level set $\{z=1\}$ indicates the position of the body to be determined via
TO.
In the following, we shall interpret $z$ as a phase indicator and
assume the region not occupied by the body to be filled by a very compliant
medium, again of elastoplastic type. 
This approach is rather classical
\cite{Allaire2002} and allows for a sound mathematical treatment. In
particular, by taking the material parameters to be  suitably dependent on
$z$, all state quantities will be assumed to be defined in the whole
design region
$\Omega$. The mechanical problem will be hence addressed in the fixed
domain $\Omega$ and the actual position of the body to be determined
via TO is identified via $z$. In the following, we refer to $z$ as
{\it phase field} or {\it phase}, alluding to the interpretation of the material in
$\Omega$ as a two-phase system.

We assume linear material response, namely, the stress $\stress$ of the medium is obtained as
$\stress = {\mathbb C}(z) {\strain }$ where ${\mathbb C}(z)$ is the positive-definite symmetric
{\it elasticity tensor}. On the other hand, the
time-evolution of $\p$ is driven by the normality {\it flow rule}
\begin{equation}\label{flow}
  d(z)\partial |\dot \p| \ni \stress - {\mathbb H}(z)\p\,.
  \end{equation}
Here, $d(z)>0$ represents the {\it yield stress} which activates
plasticization and the symbol~$\partial $ stands for the set-valued subdifferential in
the sense of Convex Analysis, namely $\partial|\dot \p| = \dot \p/|\dot
\p|$ for $\dot \p \not = { 0}$ and $\partial | 0| =
\{\qq \in \R^{n\times n}_{\rm dev} \  :  \ |\qq|\leq 1\}$.  Eventually, ${\mathbb H}(z)\p$  represents the {\it
  backstress} due to kinematic hardening, here modulated by the positive-definite symmetric {\it kinematic hardening tensor}
${\mathbb H}(z)$~\cite{Han}. The flow rule~\eqref{flow} is of course
to be complemented by an initial condition for~$\p$ which we will take
as $\p(0)=0$ for simplicity. 

The body is assumed to be clamped on the portion~$\Gamma_{D}$ of
the boundary $\partial \Omega$  and to evolve quasistatically under
the combined effect of
the time-dependent body force $\ell(z) f$ and of the time-dependent boundary traction $g$ on the portion
$\Gamma_{N}$ of $\partial \Omega$. The {\it quasistatic equilibrium system}
hence reads 
\begin{equation}
  \label{quasistatic}
  \nabla \cdot \boldsymbol{\sigma} + \ell(z) f(t) = 0 \ \ \text{in} \ \Omega,  \quad u =0 \ \
  \text{on} \ \Gamma_{\rm D}, \quad \stress n = g(t) \ \ \text{on} \
  \Gamma_{\rm N}
\end{equation}
where $n$ indicates the outward pointing normal to $\partial
\Omega$ and the term $\ell(z)$ corresponds to the {\it density} of the
medium at phase $z$.  Under suitable assumptions on data, see Section
\ref{s.setting} below, for each $z \in L^\infty(\Omega)$ one can
uniquely identify a trajectory $t \in [0,T] \mapsto (u(t),\p(t))$
solving the {\it quasistatic elastoplastic evolution system}
\eqref{response}--\eqref{quasistatic} in a suitable weak sense, see
Definition \ref{d.quasistatic} and the comments thereafter.

Our aim is to identify phases $z$ which, together with their
associated quasistatic elastoplastic evolutions  $t \in [0,T] \mapsto
(u(t),\p(t))$, minimize the 
compliance-type functional ${\mathcal J_\delta}(z,u)$ given by
\begin{align}\label{e.target}
\J_\delta(z, u) := &\int_{\Om} \ell(z) \,  f(T) {\, \cdot\,} u(T) \, \di x + \int_{\Gamma_{N}} g(T) {\, \cdot\,} u(T) \, \di \HH^{n-1} 
\\
&\nonumber
-  \int_{0}^{T} \int_{\Om} \ell(z) \, \dot{f}( \tau) {\, \cdot\,} u (\tau) \, \di x \,\di \tau  - \int_{0}^{T} \int_{\Gamma_{N}} \dot{g}(\tau) {\, \cdot\, } u(\tau) \, \di \HH^{n-1} \di \tau 
\\
&
+ \int_{\Om} \frac{\delta}{2} | \nabla{z}|^{2} + \frac{z^{2}(1-z)^{2}}{2\delta} \, \di x \,.\nonumber
\end{align}
The first four terms in~$\J_\delta$ measure the compliance of the medium,
integrated over the reference time interval. The last two terms in~$\J_\delta$ are the classical Modica-Mortola functional~\cite{MR0445362}. Under the modulation of the user-defined small parameter
$\delta>0$, the gradient term penalizes changes in~$z$ whereas the
double-well term favours the values~$0$ and~$1$. The combination of the
two last terms in~$\J_\delta$ expresses the competition between phase
separation and minimization of transitions between
phases. In the limit $\delta\to 0$ one recovers a sharp-interface
situation, where minimizing phases $z$ take exclusively values~$0$ or~$1$ and the length of the interface separating the two regions~$\{z=0\}$ and $\{z=1\}$ is penalized, see Section \ref{s.gamma}.

Our main TO problem reads
\begin{equation}
  \label{TOmain}
  \min\big\{{\mathcal J}_\delta(z,u) \ : \  (u,\p) \ \ \text{solve
      \eqref{response}--\eqref{quasistatic} given $z$}\big\}\,.
\end{equation}
The main contribution of this paper is in proving that this TO problem
admits solutions, in investigating
its discretization and regularization, and in providing first-order
optimality conditions.

More precisely, in order to tackle the TO problem \eqref{TOmain} we proceed by subsequent
approximations. At first, we investigate a time-discrete version of
\eqref{TOmain}, where continuous-in-time states are replaced by the
time-discrete solutions of the incremental elastoplastic problem, see
Definition \ref{d.discrete-evolution}. The time-discrete TO problem is
proved to admit solutions (Proposition \ref{p.2}) which converge to solutions of the
time-continuous \eqref{TOmain} as the fineness of the time partition
goes to $0$ (Corollary~\ref{c.1}). 

The time-discrete TO problem is then regularized by replacing
the nonsmooth term $|\dot \p|$ in the flow rule \eqref{flow} by the
smooth function $h_\gamma(\dot \p) = (|\dot
\p|^2+\gamma^{-2})^{1/2}-1/\gamma$ depending on $\gamma>0$. The
corresponding approximate time-discrete TO problem admits solutions
(Proposition \ref{p.4}) which converge to solutions of the
time-discrete TO problem as $\gamma \to +\infty$ (Corollary
\ref{c.2}). 
Introducing the regularization via $h_\gamma$ is instrumental to obtain
the differentiability of the {\it control-to-state} map $z \mapsto (u,\p )$ which is in turn
needed in order to derive first-order optimality conditions, see also~\cite{Alm-Ste_20,delosReyes, Herzog2, Wachsmut2}. This
differentiability is tackled in Section \ref{s.diff} in the frame of
the approximate time-discrete TO problem (Theorem \ref{t.1}) and
allows to prove corresponding first-order optimality conditions
(Corollary~\ref{c.approx-optimality}). The passage to the limit as
$\gamma \to +\infty$ first and then as the fineness of the time
partition goes to $0$ provide the first-order optimality conditions
for the time-discrete TO problem (Theorem \ref{t.discrete-optimality})
and the time-continuous TO problem~\eqref{TOmain} (Theorem~\ref{t.continuous-optimality}), which are the main results of the
paper.

All the above mentioned results are obtained in the setting of the
phase-field approximation $\delta >0$. Still, the existence and the
convergence results are valid in the sharp-interface case $\delta=0$
as well and the limit $\delta \to 0$ can be rigorously ascertained. We
give some detail in this direction in Section \ref{s.gamma} for the
time-continuous TO problem \eqref{TOmain}. In particular, we prove
that solutions to \eqref{TOmain} for $\delta>0$ converge to solutions
to \eqref{TOmain} for $\delta=0$ as $\delta \to 0$ by means of an
evolutive $\Gamma$-convergence argument (Proposition
\ref{p.11}). Let us remark however that, due to the
limited regularity of solutions to \eqref{TOmain} for $\delta=0$, first-order optimality
conditions are available for the case $\delta>0$ only.

 Before moving on, let us comment on the literature and put our work in perspective. 
The mathematical TO literature in the static {\it elastic} setting is abundant, see~\cite{Bourdin,Burger,Penzler} and 
\cite{Blank2,Blank1,Carraturo} for a selection of existence results
and first-order optimality conditions in different linear and
nonlinear settings. Results in the {\it elastoplastic} setting are
available in the two-dimensional case, both in the static~\cite{haslinger1, haslinger2, Hlavacek1, Hlavacek3}  and in the
evolutive regime~\cite{Hlavacek2},  but exclusively  under the a priori assumption that
the unknown optimal shape $\{z=1\}$ is Lipschitz regular. 
The beam structure
 and frame optimization was investigated in~\cite{Karkauskas,
   Khanzadi, Pedersen} from the point of view of the existence of
 minimizers. First-order optimality conditions in terms of shape
 derivatives appeared in~\cite[Chapters~4.8 and 4.9]{Sokolowski2} for
 an elastic torsion problem and for the viscoplastic model of
 Perzyna, see also~\cite{Boissier, Maury}. To the best of our
 knowledge, the existence analysis and the study of optimality
 conditions in the corresponding regularity setting are unprecedented
 for quasistatic evolution TO problems for elasto-plasticity.

 On the
 other hand, control problem for quasistatic elastoplasticity have
 already been studied and the reader is referred to the 
 analysis in \cite{Wachsmut1,Wachsmut2,Wachsmut3}, see also the
 general theory in \cite{Rindler,Rindler2}. Compared
 with these contributions, where controls usually are modeled as imposed
 forces, in the frame of TO the action of controls is more involved,
 for they modify the elastic response via material parameters. Correspondingly, our analysis is
 at specific places more involved than that in the above papers, albeit being inspired by the same
 general principles.

 In  our recent paper
\cite{Alm-Ste_20}, we have tackled the three-dimensional
{\it static} kinematic-hardening case and analyzed the existence of
solutions, the
first-order optimality conditions, and the sharp-interface limit.
This indeed sets the basis for the current contribution, which however
focuses on the quasistatic evolutive case. Moving from static state-problem
formulations, based on the minimization of one single functional, to
evolutive formulations, based on the time-continuous limits of
sequences of time-discretizations in the frame of rate-independent
processes \cite{Mielke-Roubicek}  is analytically
challenging. Remarkably, in order to tackle the various limiting
procedures one
has to resort to {\it evolutive} $\Gamma$-convergence
techniques \cite{mrs}, which are more involved than their static
counterparts.

 Let us now present the structure of the paper and of our results:
\begin{itemize}\setlength\itemsep{0.6em}
  \item[] {\bf Section~\ref{s.setting}} is devoted to discussion of the
    model, notation, and assumptions on
data.
 \item[] {\bf Section~\ref{s.approx}} brings to statement of the  
  time-continuous TO problem, as well as of its time-discrete and
  approximate time-discrete versions. Here, we also check existence of
  optimal solutions and convergence of
  time-discrete to time-continuous and approximate time-discrete to
  time-discrete solutions.
   \item[] {\bf Section~\ref{s.gamma}} focuses on the 
   sharp interface limit $\delta \to 0$. In particular, we prove that
   solutions of time-continuous TO
  problem~\eqref{TOmain} converge to solutions of the corresponding
  sharp-interface limiting TO problem for $\delta=0$. This is based on
  an evolutive Modica-Mortola argument. 
   \item[] {\bf Section~\ref{s.diff}} contains the investigation of
     the 
  differentiability of the control-to-state map for the approximate
  time-discrete TO problem, where $\gamma<+\infty$. Correspondingly, a detailed analysis of
   first-order
  optimality conditions in the approximate
  time-discrete case is presented. 
   \item[] {\bf Section~\ref{s.optimality}} eventually leads to
  first-order optimality  conditions for both time-continuous and the time-discrete
  TO problems. These ensue  by passing to the limit in the
  corresponding ones for the approximate
  time-discrete TO problem from Section~\ref{s.diff}.
   \item[] {\bf The Appendix} features a technical 
  convergence argument  which is used in the study of
  discrete-to-continuous limits for  quasistatic evolutions.
\end{itemize}

\section{Setting of the problem}
\label{s.setting}

 We devote this section to fixing notation and assumptions on
data. In the following,~$\M^n$ indicates the space of 2-tensors in
$n$ dimensions, indicated in bold face in the following, and~$\M^n_S$ is the subspace of
symmetric 2-tensors. The symbol~$\M^n_D$ indicates symmetric and
deviatoric 2-tensors, namely those with vanishing trace. The symbol
$\cdot$ indicates contraction with respect to all indices. In
particular ${\boldsymbol A}\cdot {\boldsymbol B} = A_{ij}B_{ij}$
and $u\cdot v=u_iv_i$ (summation convention on repeated indices)  for
all ${\boldsymbol A}, {\boldsymbol B}\in\M^n$, $u,v \in \R$.

The elasticity tensor~$\C$ and the
kinematic-hardening tensor~$\ha$  are asked to be {\it isotropic} for all~$z$. In
particular, we ask for  
\begin{equation}\C( z )  := 2\mu(z) \mathbb I + \lambda(z) ( \boldsymbol{\mathrm{I}}
\otimes  \boldsymbol{\mathrm{I}}), \quad  \ha(z):= h(z) \mathbb
I  \label{eq:data3}
\end{equation}
where~$\lambda(z)$ and~$\mu(z)$ are the Lam\'e coefficients,~$h(z)$  is the
hardening modulus, and $\mathbb I$ and $\boldsymbol{\mathrm{I}}$ denote the identity 4 and
2-tensor, respectively. Isotropy in particular guarantees that~$\C$ and~$\ha$
map~$\M^n_D$ to~$\M^n_D$.

 We assume the  material coefficients to be
differentiable with respect to $z$ and to be defined in all of~$\mathbb R$. In particular, we
ask 
\begin{equation}
\mu, \,  \lambda,  \, h, \, d \in C^1(\mathbb R)\,.\label{eq:data0}
\end{equation}
  We moreover define them as 
constant on $\{z\leq 0\}$ and $\{ z\geq 1\}$.  This last provision
allows us to recover the property $z \in [0,1]$ a posteriori, without
the need of enforcing it a-priori as a constraint. The reader is referred
to~\eqref{min2},~\eqref{min3}, and~\eqref{min4} for additional
details. 

 All material  coefficients  are asked   to be positive and bounded, uniformly with
respect to $z$, namely,  we assume that 
\begin{equation}\label{e.hp-bound}
\exists 0<\alpha<\beta<+\infty \ \forall z \in [0,1]: \quad \alpha \leq \mu(z)\,,
\, \lambda(z), \, h(z), \,  d(z) \leq \beta\, .
\end{equation}
This in particular implies that $\C$ and $\ha$ are uniformly
positive definite and bounded, independently of $z$.  Indeed, one
can find  $0< \alpha_{\C} < \beta_{\C}<+\infty$ and $0 <
\alpha_{\ha} < \beta_{\ha} <+\infty$ such that 
\begin{align}
& \alpha_{\C} | \boldsymbol{\e}|^{2} \leq \C(z) \boldsymbol{\e} {\, \cdot\,} \boldsymbol{\e} \leq \beta_{\C} | \boldsymbol{\e}|^{2} \qquad \text{for every $\boldsymbol{\e} \in \M^{n}_{S}$}\,, \label{e.C} \\[1mm]
&\alpha_{\ha} | \q|^{2} \leq \ha(z) \q {\, \cdot\,} \q \leq \beta_{\ha} |\q |^{2} \qquad \text{for every $\q \in \M^{n}_{D}$} \label{e.H}\,.
\end{align} 

 The design domain $\Om\subset\R^{n}$ is taken to be open, connected, and
with Lipschitz boundary~$\partial \Om$.  We also fix  two
subsets~$\Gamma_{N}, \Gamma_{D}$ of~$\partial \Om$, which from now on
will be referred to as {\it Neumann} and {\it Dirichlet} part of~$\partial \Om$, respectively. We assume $\Gamma_D,\, \Gamma_N \subset \partial \Omega $ to be open in
the topology of~$\partial \Omega$ with $\overline \Gamma_N \cap
\overline \Gamma_D=\emptyset$,  $\overline \Gamma_N \cap
\overline \Gamma_D =\partial \Omega$,  and where~$\overline \Gamma_N$ and~$\overline \Gamma_D$ are closures in~$\partial \Omega$. We moreover assume that~$\Gamma_D$ has positive surface measure, namely~$\mathcal
H^{n-1} (\Gamma_D)>0$, where the latter is the $(n{-}1)$-Hausdorff surface measure in~$\R^n$. 
Furthermore, we suppose that~$\Omega \cup \Gamma_{N}$ is regular {\it in the sense of Gr\"oger}~\cite[Definition~2]{Groeger}, that is, for every $x \in \partial\Om$ there exists an open neighborhood~$U_{x} \subseteq \R^{n}$ of~$x$ and a bi-Lipschitz map~$\Psi_{x} \colon U_{x} \to \Psi(U_{x})$ such that $\Psi_{x} ( U_{x} \cap (\Om \cup \Gamma_{N}))$ coincides with one of the following sets:
\begin{eqnarray*}
&& \displaystyle V_{1} \coloneq \{ y \in \R^{n} : \, | y | \leq 1 , \, y_{n} <0\}\,, \\ [1mm]
&& \displaystyle V_{2} \coloneq \{ y \in \R^{n} : \, | y | \leq 1, \, y_{n}\leq 0\} \,, \\ [1mm]
&& \displaystyle V_{3} \coloneq \{ y \in V_{2} ; \, y_{n} <0 \text{ or } y_{1} >0\} \,,
\end{eqnarray*}
where $y_{i}$ is the $i$-th component of~$y \in \R^{n}$. This last assumption  is crucially used  in the proof of Theorem~\ref{t.1}. 

For every $w \in H^{1}(\Om; \R^{n})$, we define the set of admissible
 displacements
\begin{align*}
\A(w) := \{ (u, \strain, \p) \in H^{1} ( \Om ; \R^{n} ) \times L^{2} ( \Om ; \M^{n}_{S} ) \times L^{2} ( \Om ; \M^{n}_{D} ) : \, \e u = \strain + \p , \, u= w \text{ on $\Gamma_{D}$} \}\,,
\end{align*}
where $\e u$ denotes the symmetric part of the gradient of~$u$, 
namely $\e u=(\nabla u + \nabla u)^\top/2$. 

As concerns data, we assume the  {\it volume-force density
  per unit mass}  $f  $, the  {\it
  surface-traction density}~$g  $, and the Dirichlet {\it boundary
  displacement}~$w  $, to  satisfy 
\begin{align}
  \label{eq:data}
  f \in H^{1}(0,T; L^{p}(\Om; \R^{n})), \quad g \in H^{1}(0,T; L^{p}(\Gamma_{N};
\R^{n})), \quad w \in
H^{1}(0,T; W^{1,p}(\Om; \R^{n}))
\end{align}
for some given $p\in (2,+\infty)$. Additionally, we assume that 
\begin{align}
  \label{eq:data2}
  f(0)=g(0)=w(0)=0\,.
\end{align}
This last requirement ensures the compatibility of the initial datum 
\begin{align}
  \label{eq:initial}
  (u(0),\strain(0),\p(0))=(0,0,0)\,.
\end{align}
The assumptions \eqref{eq:data3}--\eqref{eq:initial} of this Section
are assumed throughout the paper, without further explicit mention.


\section{The Topology  Optimization problem and its approximations}\label{s.approx}

 This section is devoted to make the topology optimization problem
precise and present its time-discretization and regularization. In
particular, we prove the existence of optimal phase-fields $z$ in the various
settings, which are then connected via variational convergence
arguments. 

 Let us start by defining {\it quasistatic evolutions} of the
elastoplastic system given the phase-field~$z$. We follow here the energetic formulation of quasistatic evolutions~\cite{Mielke-Roubicek}, in which the elastoplastic system is driven by
energy-storage and energy-dissipation mechanism, calling for 
the definition of the {\it energy}~$\E$ and the {\it dissipation}~$\D$. We define
\begin{align*}
& \E(t, z, u, \strain,  \p) := \frac12 \int_{\Om} \C(z)  \strain {\,
  \cdot\,} \strain \,\di x + \frac12 \int_{\Om} \ha (z) \p{\, \cdot\,}
  \p \, \di x - \int_{\Om} \ell(z) f(t) {\, \cdot\,} u\, \di x -
  \int_{ \Gamma_N} g(t) {\, \cdot\,} u \, \di \HH^{n-1}
\\
&
\D(z, \qq) := \int_{\Om} d(z) | \qq| \, \di x 
\end{align*}
for every $z \in L^{\infty}(\Om)$, every $(u, \strain, \p) \in
\A(w(t))$, and every $\qq \in L^{1}(\Om; \M^{n}_{D})$. For  given
 $\p\colon [0,T] \to L^{2}(\Om; \M^{n}_{D})$ and~$z \in
L^{\infty}(\Om)$, we further define the {\it total dissipation functional}
\begin{align*}
\mathcal{V}([0,t]; z, \p(\cdot)) := \sup\, \bigg\{ \sum_{t_{j} \in \mathcal{P}} \D \big( z, \p(t_{j}) - \p(t_{j-1})\big) : \text{ $\mathcal{P}$ is a partition of $[0,t]$} \bigg\} \,.
\end{align*}
 With these ingredients at hand, we are able to pose the following definition. 

\begin{definition}[Quasistatic evolution given $z$]\label{d.quasistatic}
 Let $z \in L^\infty(\Om)$ be given.  We say that a triple $(u,
\strain, \p) \colon [0,T] \to H^{1}(\Om; \R^{n}) \times L^{2}(\Om;
\M^{n}_{S}) \times L^{2}(\Om; \M^{n}_{D})$ is a \emph{quasistatic
  evolution given $z$}  if $(u(0) , \strain(0) , \p(0) ) = ( 0, 0, 0)$ and the following conditions hold:
\begin{itemize}

\item[\emph{(i)}] for every $t \in[0,T]$ and every $(\hat u, \hat
  \strain, \hat \p) \in \A(w(t))$
\begin{equation}\label{e.1}
\E(t, z, u(t), \strain(t), \p(t)) \leq \E(t, z, \hat u, \hat \strain,
\hat \p) + \D(z, \hat \p - \p(t))\,;
\end{equation}

\item [\emph{(ii)}] for every $t \in[0,T]$:
\begin{align}\label{e.2}
\E(t, & z, u(t), \strain(t), \p(t)) + \mathcal{V}([0,t]; z, \p(\cdot))
\\
&
=  \int_{0}^{t} \int_{\Om} \C(z) \strain(\tau) {\, \cdot\,} \e \dot{w}(\tau) \, \di x \, \di \tau 
-  \int_{0}^{t} \int_{\Om} \ell(z) \dot{f}(\tau) {\,\cdot\,} u(\tau) \, \di x \, \di \tau \nonumber
\\
&\qquad
- \int_{0}^{t} \int_{\Om} \ell(z) f(\tau) {\, \cdot\,} \dot{w}(\tau) \, \di x \, \di \tau 
- \int_{0}^{t} \int_{\Gamma_{N}} \dot{g}(\tau) {\, \cdot\,} u (\tau) \,\di \HH^{n-1}\,  \di \tau \nonumber
\\
&
\qquad 
- \int_{0}^{t} \int_{\Gamma_{N}} g( \tau) {\, \cdot\,} \dot{w}(\tau) \, \di \HH^{n-1} \, \di \tau\,. \nonumber
\end{align}
\end{itemize}
\end{definition}

Note that conditions \eqref{e.1}--\eqref{e.2} are equivalent to the
classical weak formulation of the quasistatic elastoplastic problem \eqref{response}--\eqref{quasistatic}. In
particular, a trajectory $(u(\cdot), \strain(\cdot),
\p(\cdot))$  fulfilling the initial
condition is a quasistatic evolution in the sense
of Definition \ref{d.quasistatic} if and only if 
$$ \int_{\Om} \C(z) (\e u(t) - \p(t))  {\,
  \cdot\,} \e v  \,\di x = \int_{\Om} \ell(z) f(t) {\, \cdot\,} v\, \di x +
  \int_{ \Gamma_N} g(t) {\, \cdot\,} v \, \di \HH^{n-1}$$
for all $v \in H^1(\Omega;\R^n)$ with $v = 0$ on $\Gamma_D$ and all $t \in [0,T]$ and the flow rule \eqref{flow} holds
almost everywhere. As such, the classical elastoplasticity theory
\cite{Han} ensures that, for every  $z \in L^{\infty}(\Om)$ there
exists unique a quasistatic evolution~$(u(\cdot), \strain(\cdot),
\p(\cdot))$  in the sense of Definition \ref{d.quasistatic}. In
fact, for such trajectory one can also check that  $(u(\cdot), \strain(\cdot), \p(\cdot)) \in H^{1}(0,T; H^{1}(\Om; \R^{n}) \times L^{2}(\Om; \M^{n}_{S}) \times L^{2}(\Om; \M^{n}_{D}))$, so that the dissipative term in the first line of~\eqref{e.2} can be rewritten as
\begin{displaymath}
\mathcal{V}([0, T ]; z, \p(\cdot)) = \int_{0}^{T} \int_{\Om} d(z) \, | \dot{\p}(t) | \, \di x \, \di t \,.
\end{displaymath}

We further note that the initial condition $(u(0), \strain(0), \p(0))= (0,0,0)$ has been fixed in such a way that the elastoplastic body~$\Om$ is at equilibrium at time $t=0$.
 
The TO problem consists in minimizing the compliance-type
target functional $\J_\delta(z,u(\cdot))$ from~\eqref{e.target} under the
constraint that~$u(\cdot)$ is the first component of the quasistatic
evolution given~$z$. In particular, we are interested in the following
\begin{align}\label{min2}
&\min\big\{\J_\delta(z,u(\cdot) ) : \ z \in H^1(\Om)   \text{ and} \    (u(\cdot), \strain(\cdot),
\p(\cdot)) \text{ is a quasistatic evolution  given $z$}\big\} \,. 
\end{align}

The existence of an optimal phase-field $z $ solving~\eqref{min2} can
be proved by applying the Direct Method as we show in the next
proposition.

\begin{proposition}[Existence]\label{p.1}
 The \emph{TO} problem~\eqref{min2} admits a solution. In particular, every solution~$z$ satisfies $0 \leq z \leq 1 $ almost everywhere in~$\Om$.
\end{proposition} 

\begin{proof}
Note that $\J_\delta(z,u(\cdot))>-\infty$ for all $z \in
H^{1}(\Om) $ and for the corresponding quasistatic evolution
$(u(\cdot),\strain(\cdot),\p(\cdot))$. Let  $z_{j} \in
H^{1}(\Om)$ be a minimizing sequence for \eqref{min2}. By
the assumptions on~$\C$, $\ha$, $d$, and~$\ell$, we may assume without
loss of generality that $z_{j} \in [0,1]$ almost everywhere,
so that, up to a not relabeled subsequence, $z_{j}
\rightharpoonup z$ weakly in $H^{1}(\Om)$ and $0 \leq z \leq 1$
almost everywhere. Let us denote by $(u_{j} (\cdot), \strain_{j}
(\cdot), \p_{j}(\cdot))$ the quasistatic evolution given~$z_{j}$. In view of the energy balance~\eqref{e.2} and of the hypotheses~\eqref{e.hp-bound}--\eqref{e.H}, $(u_{j}(\cdot), \strain_{j} (\cdot), \p_{j} (\cdot))$ is bounded in~$L^{\infty}(0,T; H^{1} ( \Om; \R^{n}) \times L^{2} ( \Om; \M^{n}_{S}) \times L^{2} (\Om; \M^{n}_{D}) )$ and~$\dot{\p}_{j}$ is bounded in $L^{1}(0,T; L^{1}(\Om; \M^{n}_{D}))$. Therefore, by Helly's Selection Principle, $\p_{j}(t) \rightharpoonup \p(t)$ weakly in~$L^{2}(\Om; \M^{n}_{D})$ for every $t \in [0,T]$ for some $\p \in L^{\infty} (0,T; L^{2}(\Om; \M^{n}_{D}))$.

Let us fix $t\in[0,T]$. By the boundedness of~$u_{j}(t)$ and
of~$\strain_{j}(t)$, we may assume that, up to a not relabeled
subsequence, $u_{j}(t) \rightharpoonup u(t)$ weakly in~$H^{1}(\Om;
\R^{n})$ and $\strain_{j}(t) \rightharpoonup \strain (t) $ weakly
in~$L^{2}(\Om; \M^{n}_{S})$. For every $(\hat u, \hat \strain, \hat
\p) \in \A(w(t))$, we test the equilibrium condition~\eqref{e.1} for
$(u_{j}(t), \strain_{j}(t), \p_{j}(t))$ by the triple
\begin{displaymath}
 (\hat{u}_j(t), \hat{\strain}_j(t), \hat{\p}_j(t)) := ( \hat u +
u_{j}(t) - u (t), \hat \strain + \strain_{j}(t) - \strain (t) , \hat
\p + \p_{j}(t) - \p(t))  \in \A(w(t)).
\end{displaymath}
By exploiting the quadratic character of $\E$ and passing to the
limit as $j\to +\infty$ we
deduce that~$(u(t), \strain (t), \p(t))$ is the unique solution
of~\eqref{e.1}. In particular, the whole sequence $(u_{j}(t), \strain_{j} (t),
\p_{j} (t))$ converges to~$(u(t), \strain (t), \p(t))$ weakly
in~$H^{1}(\Om; \R^{n}) \times L^{2}(\Om; \M^{n}_{S}) \times L^{2}(\Om;
\M^{n}_{D})$. Moreover, $(u_{j}, \strain_{j}, \p_{j})$ converges
weakly$^*$ in $L^{\infty}(0,T; H^{1}(\Om; \R^{n}) \times L^{2}(\Om;
\M^{n}_{S}) \times L^{2}(\Om; \M^{n}_{D}))$ to~$(u, \strain,
\p)$. This last convergence implies that for every~$t \in[0,T]$
\begin{align*}
\E(t,  z, u(t), \strain(t),  \p(t)) & + \mathcal{V}([0,t]; z, \p(\cdot))  
\\
&
\leq  \int_{0}^{t} \int_{\Om} \C(z) \strain(\tau) {\, \cdot\,} \e \dot{w}(\tau) \, \di x \, \di \tau - \int_{0}^{t} \int_{\Om}\ell(z)  \dot{f}(\tau) {\,\cdot\,} u(\tau) \, \di x \, \di \tau 
\\
&
\qquad - \int_{0}^{t} \int_{\Om} \ell(z) f(\tau) {\, \cdot\,} \dot{w}(\tau) \, \di x \, \di \tau 
- \int_{0}^{t} \int_{\Gamma_{N}}  \dot{g}(\tau) {\, \cdot\,} u (\tau) \,\di \HH^{n-1}\,  \di \tau 
\\
&\qquad 
- \int_{0}^{t} \int_{\Gamma_{N}}  g( \tau) {\, \cdot\,} \dot{w}(\tau) \, \di \HH^{n-1} \, \di \tau \,.
\end{align*}
The opposite inequality can be recovered by exploiting the equilibrium
condition~\eqref{e.1} by applying
\cite[Prop. 5.7]{Mielke05}. Hence, the triple $(u(\cdot), \strain
(\cdot), \p(\cdot))$ is the unique quasistatic evolution  given~$z \in H^{1}(\Om; [0,1])$.
 As the target functional~$\J_\delta$ is lower semicontinuous, we deduce
that~$z$ is a solution of~\eqref{min2}.

The second part of the statement is clear in view of our hypotheses on~$\C$, $\ha$, $d$, and~$\ell$.
\end{proof}

 The existence of solutions to \eqref{min2} being proved, in the
remainder of this section we focus on their approximation. At first, we
discretize the quasistatic evolution constraint in time. Subsequenly,
we regularize the flow rule. This will be instrumental to obtaining
 first-order optimality conditions, which we then tackle in Sections \ref{s.diff}--\ref{s.optimality}.

Let us hence start by a time discretization of the quasistatic
evolution problem 
(see also~\cite{Rindler, Rindler2, Wachsmut1}). Precisely, fixed~$k \in \mathbb{N}$ and $\tau_{k}:= T/k$, we define for $i = 0, \ldots, k$ the time nodes $t^{k}_{i}:= i \tau_{k}$ and the functions
\begin{equation}\label{e.fgw}
f^{k}_{i} := f(t^{k}_{i})\,, \quad  g^{k}_{i} := g(t^{k}_{i})\,, \quad  w^{k}_{i} := w(t^{k}_{i})\,.
\end{equation}
For later use, we further set for $t \in [t^{k}_{i-1}, t^{k}_{i})$
\begin{equation}\label{e.fgw2}
\begin{split}
& f_{k}(t) := f^{k}_{i-1} + \frac{t - t^{k}_{i-1}}{\tau_{k}} ( f^{k}_{i} - f^{k}_{i-1}) \,,\qquad  g_{k}(t) := g^{k}_{i-1} + \frac{t - t^{k}_{i-1}}{\tau_{k}} ( g^{k}_{i} - g^{k}_{i-1})\,,
\\
&
w_{k}(t) := w^{k}_{i-1} + \frac{t - t^{k}_{i-1}}{\tau_{k}} ( w^{k}_{i} - w^{k}_{i-1})\,.
\end{split}
\end{equation} 
Notice that $( f_{k}, g_{k}, w_{k})$ converges to~$(f, g, w) $ in $H^{1}(0,T; L^{2}(\Om; \R^{n}) \times L^{2}(\Gamma_{N}) \times H^{1}(\Om; \R^{n}) )$. We define the time-discrete energy functional
\begin{align*}
& \E_{k}(t^{k}_{i},z,  u, \strain, \p) := \frac{1}{2} \int_{\Om} \C(z)  \strain{\, \cdot\,} \strain \, \di x + \frac12 \int_{\Om} \ha(z) \p {\, \cdot\,} \p \, \di x - \int_{\Om} \ell (z) f^{k}_{i} { \, \cdot\,} u \, \di x - \int_{\Gamma_{N}} g^{k}_{i} {\, \cdot\,} u \, \di \HH^{n-1}
\end{align*}
and the discrete target functional
\begin{align}\label{e.target2}
\J_{k,\delta} (z, (u_{i})_{i=0}^{k} ) :=   & \int_{\Om}\ell (z) \,  f^{k}_{k} {\, \cdot\,} u_{k} \, \di x + \int_{\Gamma_{N}} g^{k}_{k} {\, \cdot\,} u_{k} \, \di \HH^{n-1} 
\\
&
-  \sum_{i=0}^{k-1} \bigg( \int_{\Om} \ell(z) \big( f_{i+1}^{k} - f^{k}_{i} \big) {\, \cdot\,} u_{i}  \, \di x + \int_{\Gamma_{N}} \big( g^{k}_{i+1} - g^{k}_{i} \big) {\, \cdot\,} u_{i}  \, \di \HH^{n-1} \bigg) \nonumber
\\
&
 + \int_{\Om} \frac{\delta}{2} | \nabla{z}|^{2} + \frac{z^{2}(1-z)^{2}}{2\delta} \, \di x \,, \nonumber
\end{align}
 for $z \in H^1(\Om)\cap L^\infty(\Om)$ and $(u_{i}) _{i=0}^{k} \in \big( H^{1}(\Om; \R^{n})\big)^{k+1}$. In the sequel, we will use a similar notation for $(\strain_{i})_{i=0}^{k} \in \big(L^{2}(\Om; \M^{n}_{S})\big)^{k+1}$ and $ (\p_{i})_{i=0}^{k} \in \big(L^{2}(\Om; \M^{n}_{D})\big)^{k+1}$

In the minimization of the time-discrete target
functional~$\J_{k,\delta}$ we ask the triple $(u_{i}, \strain_{i},
\p_{i})_{i=0}^{k} \in \big( H^{1}(\Om; \R^{n}) \times L^{2}(\Om;
\M^{n}_{S}) \times L^{2}(\Om; \M^{n}_{D})\big)^{k+1}$ 
to be a
time-discrete quasistatic evolution given $z$, whose definition is given here below.

\begin{definition}[Time-discrete quasistatic evolution given
  $z$]\label{d.discrete-evolution}
Let $z\in L^\infty(\Om)$ be given and
$f^{k}_{i}, g^{k}_{i}, w^{k}_{i}$ be defined as in~\eqref{e.fgw}. We say that $(u_{i},
\strain_{i}, \p_{i})_{i=0}^{k} \in \big( H^{1}(\Om; \R^{n}) \times
L^{2}(\Om; \M^{n}_{S}) \times L^{2}(\Om; \M^{n}_{D})\big)^{k+1}$ is a
\emph{time-discrete quasistatic evolution given $z$} if $(u_{0}, \strain_{0}, \p_{0}) = (0,0,0)$ and the following holds: for every $i = 1,\ldots, k$, $(u_{i}, \strain_{i}, \p_{i}) \in \A(w^{k}_{i})$ and
\begin{equation}\label{e.3}
\E_{k} (t^{k}_{i}, z,  u_{i}, \strain_{i}, \p_{i}) + \D(z, \p_{i} - \p_{i-1}) \leq \E_{k} ( t^{k}_{i}, z, \hat u, \hat \strain, \hat \p) + \D(z, \hat \p - \p_{i-1})
\end{equation}
for every $(\hat u, \hat \strain, \hat \p) \in \A(w^{k}_{i})$.
\end{definition}

As a consequence of~\eqref{e.3} we have that every time-discrete quasistatic evolution $(u_{i}, \strain_{i}, \p_{i})_{i=0}^{k}$ satisfies the following energy inequality: for every $i = 1, \ldots, k$
\begin{align}\label{e.4}
\E_{k} (t^{k}_{i}, & z, u_{i}, \strain_{i}, \p_{i} ) + \sum_{j=1}^{i} \D( z , \p_{j} - \p_{j-1} )
\\
&
 \leq   \sum_{j=1}^{i}  \int_{\Om} \C (z) \strain_{j-1} {\, \cdot\,} \e ( w_{j}^{k} - w_{j-1}^{k} ) \, \di x 
-\int_{\Om} \ell(z) ( f^{k}_{j} - f^{k}_{j-1} ) {\,\cdot\,} u_{j-1} \, \di x \nonumber
\\
&\qquad 
-  \int_{\Om}\ell(z)  f_{j} {\, \cdot\,} ( w^{k}_{j} - w^{k}_{j-1} ) \, \di x  
+ \int_{\Gamma_{N}} ( g^{k}_{j} - g^{k}_{j-1} ) {\, \cdot\,} u_{j-1} \,\di \HH^{n-1}  \nonumber
\\
&\qquad 
+ \int_{\Gamma_{N}} g_{j}^{k} {\, \cdot\,} ( w^{k}_{j} - w^{k}_{j-1} ) \, \di \HH^{n-1}  
  +  \int_{\Om} \C (z) \e ( w_{j}^{k} - w_{j-1}^{k} )   {\, \cdot\,} \e ( w_{j}^{k} - w_{j-1}^{k} ) \, \di x \,. \nonumber
\end{align}

Furthermore, we note that a time-discrete quasistatic evolution can always be constructed by iteratively solving the minimum problems
\begin{equation}\label{e.mma}
\min \, \{ \E_{k} ( t^{k}_{i}, z,  u, \strain, \p) + \D(z, \p - \p_{i-1}) : \, (u, \strain, \p) \in \A(w^{k}_{i})\} \,,
\end{equation}
for $i \geq 1$, where we have set $(u_{0}, \strain_{0}, \p_{0}) = (0,0,0)$. In particular, given the data $f$, $g$, and~$w$, the time-discrete quasistatic evolution is unique, as the solution of the minimum problem~\eqref{e.mma} is unique.

The time-discrete TO problem reads as
\begin{align}\label{min3}
\min\, \{\J_{k,\delta}(z, (u_{i})_{i=0}^{k}): &\text{ $z \in H^{1}(\Om)$ and $(u_{i}, \strain_{i}, \p_{i})_{i=0}^{k}$ }
\\
& \text{ is a time-discrete quasistatic evolution given $z$} \} \,. \nonumber
\end{align}

\begin{proposition}[Existence, time-discrete]\label{p.2}
 The time-discrete \emph{TO} problem~\eqref{min3} admits a solution. In particular, every solution~$z$ satisfies $0\leq z \leq 1$ almost everywhere in~$\Om$.
 \end{proposition}

\begin{proof}
The proof is an application of the Direct Method. Let $z_{j} \in
H^{1}(\Om)$ be a minimizing sequence for~\eqref{min3} with
corresponding time-discrete quasistatic evolution $(u^{j}_{i},
\strain^{j}_{i}, \p^{j}_{i})_{i=0}^{k} \in \big( H^{1}(\Om; \R^{n})
\times L^{2}(\Om; \M^{n}_{S}) \times L^{2}(\Om;
\M^{n}_{D})\big)^{k+1}$. As material parameters are constant
on $\{z\leq 0\}$ and~$\{z\geq 1\}$ we may assume, without loss of generality, that $z_j \in [0,1]$ almost
everywhere.  By~\eqref{e.4} and by the regularity of~$f$,~$g$, and~$w$ we have that $(u^{j}_{i}, \strain^{j}_{i}, \p^{j}_{i})_{i=0}^{k}$ is uniformly bounded in $\big( H^{1}(\Om; \R^{n}) \times L^{2}(\Om; \M^{n}_{S}) \times L^{2}(\Om; \M^{n}_{D}) \big)^{k+1}$, so that we may assume, up to a subsequence, that $(u^{j}_{i}, \strain^{j}_{i}, \p^{j}_{i})_{i=0}^{k} \rightharpoonup (u_{i}, \strain_{i}, \p_{i})_{i=0}^{k}$ weakly in $\big( H^{1}(\Om; \R^{n}) \times L^{2}(\Om; \M^{n}_{S}) \times L^{2}(\Om; \M^{n}_{D})\big)^{k+1}$. Furthermore,~$z_{j}$ is bounded in~$H^{1}(\Om)$ and therefore there exists~$z \in H^{1}(\Om; [0,1])$, such that $z_{j} \rightharpoonup z$ weakly in~$H^{1}(\Om)$. It is easy to see that $\J_{k, \delta} (z, (u_{i})_{i=0}^{k} ) \leq \liminf_{j \to \infty} \J_{k, \delta} (z_{j}, (u^{j}_{i})_{i=0}^{k})$. Thus, it remains to show that $(u_{i}, \strain_{i}, \p_{i})$ still satisfies~\eqref{e.3} for every $i=1, \ldots, k$. This can be done recursively on~$i$. For $i = 1$, we have by construction that $(u_{0}^{j}, \strain^{j}_{0}, \p^{j}_{0}) = (0,0,0)$ for every $j$, so that~$(u_{1}^{j}, \strain^{j}_{1}, \p^{j}_{1})$ solves 
\begin{align*}
\min\, \{ \E_{k}(t^{k}_{1}, z_{j}, u, \strain, \p) + \D(z_{j}, \p) : \, (u, \strain, \p) \in \A(w^{k}_{1}) \}\,.
\end{align*}
The weak convergence of~$(u_{1}^{j}, \strain^{j}_{1}, \p^{j}_{1})$ to~$(u_{1}, \strain_{1}, \p_{1})$ and the lower semicontinuity of~$\E_{k}$ and~$\D$ imply that~$(u_{1}, \strain_{1}, \p_{1})$ satisfies~\eqref{e.3} and that
\begin{displaymath}
\E_{k}(t^{k}_{1}, z, u_{1}, \strain_{1}, \p_{1}) = \lim_{j \to \infty}\, \E_{k}(t^{k}_{1}, z, u^{j}_{1}, \strain^{j}_{1}, \p^{j}_{1})  \,.
\end{displaymath}
The previous equality yields the strong convergence of $(u^{j}_{1}, \strain^{j}_{1}, \p^{j}_{1})$ to $(u_{1}, \strain_{1}, \p_{1})$ in $H^{1}(\Om; \R^{n}) \times L^{2}(\Om; \M^{n}_{S}) \times L^{2}(\Om; \M^{n}_{D})$.

For $i\geq2$, let us assume that~\eqref{e.3} is satisfied at time~$t^{k}_{i-1}$ and that $( u^{j}_{i-1}, \strain^{j}_{i-1}, \p^{j}_{i-1}) \to (u_{i-1}, \strain_{i-1}, \p_{i-1})$ in~$H^{1}(\Om; \R^{n}) \times L^{2}(\Om; \M^{n}_{S}) \times L^{2}(\Om; \M^{n}_{D})$. Arguing as above we deduce that $(u_{i}, \strain_{i}, \p_{i}) \in \A(w^{k}_{i})$ satisfies~\eqref{e.3}. Moreover, the minimality implies that 
\begin{displaymath}
\E_{k} (t^{k}_{i}, z, u_{i}, \strain_{i}, \p_{i}) = \lim_{j \to \infty} \, \E_{k} (t^{k}_{i}, z_{j}, u^{j}_{i}, \strain^{j}_{i}, \p^{j}_{i}) \,,
\end{displaymath}
which yields $( u^{j}_{i}, \strain^{j}_{i}, \p^{j}_{i}) \to (u_{i}, \strain_{i}, \p_{i})$ in $H^{1}(\Om; \R^{n}) \times L^{2}(\Om; \M^{n}_{S}) \times L^{2}(\Om; \M^{n}_{D})$.

The second part of the statement is clear in view of our hypotheses on~$\C$, $\ha$, $d$, and~$\ell$.
\end{proof}

In the following proposition, we state an auxiliary result regarding
the convergence of a sequence of time-discrete quasistatic evolutions
to a quasistatic evolution. The proof is provided in 
Appendix~\ref{appendixA}. Such a result will be used to show that a
sequence of minimizers of~\eqref{min3} converges to a minimizers of
the time-continuous problem~\eqref{min2} as the time-step~$\tau_{k}$
tends to~$0$, as well as to obtain suitable first-order optimality condition for~\eqref{min2}, starting from those of~\eqref{min3} (see Corollary~\ref{c.1} and Theorems~\ref{t.discrete-optimality} and~\ref{t.continuous-optimality}, respectively).

\begin{proposition}[Convergence of time-discrete quasistatic evolutions]\label{p.3}
Let $z_{k}, z \in H^{1}(\Om; [0,1])$ be such that $z_{k}
\rightharpoonup z$ weakly in~$H^{1}(\Om)$. For every~$k$, let
$(u^{k}_{i}, \strain^{k}_{i}, \p^{k}_{i})_{i=0}^{k}$ be the
time-discrete quasistatic evolution  associated with~$z_{k}$ and let
$(u(\cdot), \strain (\cdot), \p(\cdot))$ be the quasistatic evolution associated with~$z$ according
to~Definition~\ref{d.quasistatic}. 
Let us further set
\begin{equation}\label{e.affine-interpolant}
\begin{split}
& u_{k}(t) := u^{k}_{i} + \frac{t - t^{k}_{i}}{\tau_{k}} (u^{k}_{i+1} - u^{k}_{i}) \,,\qquad \strain_{k}(t) := \strain^{k}_{i} + \frac{t - t^{k}_{i}}{\tau_{k}} (\strain^{k}_{i+1} - \strain^{k}_{i})\,,
\\
&
\p_{k}(t) := \p^{k}_{i} + \frac{t - t^{k}_{i}}{\tau_{k}} (\p^{k}_{i+1} - \p^{k}_{i})\,.
\end{split}
\end{equation}
Then, $(u_{k}, \strain_{k}, \p_{k})$ converges to $(u, \strain, \p)$ in $H^{1} ( 0,T; H^{1}(\Om; \R^{n}) \times L^{2}(\Om; \M^{n}_{S}) \times L^{2}(\Om; \M^{n}_{D}))$.
\end{proposition}

\begin{proof}
See Appendix~\ref{appendixA}.
\end{proof}

As a corollary of Proposition~\ref{p.3} we infer the convergence of
 time-discrete TO minimizers of~\eqref{min3} to  time-continuous  TO
minimizers of ~\eqref{min2}.

\begin{corollary}[Convergence of time-discrete TO minimizers]\label{c.1}
Under the assumptions of Proposition~\ref{p.3}, let $z_{k} \in H^{1}(\Om; [0,1])$ be a sequence of minimizers of~\eqref{min3}. Then, there exists~$z \in H^{1}(\Om; [0,1])$ solution of~\eqref{min2} such that, up to a subsequence, $z_{k} \rightharpoonup z$ weakly in~$H^{1}(\Om)$.
\end{corollary}

\begin{proof}
By  inequality~\eqref{e.4}, by the assumptions~\eqref{e.C}--\eqref{e.H}, and by the regularity of~$f$, $g$, and~$w$, the time-discrete evolutions are bounded
uniformly with respect to $k \in \mathbb{N}$. Hence, we
deduce from minimality~\eqref{e.mma} of~$z_{k}$ that~$z_{k}$ is bounded in~$H^{1}(\Om)$. Up to a subsequence, $z_{k} \rightharpoonup z$ weakly in~$H^{1}(\Om)$ and $z \in H^{1}(\Om; [0,1])$. In view of Proposition~\ref{p.3}, the time-discrete quasistatic evolution associated with~$z_{k}$ converges to~$(u, \strain, \p)$ in $H^{1}(0,T; H^{1}(\Om; \R^{n}) \times L^{2}(\Om; \M^{n}_{S}) \times L^{2}(\Om; \M^{n}_{D}))$, where $(u, \strain, \p)$ is the quasistatic evolution associated with~$z$. 

To show the minimality of~$z$, let us fix a competitor~$\hat{z} \in
H^{1}(\Om;[0,1])$ and consider the quasistatic evolution~$(\hat{u},
\hat{\strain}, \hat{\p})$ associated with~$\hat{z}$. For every~$k$, we
can construct the time-discrete quasistatic evolution associated
with~$\hat{z}$ according to Definition
\ref{d.discrete-evolution}. Let us denote by~$(\hat{u}_{k}, \hat{\strain}_{k}, \hat{\p}_{k})$ the piecewise affine functions
\begin{align*}
& \hat{u}_{k} (t) := \hat{u}^{k}_{i-1} + \frac{(t - t^{k}_{i-1})}{\tau_{k}} (\hat{u}^{k}_{i} - \hat{u}^{k}_{i-1}) \,, \qquad \hat{\strain}_{k}(t) := \hat{\strain}^{k}_{i-1} + \frac{(t - t^{k}_{i-1})}{\tau_{k}} (\hat{\strain}^{k}_{i} - \hat{\strain}^{k}_{i-1}) \,,\\
& \hat{\p}_{k} (t) := \hat{\p}^{k}_{i-1} + \frac{(t - t^{k}_{i-1})}{\tau_{k}} (\hat{\p}^{k}_{i} - \hat{\p}^{k}_{i-1}) \,,
\end{align*}
for $t \in [t^{k}_{i-1}, t^{k}_{i})$. In view of
Proposition~\ref{p.3}, we have that  $(\hat{u}_{k}, \hat{\strain}_{k},
\hat{\p}_{k})$ converges to $(\hat{u}, \hat{\strain}, \hat{\p})$ in
$H^{1}(0,T; H^{1}(\Om; \R^{n}) \times L^{2}(\Om; \M^{n}_{S}) \times
L^{2}(\Om; \M^{n}_{D}))$. By the minimality of~$z_{k}$ we have that
\begin{equation}\label{e.5-11}
\J_{k,\delta}(z_{k}, (u^{k}_{i})_{i=0}^{k} ) \leq \J_{k,\delta}( \hat{z}, (\hat{u}^{k}_{i})_{i=0}^{k})\,.
\end{equation}
Hence, passing to the liminf in~\eqref{e.5-11}  as~$k\to\infty$ we deduce that
\begin{align*}
\J_\delta(z, (u_{i})_{i=0}^{k}) \leq \liminf_{k\to \infty} \, \J_{k,\delta}(z_{k}, (u^{k}_{i})_{i=0}^{k}) \leq \liminf_{k \to \infty} \, \J_{k,\delta} ( \hat{z}, ( \hat{u}^{k}_{i})_{i=0}^{k}) =  \J_{\delta}  (\hat{z}, (\hat{u}_{i})_{i=0}^{k} ) \,.
\end{align*}
We conclude by the arbitrariness of~$\hat{z} \in H^{1}(\Om; [0,1])$.
\end{proof}

For the computation of the first-order optimality conditions for~\eqref{min2}, the time-discrete approximation introduced in~\eqref{e.target2}--\eqref{min3} is still insufficient, as the dissipation term~$\D$ is not differentiable. As in~\cite{Alm-Ste_20} (see also~\cite{delosReyes, Herzog2, Wachsmut2}), we define the regularized dissipation
\begin{align*}
&\D_{\gamma} (z, \qq) := \int_{\Om} d(z) h_{\gamma} ( \qq) \, \di x \,, \qquad \text{for every $\qq \in L^{2}(\Om; \M^{n}_{D})$}\,,\\
& h_{\gamma}(\q) := \sqrt{ | \q |^{2} + \frac{1}{\gamma^{2}} } - \frac{1}{\gamma} \qquad \text{for every $\q \in \M^{n}_{D}$}\,.
\end{align*}
In particular, $h_{\gamma} \in C^{\infty}(\M^{n}_{D})$ is convex and satisfies
\begin{align}
& |h_{\gamma}( \q_{1}) - h_{\gamma} (\q_{2}) | \leq | \q_{1} - \q_{2}| \,, \label{e.h1}\\
& | \nabla_{\q} h_{\gamma} (\q_{1}) - \nabla_{\q} h_{\gamma}(\q_{2}) | \leq 2\gamma | \q_{1} - \q_{2}| \,.  \label{e.h2}
\end{align}

Accordingly, we formulate the concept of approximate time-discrete quasistatic evolution as follows.

\begin{definition}[Approximate time-discrete quasistatic
  evolution given $z$]\label{d.approximate-quasistatic}
Let $z \in L^\infty(\Om)$ be given  and let
$f^{k}_{i}, g^{k}_{i}, w^{k}_{i}$ be defined as in~\eqref{e.fgw}. We say that $(u_{i}, \strain_{i}, \p_{i})_{i=0}^{k} \in \big( H^{1}(\Om; \R^{n}) \times L^{2}(\Om; \M^{n}_{S}) \times L^{2}(\Om; \M^{n}_{D})\big)^{k+1}$ is an \emph{approximate time-discrete quasistatic evolution} if $(u_{0}, \strain_{0}, \p_{0}) = (0,0,0)$ and the following holds: for every $i = 1,\ldots, k$, $(u_{i}, \strain_{i}, \p_{i}) \in \A(w^{k}_{i})$ and
\begin{equation}\label{e.5}
\E_{k} ( t^{k}_{i}, z, u_{i}, \strain_{i}, \p_{i}) + \D_{\gamma}(z, \p_{i} - \p_{i-1}) \leq \E_{k} ( t^{k}_{i}, z, u, \strain, \p) + \D_{\gamma}(z, \p - \p_{i-1})
\end{equation}
for every $(u, \strain, \p) \in \A(w^{k}_{i})$.
\end{definition}

As for a time-discrete quasistatic evolutions, for every $z \in L^{\infty}(\Om)$ and every $k \in \mathbb{N}$ an approximate time-discrete evolution $(u_{i}, \strain_{i}, \p_{i})_{i=0}^{k}$ is uniquely determined by iteratively solving the minimum problems
\begin{equation}\label{er.3}
\min \, \{ \E_{k}(t^{k}_{i}, z, u, \strain, \p) + \D_{\gamma} (z, \p - \p_{i-1}): \, (u, \strain , \p) \in \A(w^{k}_{i}) \}
\end{equation}
for $i \geq 1$, where we have set $(u_{0} , \strain_{0}, \p_{0}) = (0,0,0)$.
The approximate time-discrete TO problem reads as
\begin{align}\label{min4}
\min\, \{\J_{k,\delta}(z, (u_{i})_{i=0}^{k}): \ &\text{ $z \in H^{1}(\Om)$ }
\\
&
 \text{and $(u_{i}, \strain_{i}, \p_{i})_{i=0}^{k} \in \big( H^{1}(\Om, \R^{n}) \times L^{2}(\Om; \M^{n}_{S}) \times L^{2}(\Om; \M^{n}_{D}) \big) ^{k+1}$ } \nonumber
\\
&
 \text{is an approximate time-discrete quasistatic evolution given~$z$} \} \,. \nonumber
\end{align}

\begin{proposition}[Existence, approximate time-discrete]\label{p.4}
 The approximate
time-discrete \emph{TO} \linebreak problem \eqref{min4} admits a solution. In particular, every solution~$z$ satisfies $0 \leq z \leq 1$ almost everywhere in~$\Om$. 
\end{proposition}

\begin{proof}
Repeat the steps of the proof of Proposition~\ref{p.2} taking into
account minimality~\eqref{er.3}.
\end{proof}

 We aim now at showing the convergence of solutions to the approximate
time-discrete TO  problem~\ref{min4} to solutions of the time-discrete
TO problem~\eqref{min3}. To this end, we first have to discuss the convergence of approximate time-discrete quasistatic evolutions to a time-discrete quasistatic evolution as the regularization parameter~$\gamma$ tends to~$+\infty$. This is the subject of the following proposition.

\begin{proposition}[Convergence of approximate time-discrete
  quasistatic evolutions]\label{p.5}
Let $k \in \mathbb{N}$ be fixed and let~$z_{\gamma}, z \in H^{1}(\Om; [0,1])$ be such that $z_{\gamma} \rightharpoonup z$ weakly in~$H^{1}(\Om)$ as $\gamma \to +\infty$. Let us denote by~$(u^{\gamma}_{i}, \strain^{\gamma}_{i}, \p^{\gamma}_{i})_{i=0}^{k} \in \big(H^{1}(\Om; \R^{n})\times L^{2}(\Om; \M^{n}_{S}) \times L^{2}(\Om; \M^{n}_{D}) \big)^{k+1}$ the approximate time-discrete quasistatic evolution associated with~$z_{\gamma}$ and by $(u_{i}, \strain_{i}, \p_{i})_{i=0}^{k} \in \big(H^{1}(\Om; \R^{n})\times L^{2}(\Om; \M^{n}_{S}) \times L^{2}(\Om; \M^{n}_{D}) \big)^{k+1}$ the time-discrete quasistatic evolution associated with~$z$. Then, $(u^{\gamma}_{i}, \strain^{\gamma}_{i}, \p^{\gamma}_{i})_{i=0}^{k}$ converges to $(u_{i}, \strain_{i}, \p_{i})_{i=0}^{k}$ in $\big( H^{1}(\Om; \R^{n})\times L^{2}(\Om; \M^{n}_{S}) \times L^{2}(\Om; \M^{n}_{D}) \big)^{k+1}$ as~$\gamma\to + \infty$.
\end{proposition}

\begin{proof}
By minimality of~$(u^{\gamma}_{i}, \strain^{\gamma}_{i}, \p^{\gamma}_{i})$ we have that
\begin{align}\label{e.6}
& \vphantom{\int} \E_{k} (  t^{k}_{i}, z_{ \gamma}, u^{\gamma}_{i}, \strain^{\gamma}_{i}, \p^{\gamma}_{i}) + \D_{\gamma}(z_{ \gamma}, \p^{\gamma}_{i} - \p^{\gamma}_{i-1}) 
\\
&
\vphantom{\int}  \leq \E_{k} (  t^{k}_{i},z_{\gamma}, u^{\gamma}_{i-1} + w^{k}_{i} - w^{k}_{i-1}, \strain^{\gamma}_{i-1} + \e w^{k}_{i} - \e w^{k}_{i-1} , \p^{\gamma}_{i-1}) \nonumber
\\
&
=  \E_{k} (  t^{k}_{i-1}, z_{ \gamma}, u^{\gamma}_{i-1}, \strain^{\gamma}_{i-1}, \p^{\gamma}_{i-1}) + \int_{\Om} \C(z_{\gamma}) \strain^{\gamma}_{i-1} {\, \cdot\,} \e(w^{k}_{i} - w^{k}_{i-1}) \, \di x \nonumber
\\
& 
\qquad + \frac12 \int_{\Om} \C (z_{\gamma})  \e(w^{k}_{i} - w^{k}_{i-1}){\, \cdot\,} \e(w^{k}_{i} - w^{k}_{i-1}) \, \di x \nonumber
\\
&
\qquad - \int_{\Om} \ell(z_{\gamma}) (f^{k}_{i} - f^{k}_{i-1}) {\, \cdot\,} u^{\gamma}_{i-1}\, \di x - \int_{\Om} \ell(z_{\gamma})  f^{k}_{i} {\, \cdot\,} (w^{k}_{i} - w^{k}_{i-1}) \, \di x \nonumber
\\
&
\qquad - \int_{\Gamma_{N}} (g^{k}_{i} - g^{k}_{i-1}) {\, \cdot\,} u^{\gamma}_{i-1} \, \di \HH^{n-1} - \int_{\Gamma_{N}} g^{k}_{i} {\, \cdot\,} (w^{k}_{i} - w^{k}_{i-1}) \, \di \HH^{n-1} \nonumber
\end{align}
Adding the term~$\D_{\gamma} (z_{ \gamma}, \p^{\gamma}_{i-1} - \p^{\gamma}_{i-2})$ to both sides of~\eqref{e.6} and repeating the previous argument for every~$i$, we deduce that
\begin{align}\label{e.7.}
 \E_{k} &(  t^{k}_{i}, z_{\gamma},u^{\gamma}_{i}, \strain^{\gamma}_{i}, \p^{\gamma}_{i}) + \sum_{j=1}^{i} \D_{\gamma}(z_{ \gamma}, \p^{\gamma}_{j} - \p^{\gamma}_{j-1}) 
\\
&
\leq \sum_{j=1}^{i}  \int_{\Om} \C(z_{ \gamma}) \strain^{\gamma}_{j-1} {\, \cdot\,} \e(w^{k}_{j} - w^{k}_{j-1}) \, \di x + \frac12 \int_{\Om} \C (z_{ \gamma})  \e(w^{k}_{j} - w^{k}_{j-1}){\, \cdot\,} \e(w^{k}_{j} - w^{k}_{j-1}) \, \di x \nonumber
\\
&
\qquad - \int_{\Om} \ell(z_{ \gamma}) (f^{k}_{j} - f^{k}_{j-1}) {\, \cdot\,} u^{\gamma}_{j-1}\, \di x - \int_{\Om} \ell(z_{\gamma})  f^{k}_{j} {\, \cdot\,} (w^{k}_{j} - w^{k}_{j-1}) \, \di x \nonumber
\\
&
\qquad - \int_{\Gamma_{N}} (g^{k}_{j} - g^{k}_{j-1}) {\, \cdot\,} u^{\gamma}_{j-1} \, \di \HH^{n-1} - \int_{\Gamma_{N}} g^{k}_{j} {\, \cdot\,} (w^{k}_{j} - w^{k}_{j-1}) \, \di \HH^{n-1} \,,\nonumber
\end{align}
which implies that~$(u^{\gamma}_{i}, \strain^{\gamma}_{i}, \p^{\gamma}_{i})$ is bounded in~$H^{1}(\Om; \R^{n}) \times L^{2}(\Om; \M^{n}_{S}) \times L^{2}(\Om; \M^{n}_{D})$ uniformly w.r.t.~$i$ and~$\gamma$. Arguing as in Proposition~\ref{p.2}, we can prove recursively that~$(u^{\gamma}_{i}, \strain^{\gamma}_{i}, \p^{\gamma}_{i})$ converges to~$(u_{i}, \strain_{i}, \p_{i})$ in $H^{1}(\Om; \R^{n})\times L^{2}(\Om; \M^{n}_{S})\times L^{2}(\Om; \M^{n}_{D})$ as $\gamma\to +\infty$, and~$(u_{i}, \strain_{i}, \p_{i})$ satisfies~\eqref{e.3} for $i=1, \ldots, k$. This concludes the proof of the proposition.
\end{proof}

As a corollary of Proposition~\ref{p.5} we obtain the convergence of solutions of the approximate
time-discrete TO  problem~\eqref{min4} to solutions of the time-discrete
TO problem~\eqref{min3}. 

\begin{corollary}[Convergence of approximate time-discrete TO minimizers]\label{c.2}
Let~$k \in \mathbb{N}$ and let $z_{\gamma} \in H^{1}(\Om; [0,1])$ be a sequence of solutions of~\eqref{min4}. Then, there exists~$z \in H^{1}(\Om; [0,1])$ solution of~\eqref{min3} such that, up to a subsequence, $z_{\gamma} \rightharpoonup z$ weakly in~$H^{1}(\Om)$.
\end{corollary}

\begin{proof}
Repeating the argument of~\eqref{e.7.}, we infer that the approximate
time-discrete quasistatic evolution $(u^{\gamma}_{i},
\strain^{\gamma}_{i}, \p^{\gamma}_{i})_{i=0}^{k}$ corresponding
to~$z_{\gamma}$ is bounded in $ \big( H^{1}(\Om;\R^{n}) \times
L^{2}(\Om; \M^{n}_{S}) \times L^{2}(\Om; \M^{n}_{D})\big) ^{k+1}$. By
minimality, also~$z_{\gamma}$ is bounded in~$H^{1}(\Om)$ and weakly
converges to some~$z$ in~$H^{1}(\Om)$, with~$0 \leq z \leq 1$
almost everywhere. By Proposition~\ref{p.5}, we have that $(u^{\gamma}_{i}, \strain^{\gamma}_{i}, \p^{\gamma}_{i})_{i=0}^{k} \to (u_{i}, \strain_{i}, \p_{i})$ in~$H^{1}(\Om;\R^{n}) \times L^{2}(\Om; \M^{n}_{S}) \times L^{2}(\Om; \M^{n}_{D})$ as~$\gamma\to+\infty$, where~$(u_{i}, \strain_{i}, \p_{i})_{i=0}^{k}$ is the time-discrete quasistatic evolution corresponding to~$z$. From the lower semicontinuity of~$\J_{k, \delta}$ and from Proposition~\ref{p.5} we also deduce that~$z$ solves~\eqref{min3}.
\end{proof}

\section{Sharp-interface limit $\delta \to 0$}
\label{s.gamma}

We prove in this section that the sharp-interface limit $\delta \to 0$
can be rigorously ascertained. This check is performed below in the
time-continuous case of quasistatic evolutions. An analogous argument
could be developed in the case of time-discrete and approximate
time-discrete quasistatic evolutions.

Let us start by recording that the set
of quasistatic evolution is closed with respect to the convergence of
the phase field.

 \begin{proposition}[Convergence of quasistatic evolutions]\label{p.10} Let $z_{m}, z \in L^\infty(\Om; [0,1])$ be such that $z_{m}
\to z$ strongly in~$L^{1}(\Om)$. For every~$m$, let
$(u_m(\cdot), \strain_m(\cdot),\p_m(\cdot))$ be the quasistatic evolution associated with~$z_{m}$ and let
$(u(\cdot), \strain (\cdot), \p(\cdot))$ be the quasistatic evolution
associated with~$z$.
Then,~$(u_{m}, \strain_{m}, \p_{m})$ converges to $(u, \strain, \p)$ in $H^{1} (0,T; H^{1}(\Om; \R^{n}) \times L^{2}(\Om; \M^{n}_{S}) \times L^{2}(\Om; \M^{n}_{D}))$.
\end{proposition}

\begin{proof}
The argument follows closely the general approximation tool from
\cite{mrs}. The coercivity of the energy, which is independent of~$z_m$, and an application of the Helly Selection principle entails
that, up to not relabeled subsequences
$$ (u_{m}(t), \strain_{m}(t), \p_{m}(t)) \rightharpoonup (u (t), \strain (t), \p (t)) \ \
\text{in} \  \  H^{1}(\Om; \R^{n}) \times L^{2}(\Om; \M^{n}_{S})
\times L^{2}(\Om; \M^{n}_{D})$$
for all times $t\in [0,T]$. This suffices to check that 
\begin{align}
 & \E(t,z,u(t),\strain(t),\p(t))\leq  \liminf_{m \to
  \infty } \E ( t, z_m,u_m (t),\strain_m (t),\p_m (t))\,,  \label{eq:mrs1}\\
&\VV([0,t];z,  \p (\cdot)  ) \leq \liminf_{m\to
  \infty} \VV([0,t];z, \p_m(\cdot) )\,, \label{eq:mrs2}
\end{align}
which follow by lower semicontinuity. In particular, we have used the
fact that  $\C(z_m) \to \C(z) $ and ${\mathbb H}(z_m)\to {\mathbb H} (z)$  strongly in
$L^q(\Om;\R^{n\times n \times n \times n})$ for all $q<+\infty$.

On the other hand, given any $(\hat u,\hat \strain,\hat \p)\in
\A(w(t))$, by defining the {\it mutual recovery sequence}
$$ (\hat u_m,\hat \strain_m,\hat \p_m) = (\hat u +u(t)-u_m(t),\hat
\strain +\strain(t)-\strain_m(t),\hat \p +\p(t)-\p_m(t))$$
and exploiting the quadratic character of~$\E$ one can check that 
\begin{align}
  \label{eq:mrs3}
&  \limsup_{m\to \infty}  \Big(\E (t,z_m,\hat u_m,\hat\strain_m,\hat\p_m)  -
  \E (t,z_k,u_m(t),\strain_m(t),\p_m(t)) + \D(z_m,\hat\p_m-\p_m(t))
  \Big) \\
&\quad\nonumber \leq \Big(\E(t,z,\hat u,\hat\strain,\hat\p)  -
  \E(t,z,u(t),\strain(t),\p(t)) + \D(z,\hat\p-\p(t)) \Big)\,.
\end{align}
Properties~\eqref{eq:mrs1}--\eqref{eq:mrs3} allow to apply~\cite[Theorem 3.1]{mrs} ensuring that $(u(\cdot), \strain (\cdot), \p(\cdot))$ is the quasistatic evolution
associated with~$z$, as well as   
$$ \E(t,z_m,u_m(t),\strain_m(t),\p_m(t)) \to
\E(t,z,u(t),\strain(t),\p(t))$$
for all times. The latter entails that the pointwise convergence in
$H^{1}(\Om; \R^{n}) \times L^{2}(\Om; \M^{n}_{S}) \times L^{2}(\Om;
\M^{n}_{D})$ is strong. This can be further improved to a strong
convergence in \linebreak $H^1(0,T; H^{1}(\Om; \R^{n}) \times L^{2}(\Om;
\M^{n}_{S}) \times L^{2}(\Om; \M^{n}_{D}))$ by repeating the argument
of Proposition \ref{p.3}, see Appendix A.
\end{proof}

In order to discuss the sharp-interface limit $\delta \to 0$ we start
by defining the {\it
  sharp-interface} target functional 
\begin{align}\label{e.target}
\J_0(z, u) := &\int_{\Om} \ell(z) \,  f(T) {\, \cdot\,} u(T) \, \di x + \int_{\Gamma_{N}} g(T) {\, \cdot\,} u(T) \, \di \HH^{n-1} 
\\
&\nonumber
-  \int_{0}^{T} \int_{\Om} \ell(z) \, \dot{f}( \tau) {\, \cdot\,} u (\tau) \, \di x \,\di \tau  - \int_{0}^{T} \int_{\Gamma_{N}} \dot{g}(\tau) {\, \cdot\, } u(\tau) \, \di \HH^{n-1} \di \tau 
+ \frac16 {\rm Per} (\{z=1\};\Om) \nonumber
\end{align}
where now the phase $z$ is assumed to belong to $\BV(\Om)$ and take values in
$\{0,1\}$ only. The term ${\rm Per} (\{z=1\};\Om)$ is the perimeter in
$\Om$ of
the set $\{z=1\}$ and effectively penalizes phases with large boundaries. The
constant $1/6$ has no physical relevance and is just chosen to
simplify notations. 
Indeed, setting a different constant here will be
possible. Correspondingly, the {\it sharp-interface} TO
problem reads
\begin{align}\label{min2sharp}
\min\big\{\J_0(z,u(\cdot) ) : \ & z \in BV(\Om;\{0,1\}), \\
& \text{and} \    (u(\cdot), \strain(\cdot),
\p(\cdot)) \text{ is a quasistatic evolution given $z$}\big\} \,. \nonumber
\end{align}
The main result of this section is the following
convergence.
\begin{proposition}[Sharp-interface limit of TO
  minimizers]\label{p.11} 
Let $z_\delta \in H^{1}(\Om; [0,1])$ solve the \emph{TO} problem~\eqref{min2}. Then, up to a
not relabeled subsequence, $z_\delta \to z$ strongly in $L^1(\Om)$,
where $z$ solves the sharp-interface \emph{TO} problem \eqref{min2sharp}.
\end{proposition}

\begin{proof}
  The statement follows by combining the stability of Proposition
  \ref{p.10} with the classical Modica-Mortola construction
  \cite{MR0445362}. 

Let $(u_\delta,\strain_\delta,\p_\delta)$ and $(u^0,\strain^0,\p^0)$
be the quasistatic evolutions associated to $z_\delta$ and
$z=0$, respectively. From minimality we deduce that 
\begin{align*}
&\J_\delta(z_\delta,u_\delta) \leq \J_\delta(0,u^0) = \int_{\Om} \ell(0) \,  f(T) {\, \cdot\,} u^0(T) \, \di x + \int_{\Gamma_{N}} g(T) {\, \cdot\,} u^0(T) \, \di \HH^{n-1} 
\\
&\nonumber
-  \int_{0}^{T} \int_{\Om} \ell(0) \, \dot{f}( \tau) {\, \cdot\,} u^0
  (\tau) \, \di x \,\di \tau  - \int_{0}^{T} \int_{\Gamma_{N}}
  \dot{g}(\tau) {\, \cdot\, } u^0(\tau) \, \di \HH^{n-1} \di \tau<+\infty\,. 
\end{align*}
As $\sup_\delta \J_\delta(z_\delta,u_\delta)<+\infty$ one can extract
a not relabeled subsequence such that $z_\delta \rightharpoonup z$
weakly in~$BV(\Omega)$ and strongly in~$L^1(\Om)$. Owing to
Proposition \ref{p.10} we hence have that
$(u_\delta,\strain_\delta,\p_\delta)$ converges to $(u,\strain,\p)$
strongly in $H^1(0,T; H^{1}(\Om; \R^{n}) \times L^{2}(\Om;
\M^{n}_{S}) \times L^{2}(\Om; \M^{n}_{D}))$ where $(u,\strain,\p)$ is
the quasistatic evolution given $z$. We can hence use the fact
that~\cite{MR0445362}
$$\frac16{\rm Per}(\{z=1\};\Om) \leq \liminf_{\delta \to 0}\int_{\Om}
\frac{\delta}{2} | \nabla{z_\delta}|^{2} +
\frac{z^{2}_\delta(1-z_\delta)^{2}}{2\delta} \, \di x $$
in order to check that 
\begin{equation}\label{gammainf}
\J_0(z,u) \leq \liminf_{\delta \to
  0}\J_\delta(z_\delta,u_\delta) \,.
\end{equation}
In order to prove that~$z$ actually
solves the sharp-interface TO problem~\eqref{min2sharp}, let $\hat z
\in BV(\Om;\{0,1\})$ be given and let $\hat z_\delta \in H^1(\Om)$
be the corresponding Modica-Mortola recovery sequence from~\cite{MR0445362}. This fulfills
\begin{equation}\hat  z_\delta \to \hat z \ \ \text{strongly in} \ \ L^1(\Om)  \ \ \text{and} \
\ \liminf_{\delta \to 0}\int_{\Om}
\frac{\delta}{2} | \nabla{\hat z_\delta}|^{2} +
\frac{\hat z^{2}_\delta(1-\hat z_\delta)^{2}}{2\delta} \, \di x \to \frac16{\rm
  Per}(\{\hat z=1\};\Om)\,.\label{modica}
\end{equation}
Let now
$(\hat u_\delta,\hat \strain_\delta,\hat \p_\delta)$ be the quasistatic evolution given $
\hat z_\delta$ and use again Proposition \ref{p.10} in order to check that
$\hat u_\delta \to \hat u$ in $H^1(0,T;H^1(\Om;\R^n))$ where $(\hat u,\hat
\strain, \hat \p)$ is the quasistatic evolution given~$
\hat z$. We can hence use convergence~\eqref{modica} in order to get
that 
\begin{equation}
\label{gammasup}
\J_\delta(\hat z_\delta,\hat u_\delta) \to \J_0(\hat z, \hat u)\,.
\end{equation}
By combining~\eqref{gammainf} and~\eqref{gammasup} we have that~$z$ solves the sharp-interface TO problem. 
\end{proof}


\section{Differentiability of the state operator for $\gamma<+\infty$}
\label{s.diff}

In preparation for obtaining first-order optimality conditions in
Section \ref{s.optimality}, we develop here the analysis of the
control-to-state operator~$S_{k, \gamma} \colon L^{\infty}(\Om) \to
\big( H^{1}( \Om; \R^{n}) \times L^{2}(\Om; \M^{n}_{S}) \times
L^{2}(\Om, \M^{n}_{D}) \big)^{k+1}$.
For fixed $\gamma \in (0,+\infty)$ and $k \in \mathbb{N}$, the
operator $S_{k, \gamma} $ maps a control~$z \in L^{\infty}(\Om)$
in the unique corresponding approximate time-discrete quasistatic
evolution~$(u^{k}_{i}, \strain^{k}_{i}, \p^{k}_{i})_{i=0}^{k}$. The
differentiability result is stated in Theorem~\ref{t.1}. For this
statement, an auxiliary functional has to be introduced. 
For every $i = 1, \ldots, k$, every $(v, \eeta, \qq) \in H^{1}(\Om; \R^{n}) \times L^{2}(\Om; \M^{n}_{S}) \times L^{2}(\Om; \M^{n}_{D})$, every $ \p  \in L^{2}(\Om; \M^{n}_{D})$, and every $\varphi \in L^{\infty}(\Om)$, we define the functional
\begin{align}\label{e.7}
\F^{\p} _{\gamma} (t_{i}^{k}, z, \varphi, v, \eeta, \qq) :=&\ \frac12 \int_{\Om} \C(z) \eeta{\, \cdot\,} \eeta\, \di x + \frac12 \int_{\Om} \ha(z) \qq {\, \cdot\,} \qq \, \di x 
\\
&
+\int_{\Om} (\C'(z) \varphi) \strain^{k}_{i} {\, \cdot \,} \eeta\, \di x + \int_{\Om} (\ha'(z) \varphi) \p^{k}_{i} {\, \cdot\,} \qq \, \di x \nonumber
\\
&
 + \int_{\Om} \varphi \, d'(z)  \nabla_{\q} h_{\gamma} ( \p^{k}_{i} - \p^{k}_{i-1}) {\, \cdot\,} \qq \, \di x \nonumber
\\
&
+ \frac12 \int_{\Om} d(z) \nabla^{2}_{\q} h_{\gamma} (\p^{k}_{i} - \p^{k}_{i-1})(\qq - \p){\, \cdot\,} (\qq - \p) \, \di x \nonumber
\\
&
- \int_{\Om} \varphi \, \ell'(z) f^{k}_{i} {\, \cdot\,} v \, \di x \,, \nonumber
\end{align} 
where we recall that $(u_{0}, \strain_{0}, \p_{0}) = ( 0,0,0)$.

\begin{theorem}[Differentiability of the control-to-state
  operator $S_{k, \gamma}$]\label{t.1}
Let~$\gamma \in (0,+\infty)$. Then, the control-to-state operator~$S_{k, \gamma}\colon L^{\infty}(\Om) \to
\big( H^{1}( \Om; \R^{n}) \times L^{2}(\Om; \M^{n}_{S}) \times
L^{2}(\Om, \M^{n}_{D}) \big)^{k+1}$ is Frech\'et differentiable. Denoting by $(u^{k}_{i}, \strain^{k}_{i}, \p^{k}_{i})_{i=1}^{k}$ the approximate time-discrete quasistatic evolution associated with~$z \in L^{\infty}(\Om)$, for every~$\varphi \in L^{\infty}(\Om)$ the derivative of~$S_{k, \gamma}$ in~$z$ in the direction~$\varphi$ is given by the vector~$(v^{k, \varphi}_{i}, \eeta^{k, \varphi}_{i}, \qq^{k, \varphi}_{i})_{i=0}^{k} \in \A(0)^{k+1}$ defined recursively as the unique solution of
\begin{equation}\label{e.der}
\min\,\big\{\F^{\qq^{k, \varphi}_{i-1}} _{\gamma} (t_{i}^{k}, z, \varphi, v, \eeta, \qq) : \, (v, \eeta, \qq) \in \A(0) \big\}\,,
\end{equation}
where we have set $\qq^{k, \varphi}_{-1}= 0$.
\end{theorem}

\begin{remark}
Since $f^{k}_{0} = g^{k}_{0} = w^{k}_{0} = 0$ and~$\qq^{k, \varphi}_{-1} = 0$, it is easy to see that $( v^{k, \varphi}_{0}, \eeta^{k, \varphi}_{0}, \qq^{k, \varphi}_{0}  ) = (0,0,0)$ for every $\varphi \in L^{\infty}(\Om)$.
\end{remark}

\begin{remark}
Notice that the incremental minimum problems~\eqref{e.der} define a linear operator from~$L^{\infty}(\Om)$ to~$\big( H^{1}(\Om) \times L^{2}(\Om) \times L^{2}(\Om) \big)^{k+1}$.
\end{remark}

As a corollary of Theorem~\ref{t.1} we get the first-order optimality
conditions for the regularized optimization problem~\eqref{min4}. This
will be the starting point of the analysis of Section~\ref{s.optimality}.

\begin{corollary}[Optimality conditions for the approximate
  time-discrete TO problem]\label{c.approx-optimality}
Under \linebreak the assumptions of Theorem~\ref{t.1}, if~$z \in H^{1}(\Om; [0,1])$ is a solution of~\eqref{min4} with associated approximate time-discrete quasistatic evolution $(u^{k}_{i}, \strain^{k}_{i}, \p^{k}_{i})_{i=0}^{k}$, then there exists $(\overline{u}^{k}_{i}, \overline{\strain}^{k}_{i}, \overline{\p}^{k}_{i})_{i=1}^{k+1} \in \A(0)^{k+1}$ such that $(\overline{u}^{k}_{k+1}, \overline{\strain}^{k}_{k+1}, \overline{\p}^{k}_{k+1}) = (0,0,0)$ and for every $i=k, \ldots, 1$, every $\varphi \in L^{\infty}(\Om) \cap H^{1}(\Om)$, and every $(v, \eeta, \qq) \in \A(0)$
\begin{align} \label{e.8}
&  \int_{\Om} \C(z)  \overline{\strain}^{k}_{i} {\, \cdot\,} \eeta \, \di x  + \int_{\Om} \ha(z) \overline{\p}^{k}_{i} {\, \cdot\,} \qq \, \di x + \int_{\Om} d(z) \nabla^{2}_{\q} h_{\gamma} ( \p^{k}_{i} - \p^{k}_{i-1}) (\overline{\p}^{k}_{i} - \overline{\p}^{k}_{i+1} ) {\, \cdot\,} \qq \, \di x
\\
&
\qquad \qquad  \qquad \quad   - \int_{\Om} \ell(z) f^{k}_{i} {\, \cdot\,} v \, \di x - \int_{\Gamma_{N}} g^{k}_{i} {\, \cdot\,} v \, \di \HH^{n-1} = 0 \,, \nonumber
\\[2mm]
& \sum_{j=1}^{k} \bigg( \int_{\Om} \varphi \, \ell'(z) ( f^{k}_{j} - f^{k}_{j-1}) {\, \cdot\,} (u^{k}_{j} + \overline{u}^{k}_{j} )\, \di x  \label{e.9}
\\
&
\qquad \qquad - \int_{\Om} \big( \C'(z) \varphi \big) ( \strain^{k}_{j} - \strain^{k}_{j-1}) {\, \cdot\,} \overline{\strain}^{k}_{j} \, \di x - \int_{\Om} \big( \ha'(z) \varphi \big) ( \p^{k}_{j} - \p^{k}_{j-1}) {\, \cdot\,} \overline{\p}^{k}_{j} \, \di x \bigg) \nonumber
\\
&
\qquad \qquad - \int_{\Om} \varphi \, d'(z) \nabla_{\q} \big( h_{\gamma} ( \p^{k}_{j} - \p^{k}_{j-1}) - \nabla_{\q} h_{\gamma} ( \p^{k}_{j-1} - \p^{k}_{j-2}) \big) {\, \cdot\,} \overline{\p}^{k}_{j} \, \di x \bigg)  \nonumber
\\
&
\qquad \qquad + \int_{\Om} \delta \nabla{z}{\, \cdot\,} \nabla{\varphi} + \frac{\varphi}{\delta} ( z (1 - z)^{2} - z^{2}(1-z)) \, \di x = 0 \,, \nonumber
\end{align} 
where $\p^{k}_{-1} := 0$.
\end{corollary}

\begin{proof}
Let $z \in H^{1}(\Om; [0,1])$ be a minimizer of~$\J_{k, \delta}$ with corresponding approximate time-discrete quasistatic evolution $(u^{k}_{i}, \strain^{k}_{i}, \p^{k}_{i})_{i=0}^{k}$. Let $\varphi \in L^{\infty}( \Om) \cap H^{1}(\Om)$ and $t \in \R \setminus \{0\}$. Setting $z_{t}:= z + t\varphi$ and denoting by $(u^{k}_{i, t}, \strain^{k}_{i, t}, \p^{k}_{i, t})_{i=0}^{k}$ the approximate time-discrete quasistatic evolution corresponding to~$z_{t}$, we have that $\J_{k, \delta}(z, (u^{k}_{i})_{i=0}^{k}) \leq \J_{k, \delta}(z_{t}, (u^{k}_{i, t}))$. Differentiating~$\J_{k, \delta}(z_{t}, (u^{k}_{i, t}))$ w.r.t.~$t$, we deduce from the minimality of~$z$ and from Theorem~\ref{t.1} that
\begin{align}\label{e.derivata-costo}
\int_{\Om} & \varphi \, \ell'(z) \, f^{k}_{k} {\, \cdot\,} u_{k}\, \di x + \int_{\Om} \ell(z) \, f^{k}_{k} {\, \cdot\,} v^{k, \varphi}_{k} \, \di x + \int_{\Gamma_{N}} g^{k}_{k} {\, \cdot\,} v^{k, \varphi}_{k} \, \di \HH^{n-1}   
\\
&
- \sum_{j=1}^{k} \Bigg (\int_{\Om} \varphi \, \ell'(z) ( f^{k}_{j} - f^{k}_{j-1} ) {\, \cdot\,} u^{k}_{j-1} \, \di x + \int_{\Om} \ell(z) (f^{k}_{j} - f^{k}_{j-1}) {\, \cdot\,} v^{k, \varphi}_{j-1} \, \di x \Bigg) \nonumber
\\
&
- \sum_{j=1}^{k} \int_{\Om} ( g^{k}_{j} - g^{k}_{j-1})  {\, \cdot\,} v^{k, \varphi}_{j-1} \, \di \HH^{n-1} + \int_{\Om} \delta \nabla z{\, \cdot\,} \nabla \varphi + \frac{\varphi}{\delta} ( z (1-z)^{2} - z^{2}(1-z)) \, \di x = 0\nonumber
\end{align}
for every $\varphi \in H^{1}(\Om) \cap L^{\infty}(\Om)$, where $(v^{k, \varphi}_{i}, \eeta^{k, \varphi}_{i}, \qq^{k, \varphi}_{i})_{i=0}^{k} \in \A(0)^{k}$ has been defined in Theorem~\ref{t.1}.  We further set $(v^{k, \varphi}_{-1}, \eeta^{k, \varphi}_{-1}, \qq^{k, \varphi}_{-1}) := (0,0,0)$.

We now define $(\overline{u}^{k}_{i}, \overline{\strain}^{k}_{i}, \overline{\p}^{k}_{i}) \in \A(0)$ as the unique solution of the minimum problem
\begin{align*}
\min\, \bigg \{ \frac12 \int_{\Om}  \C(z) \eeta {\, \cdot\,} \eeta \, \di x &+ \frac12 \int_{\Om} \ha(z) \qq {\,\cdot\,} \qq \, \di x 
\\
&
+ \frac12 \int_{\Om} d(z) \nabla^{2}_{\q} h_{\gamma} (\p^{k}_{i} - \p^{k}_{i-1} )( \qq - \overline{\p}^{k}_{i+1} ) {\, \cdot\,} (\qq - \overline{\p}^{k}_{i+1}) \, \di x 
\\
&
- \int_{\Om} \ell(z) f^{k}_{i}{\, \cdot\,} v \, \di x - \int_{\Gamma_{N}} g^{k}_{i}{\, \cdot\,} v\, \di \HH^{n-1}: \, (v, \eeta, \qq) \in \A(0) \bigg\}
\end{align*}
for $i=k, \ldots, 1$, where we have set $(\overline{u}^{k}_{k+1}, \overline{\strain}^{k}_{k+1}, \overline{\p}^{k}_{k+1})  := (0, 0, 0)$. In particular, $(\overline{u}^{k}_{i}, \overline{\strain}^{k}_{i}, \overline{\p}^{k}_{i})$ satisfies~\eqref{e.8}.

In order to deduce~\eqref{e.9}, we notice that by Theorem~\ref{t.1} and by~\eqref{e.8} and using that  $f^{k}_{0} = g^{k}_{0} = \overline{u}^{k}_{k+1} = 0$ and $\overline \strain^{k}_{k+1} = \overline{\p}^{k}_{k+1} = 0$ we have that
\begin{align*}
& \int_{\Om} \ell(z) \, f^{k}_{k} {\, \cdot\,} v^{k, \varphi}_{k} \, \di x + \int_{\Gamma_{N}} g^{k}_{k} {\, \cdot\,} v^{k, \varphi}_{k} \, \di \HH^{n-1}
\\
&
\qquad  - \sum_{j=1}^{k}  \bigg(  \int_{\Om}  \ell(z) ( f^{k}_{j} - f^{k}_{j-1}) {\, \cdot\,} v^{k, \varphi}_{j-1} \, \di x + \int_{\Om} (g^{k}_{j} - g^{k}_{j-1}) {\, \cdot\,} v^{k, \varphi}_{j-1} \, \di \HH^{n-1}  \bigg)
\\
&
=  \int_{\Om} \C(z) \overline{\strain}^{k}_{k} {\, \cdot\,} \e v^{k, \varphi}_{k} \, \di x - \sum_{j=1}^{k} \int_{\Om} \C(z) ( \overline{\strain}^{k}_{j} - \overline{\strain}^{k}_{j-1}) {\, \cdot\,} \e v^{k, \varphi}_{j-1}\, \di x 
\\
&
= \int_{\Om} \C(z) \e \overline{u}^{k}_{k} {\, \cdot\,} \eeta^{k, \varphi}_{k} \, \di x - \int_{\Om} \C(z) \overline{\p}^{k}_{k}{\, \cdot\,} \eeta^{k, \varphi}_{k} \, \di x + \int_{\Om} \C(z) \overline{\strain}^{k}_{k} {\, \cdot\,} \qq^{k, \varphi}_{k} \, \di x 
\\
&
\qquad\qquad  -  \sum_{j=1}^{k} \bigg( \int_{\Om} \C(z) (\e \overline{u}^{k}_{j} - \e \overline{u}^{k}_{j-1}) { \, \cdot\, } \eeta^{k, \varphi}_{j-1} \, \di x - \int_{\Om} \C(z) ( \overline{\p}^{k}_{j} -  \overline{\p}^{k}_{j-1}) { \, \cdot\, } \eeta^{k, \varphi}_{j-1} \, \di x
\\
&
\qquad\qquad + \int_{\Om} \C(z) ( \overline{\strain}^{k}_{j} -  \overline{\strain}^{k}_{j-1}) { \, \cdot\, } \qq^{k, \varphi}_{j-1} \, \di x \bigg)
\\
&
= \sum_{j=1}^{k} \bigg( \int_{\Om} \C(z) \e \overline{u}^{k}_{j} {\, \cdot\,} ( \eeta^{k, \varphi}_{j} - \eeta^{k, \varphi}_{j-1}) \, \di x - \int_{\Om} \C(z) \overline{\p}^{k}_{j} {\, \cdot\,} ( \eeta^{k, \varphi}_{j} - \eeta^{k, \varphi}_{j-1}) \, \di x 
\\
&
\qquad \qquad + \int_{\Om} \C(z) \overline{\strain}^{k}_{j} {\, \cdot\,} (\qq^{k, \varphi}_{j} - \qq^{k, \varphi}_{j-1}) \, \di x  \bigg)
\\
&
= \sum_{j=1}^{k} \bigg( \int_{\Om} \varphi \, \ell'(z) ( f^{k}_{j} - f^{k}_{j-1}) {\, \cdot\,} \overline{u}^{k}_{j} \, \di x - \int_{\Om} \big( \C'(z) \varphi \big) ( \strain^{k}_{j} - \strain^{k}_{j-1}) {\, \cdot\,} \e \overline{u}^{k}_{j} \, \di x \bigg)
\\
&
\qquad \qquad - \sum_{j=1}^{k} \bigg( \int_{\Om} \ha(z) ( \qq^{k, \varphi}_{j} - \qq^{k, \varphi}_{j-1}) {\, \cdot\,} \overline{\p}^{k}_{j} \,\di x - \int_{\Om} \big( \C'(z) \varphi \big) ( \strain^{k}_{j} - \strain^{k}_{j-1}) {\, \cdot\,} \overline{\p}^{k}_{j} \,\di x 
\\
&
\qquad \qquad + \int_{\Om} \big( \ha'(z) \varphi \big) ( \p^{k}_{j} - \p^{k}_{j-1}) {\, \cdot\,} \overline{\p}^{k}_{j} \, \di x 
\\
&
\qquad \qquad + \int_{\Om} \varphi \, d'(z) \big( \nabla_{\q} h_{\gamma} ( \p^{k}_{j} - \p^{k}_{j-1}) - \nabla_{\q} h_{\gamma} ( \p^{k}_{j-1} - \p^{k}_{j-2}) \big) {\, \cdot\,} \overline{\p}^{k}_{j} \, \di x 
\\
&
\qquad \qquad + \int_{\Om} d(z) \nabla^{2}_{\q} h_{\gamma} ( \p^{k}_{j} - \p^{k}_{j-1}) ( \qq^{k, \varphi}_{j} - \qq^{k, \varphi}_{j-1}) {\, \cdot\,} \overline{\p}^{k}_{j} \, \di x   
\\
&
\qquad \qquad -  \int_{\Om} d(z) \nabla^{2}_{\q} h_{\gamma} ( \p^{k}_{j-1} - \p^{k}_{j-2}) ( \qq^{k, \varphi}_{j-1} - \qq^{k, \varphi}_{j-2}) {\, \cdot\,} \overline{\p}^{k}_{j} \, \di x \bigg)  
\\
&
\qquad \qquad + \sum_{j=1}^{k} \bigg( \int_{\Om} \ha(z) \overline{\p}^{k}_{j} {\, \cdot\,} ( \qq^{k, \varphi}_{j} - \qq^{k, \varphi}_{j-1})  \,\di x
\\
&
\qquad \qquad + \int_{\Om} d(z) \nabla^{2}_{\q} h_{\gamma} ( \p^{k}_{j} - \p^{k}_{j-1}) ( \overline{\p}^{k}_{j} - \overline{\p}^{k}_{j+1}) { \, \cdot\,} ( \qq^{k, \varphi}_{j} - \qq^{k, \varphi}_{j-1}) \, \di x \bigg) 
\\
&
=   \sum_{j=1}^{k} \bigg( \int_{\Om} \varphi \, \ell'(z) ( f^{k}_{j} - f^{k}_{j-1}) {\, \cdot\,} \overline{u}^{k}_{j} \, \di x  - \int_{\Om} \big( \C'(z) \varphi \big) ( \strain^{k}_{j} - \strain^{k}_{j-1}) {\, \cdot\,} \overline{\strain}^{k}_{j} \, \di x
\\
&
\qquad \qquad - \int_{\Om} \big( \ha'(z) \varphi \big) ( \p^{k}_{j} - \p^{k}_{j-1}) {\, \cdot\,} \overline{\p}^{k}_{j} \, \di x 
\\
&
\qquad \qquad - \int_{\Om} \varphi \, d'(z) \big(  \nabla_{\q} h_{\gamma} ( \p^{k}_{j} - \p^{k}_{j-1}) - \nabla_{\q} h_{\gamma} ( \p^{k}_{j-1} - \p^{k}_{j-2}) \big) {\, \cdot\,} \overline{\p}^{k}_{j} \, \di x \bigg) \,,
\end{align*}
where, in the last equality, we have used the following
\begin{align*}
& \sum_{j=1}^{k} \int_{\Om} d(z) \nabla^{2}_{\q} h_{\gamma} ( \p^{k}_{j} - \p^{k}_{j-1} ) ( \overline{\p}^{k}_{j} - \overline{\p}^{k}_{j+1}) {\, \cdot\,} ( \qq^{k, \varphi}_{j} - \qq^{k, \varphi}_{j-1}) \, \di x 
\\
&
\qquad = \sum_{j=1}^{k}   \int_{\Om} d(z) \nabla^{2}_{\q} h_{\gamma} ( \p^{k}_{j} - \p^{k}_{j-1} ) ( \qq^{k, \varphi}_{j} - \qq^{k, \varphi}_{j-1}) {\, \cdot\,} \overline{\p}^{k}_{j}   \, \di x 
\\
&
\qquad \qquad - \sum_{j=1}^{k} \int_{\Om} d(z) \nabla^{2}_{\q} h_{\gamma} ( \p^{k}_{j-1} - \p^{k}_{j-2} ) ( \qq^{k, \varphi}_{j-1} - \qq^{k, \varphi}_{j-2}) {\, \cdot\,} \overline{\p}^{k}_{j}   \, \di x \,.
\end{align*}

All in all, we have proved that
\begin{align*}
& \int_{\Om} \ell(z) \, f^{k}_{k} {\, \cdot\,} v^{k, \varphi}_{k} \, \di x + \int_{\Gamma_{N}} g^{k}_{k} {\, \cdot\,} v^{k, \varphi}_{k} \, \di \HH^{n-1}
\\
&
\qquad  - \sum_{j=1}^{k}  \bigg(  \int_{\Om}  \ell(z) ( f^{k}_{j} - f^{k}_{j-1}) {\, \cdot\,} v^{k, \varphi}_{j-1} \, \di x - \int_{\Om} (g^{k}_{j} - g^{k}_{j-1}) {\, \cdot\,} v^{k, \varphi}_{j-1} \, \di \HH^{n-1}  \bigg)
\\
&
=   \sum_{j=1}^{k} \bigg( \int_{\Om} \varphi \, \ell'(z) ( f^{k}_{j} - f^{k}_{j-1}) {\, \cdot\,} \overline{u}^{k}_{j} \, \di x  - \int_{\Om} \big( \C'(z) \varphi \big) ( \strain^{k}_{j} - \strain^{k}_{j-1}) {\, \cdot\,} \overline{\strain}^{k}_{j} \, \di x
\\
&
\qquad \qquad - \int_{\Om} \big( \ha'(z) \varphi \big) ( \p^{k}_{j} - \p^{k}_{j-1}) {\, \cdot\,} \overline{\p}^{k}_{j} \, \di x \bigg) 
\\
&
\qquad \qquad - \int_{\Om} \varphi \, d'(z) \nabla_{\q} \big( h_{\gamma} ( \p^{k}_{j} - \p^{k}_{j-1}) - \nabla_{\q} h_{\gamma} ( \p^{k}_{j-1} - \p^{k}_{j-2}) \big) {\, \cdot\,} \overline{\p}^{k}_{j} \, \di x \bigg) \,,
\end{align*}
which, together with~\eqref{e.derivata-costo}, implies~\eqref{e.9}.
\end{proof}

The rest of the section is devoted to the proof of Theorem~\ref{t.1}. The next two lemmas are a reformulation of~\cite[Lemmas~3.5 and~3.6]{Alm-Ste_20}, which is needed in order to take care of the term~$\p^{k}_{i-1}$ appearing in the minimization problem~\eqref{e.5} at time~$t^{k}_{i}$ and which is also varying with the phase field~$z$. We recall that this was not the case in~\cite{Alm-Ste_20}, as the problem considered there is static.

\begin{lemma}\label{l.1}
For $z \in [0,1]$, $\gamma \in (0, +\infty)$, and $\pp \in \M^{n}_{D}$, let $F_{z, \gamma, \pp}\colon \M^{n}_{D} \to \M^{n}_{D}$ be the map defined by
\begin{equation}\label{e.Fgamma}
F_{z, \gamma, \pp}(\q) := \C(z) \q + \ha(z) \q + d(z) \nabla_{\q} h_{\gamma}(\q - \pp) \qquad \text{for every $\q \in \M^{n}_{D}$}\,.
\end{equation}
Then, there exist three constants~$C_{1}, C_{2}, C_{3} > 0$ independent of~$\gamma$ and of~$z$ and a constant~$C_{\gamma}>0$ (dependent on~$\gamma$ but not on~$z$) such that for every $z, z_{1}, z_{2} \in [0,1]$ and every $\pp, \pp_{1}, \pp_{2}, \q_{1}, \q_{2} \in \M^{n}_{D}$ the following holds:
\begin{align}
& | F_{z, \gamma, \pp} (\q_{1}) - F_{z, \gamma, \pp}(\q_{2}) | \leq C_{\gamma} | \q_{1} - \q_{2}| \,;\label{e.12}\\[1mm]
& \big( F_{z, \gamma, \pp}(\q_{1}) - F_{z, \gamma, \pp}(\q_{2}) \big){\, \cdot\,} \big( \q_{1} - \q_{2}\big) \geq C_{1}| \q_{1} - \q_{2}| ^{2}\,; \label{e.13} \\[1mm]
& \big( F_{z_{1}, \gamma, \pp_{1}} (\q_{1}) - F_{z_{2}, \gamma, \pp_{2}} (\q_{2}) \big){\, \cdot\,} \big(\q_{1} - \q_{2} \big) \geq C_{1} | \q_{1} - \q_{2}|^{2} \label{e.14} \\
& \qquad\qquad\qquad\qquad - C_{2} ( | \q_{2}| + 1 ) |z_{1} - z_{2}| | \q_{1} - \q_{2}| -  C_{3} \gamma |\q_{1} - \q_{2} | | \pp_{1} - \pp_{2}| \nonumber\,.
\end{align}
In particular, $F_{z, \gamma, \pp}$ is invertible and~$F_{z, \gamma, \pp}^{-1}\colon \M^{n}_{D} \to \M^{n}_{D}$ satisfies
\begin{equation}\label{e.15}
| F^{-1}_{z, \gamma, \pp} (\q_{1}) - F^{-1}_{z, \gamma, \pp} (\q_{2}) | \leq \widetilde{C} | \q_{1} - \q_{2} | \,.
\end{equation}
for a positive constant~$\widetilde{C}$ independent of~$z$,~$\gamma$, and~$\pp$.
\end{lemma}

\begin{proof}
Inequalities~\eqref{e.12},~\eqref{e.13}, and~\eqref{e.15} can be proved repeating the arguments of~\cite[Lemma~3.5]{Alm-Ste_20}. Let us prove~\eqref{e.14}. Let $z_{1}, z_{2} \in [0,1]$ and $\pp_{1}, \pp_{2}, \q_{1}, \q_{2} \in \M^{n}_{D}$. By a simple algebraic argument and by using~\eqref{e.13} we get
\begin{align}\label{e.16}
\big(  F_{z_{1},  \gamma, \pp_{1}} & (\q_{1}) - F_{z_{2}, \gamma, \pp_{2}} (\q_{2}) \big){\, \cdot\,} \big(\q_{1} - \q_{2} \big) 
\\
&
= \big( F_{z_{1}, \gamma, \pp_{1}} (\q_{1}) - F_{z_{1}, \gamma, \pp_{1}} (\q_{2}) \big){\, \cdot\,} \big(\q_{1} - \q_{2} \big) \nonumber
\\
&
\qquad + \big( F_{z_{1}, \gamma, \pp_{1}} (\q_{2}) - F_{z_{2}, \gamma, \pp_{2}} (\q_{2}) \big){\, \cdot\,} \big(\q_{1} - \q_{2} \big) \nonumber
\\
&
\geq C_{1} | \q_{1} - \q_{2} |^{2} + \big( F_{z_{1}, \gamma, \pp_{1}} (\q_{2}) - F_{z_{2}, \gamma, \pp_{2}} (\q_{2}) \big){\, \cdot\,} \big(\q_{1} - \q_{2} \big) \nonumber\,.
\end{align}
We now estimate the last term on the right-hand side of~\eqref{e.16} rewritten as
\begin{align}\label{e.17}
\big( F_{z_{1}, \gamma, \pp_{1}} & (\q_{2}) - F_{z_{2}, \gamma, \pp_{2}} (\q_{2}) \big){\, \cdot\,} \big(\q_{1} - \q_{2} \big)
\\
&
 = \big( \C(z_{1}) - \C(z_{2}) \big) \q_{2}{\, \cdot\,} \big( \q_{1} - \q_{2} \big)
 + \big( \ha(z_{1}) - \ha(z_{2}) \big) \q_{2}{\, \cdot\,} \big( \q_{1} - \q_{2} \big) \nonumber 
 \\
 &
 \qquad + \big(d(z_{1}) \nabla_{\q} h_{\gamma} ( \q_{2} - \pp_{1}) - d(z_{2}) \nabla_{\q} h_{\gamma} (\q_{2} - \pp_{2}) \big) {\, \cdot\,} \big( \q_{1} - \q_{2} \big) \nonumber\,.
\end{align}
By the Lipschitz continuity of~$\C$, $\ha$, and~$d$, and
by~\eqref{e.h1}--\eqref{e.h2}, we hence have
\begin{align}\label{e.18}
\big( F_{z_{1}, \gamma, \pp_{1}} & (\q_{2}) - F_{z_{2}, \gamma, \pp_{2}} (\q_{2}) \big){\, \cdot\,} \big(\q_{1} - \q_{2} \big)
\\
&
\geq - C_{2} \big( |\q_{2}| + 1 \big) |z_{1} - z_{2} | |\q_{1} - \q_{2}| - C_{3} \gamma | \pp_{1} - \pp_{2} | | \q_{1} - \q_{2} | \,, \nonumber
\end{align}
for some positive constant~$C_{2}, C_{3}$ independent of~$z$,
$\gamma$, $\pp_{1}$, and~$\pp_{2}$. Combining~\eqref{e.16}--\eqref{e.18} we
deduce~\eqref{e.14}. Relation \eqref{e.13} entails that~$F_{z,\gamma,\pp}$ is invertible and~\eqref{e.15} follows
from~\eqref{e.13} with $\tilde C = C_1^{-1}$.
\end{proof}

\begin{lemma}\label{l.2}
For every $\gamma \in (0,+\infty)$, every $ z  \in \R$, and every $\pp \in \M^{n}_{D}$, let the map $b_{ z , \gamma, \pp} \colon \M^{n}_{S} \to \M^{n}_{S}$ be defined as
\begin{equation}\label{e.bgamma}
b_{ z , \gamma, \pp}( \boldsymbol{\e}) \coloneq \C ( z ) \big ( \boldsymbol{\e}- F^{-1}_{ z , \gamma, \pp} ( \Pi_{\M^{n}_{D}} ( \C(  z ) \boldsymbol{\e} ) ) \big) \qquad \text{for every $\boldsymbol{\e} \in \M^{n}_{S}$},
\end{equation}
where $\Pi_{\M^{n}_{D}} \colon \M^{n} \to \M^{n}_{D}$ denotes the projection operator on~$\M^{n}_{D}$. Then, there exist two positive constants~$c_{1}, c_{2}$ such that for every~$\gamma \in (0, +\infty)$, every $ z  \in \R$, every $\pp \in \M^{n}_{D}$, and every $\boldsymbol{\e}_{1}, \boldsymbol{\e}_{2} \in \M^{n}_{S}$
\begin{eqnarray}
&& \displaystyle | b_{ z , \gamma, \pp} ( \boldsymbol{\e}_{1}) - b_{ z , \gamma, \pp} (\boldsymbol{\e}_{2} ) | \leq c_{1} |\boldsymbol{\e} _{1} - \boldsymbol{\e}_{2} | \,, \label{e.19}\\[1mm]
&& \displaystyle ( b_{ z , \gamma, \pp} (\boldsymbol{\e}_{1}) - b_{ z ,\gamma, \pp} (\boldsymbol{\e}_{2}) ) \cdot (\boldsymbol{\e}_{1} - \boldsymbol{\e}_{2} ) \geq c_{2} | \boldsymbol{\e}_{1} - \boldsymbol{\e}_{2} |^{2} \,. \label{e.20}
\end{eqnarray}
\end{lemma}

\begin{proof}
The lemma can be proved as \cite[Lemma~3.6]{Alm-Ste_20}
by making use of the already established Lemma~\ref{l.1}.
\end{proof}

We are now in a position to prove an $L^{p}$-regularity estimate and a Lipschitz dependence on the phase-field variable for an approximate time-discrete quasistatic evolution. Before stating these results, we introduce the notation
\begin{align}\label{e.some-norm}
\| (u, \strain , \p) \|_{H^{1} \times L^{2} \times L^{2}} := \| u \|_{H^{1}} + \| \strain\|_{2} + \| \p \|_{2}
\end{align}
for $(u, \strain, \p) \in H^{1}(\Om; \R^{n}) \times L^{2}(\Om;
\M^{n}_{S}) \times L^{2}(\Om; \M^{n}_{D})$. The symbol $\| ( u,
\strain, \p)\|_{W^{1,r} \times L^{r} \times L^{r}}$ is used for $(u, \strain, \p) \in W^{1, r}(\Om; \R^{n}) \times L^{r}(\Om; \M^{n}_{S}) \times L^{r}(\Om; \M^{n}_{D})$. Finally, for $(u_{i}, \strain_{i}, \p_{i})_{i=0}^{k} \in \big( H^{1}(\Om; \R^{n}) \times L^{2}(\Om; \M^{n}_{S}) \times L^{2}(\Om; \M^{n}_{D})\big)^{k+1}$ the norm~$\| (u_{i}, \strain_{i}, \p_{i})_{i=0}^{k} \|_{( H^{1} \times L^{2} \times L^{2})^{k+1}}$ is defined by naturally extending~\eqref{e.some-norm}. The same is done in $\big(W^{1, r}(\Om; \R^{n}) \times L^{r}(\Om; \M^{n}_{S}) \times L^{r}(\Om; \M^{n}_{D}) \big)^{k+1}$.

\begin{lemma}\label{l.3}
Let $k \in \mathbb{N}$ and  $\gamma \in (0,+\infty)$. Then, there exists $\tilde{p} \in (2, p)$ such that the control-to-state operator~$S_{k, \gamma} \colon L^{\infty}(\Om) \to (H^{1}(\Om; \R^{n}) \times L^{2}(\Om; \M^{n}_{S}) \times L^{2}(\Om; \M^{n}_{D}))^{k+1}$ takes values in $(W^{1, \tilde{p}}(\Om; \R^{n}) \times L^{\tilde{p}}(\Om; \M^{n}_{S}) \times L^{\tilde{p}}(\Om; \M^{n}_{D}))^{k+1}$ and satisfies
\begin{align}\label{e.21}
& \| S_{k, \gamma} (z) \|_{(W^{1, \tilde{p}} \times L^{\tilde{p}} \times L^{\tilde{p}})^{k+1}}
\\
&
\qquad  \leq C \big( 1 + \| f\|_{L^{\infty}(0,T; L^{p}( \Om; \R^{n}))} + \| g\|_{L^{\infty}(0,T; L^{p}( \Gamma_{N}; \R^{n}))} + \| w \|_{L^{\infty}( [0 , T] ; W^{ 1 , p } ( \Om; \R^{n})) }\big) \nonumber 
\end{align}
for some positive constant~$C$ independent of~$i$,~$k$,~$\gamma$, and~$z$.

Furthermore, there exists a positive constant~$C_{\gamma, k}$ depending only on~$\gamma$ and~$k$ such that for every $z_{1}, z_{2} \in L^{\infty}(\Om)$ and every $q \in (2, \tilde{p}]$
\begin{align}\label{e.22}
& \| S_{k, \gamma} (z_{1}) - S_{k, \gamma}(z_{2}) \|_{(W^{1, q} \times L^q \times L^q)^{k+1}}
\\
&
\qquad  \leq C_{\gamma, k}  \big( 1 + \| f\|_{L^{\infty}(0,T; L^{p} (\Om; \R^{n}))} + \| g\|_{L^{\infty}(0,T; L^{p} (\Gamma_{N}; \R^{n}))}    \nonumber
\\
&
\qquad \qquad + \| w \|_{L^{\infty}( [0 , T] ; W^{ 1 , p } ( \Om; \R^{n}) )}\big) \| z_{1} - z_{2} \|_{\infty} \nonumber \,.
\end{align}
\end{lemma}

\begin{proof}
The proof of~\eqref{e.21}--\eqref{e.22} follows from an application of~\cite[Theorem~1.1]{Herzog}. To apply such result, we first have to recast the Euler-Lagrange equations associated to the equilibrium condition~\eqref{e.5} in terms of the sole displacement variable~$u$.

Let us fix $\gamma>0$ and~$z \in L^{\infty}(\Om)$. For simplicity of notation, let $(u_{i}, \strain_{i}, \p_{i})_{i=0}^{k} = S_{k, \gamma}(z)$ and $(u^{j}_{i}, \strain^{j}_{i}, \p^{j}_{i})_{i=0}^{k} = S_{k, \gamma} (z_{j})$, $j =1, 2$. We further recall the definition of~$f^{k}_{i}$, $g^{k}_{i}$, and~$w^{k}_{i}$ given in~\eqref{e.fgw}  and that $(u_{0}, \strain_{0}, \p_{0})= (u^{j}_{0}, \strain^{j}_{0}, \p^{j}_{0}) = (0,0,0)$.

 From the minimization problem~\eqref{e.5} we deduce the following Euler-Lagrange equation: for every~$ (v, \eeta, \qq) \in \A(0)$ and every $i=1, \ldots, k$
\begin{align}\label{e.EL}
&\int_{\Om} \C(z) (\e u_{i} - \p_{i}){\, \cdot\,} \eeta \, \di x  + \int_{\Om} \ha(z) \p_{i} {\, \cdot\,} \qq \, \di x + \int_{\Om} d(z) \nabla_{\q} h_{\gamma} ( \p_{i} - \p_{i-1} ){\, \cdot\,} \qq \, \di x
\\
& \qquad
 - \int_{\Om}  \ell(z)  f_{i}^{k}{\, \cdot\, } v \, \di x - \int_{\Gamma_{N}} g_{i}^{k} { \, \cdot\,} v \,\di  \HH^{n-1} = 0 \,. \nonumber
\end{align}
By testing~\eqref{e.EL} with $(0, \eeta, -\eeta) \in \A(0)$ for $\eeta \in L^{2}(\Om; \M^{n}_{D})$ we get that
\begin{equation}\label{e.23}
\C ( z) \p_{i} + \ha(z) \p_{i} + d(z) \nabla_{\q} h_{\gamma} ( \p_{i} - \p_{i-1} ) = \Pi_{\M^{n}_{D}}(  \C (z) \e u_{i}) \qquad \text{a.e.~in~$\Om$}\,.
\end{equation}
In view of the definition~\eqref{e.Fgamma} of~$F_{ z , \gamma, \pp}$, we  have $F_{z(x), \gamma, \p_{i-1}(x) } (\p_{i} (x) ) = \Pi_{\M^{n}_{D} }\big ( \C(z(x)) \e u_{i} (x) \big)$ and $\p_{i} (x) = F_{z(x), \gamma, \p_{i-1}(x) }^{-1} \big (\Pi_{\M^{n}_{D}} \big ( \C(z(x)) \e u_{i} (x) \big)\big)$ for a.e.~$x \in \Om$.

Recalling definition~\eqref{e.bgamma}, we define for $x \in \Om$,~$\boldsymbol{\e} \in \M^{n}_{S}$, and $i=1, \ldots, k$,
\begin{displaymath}
b_{z, \gamma, \p_{i-1}} (x, \boldsymbol{\e}) \coloneq b_{z(x), \gamma, \p_{i-1}(x)} (\boldsymbol{\e}) =  \C (z(x)) \big (\boldsymbol{\e} - F^{-1}_{z(x), \gamma, \p_{i-1}(x)} (\Pi_{\M^{n}_{D}} (\C(z(x)) \boldsymbol{\e} ) \big) \,.
\end{displaymath}
From now on, when not explicitly needed, we drop the dependence on the spatial variable~$x \in \Om$ in the definition of~$F_{z, \gamma,\p_{i-1}}^{-1}$, since all the arguments discussed below are valid uniformly in~$\Om$. We rewrite the Euler-Lagrange equation~\eqref{e.EL} in terms of the sole displacement~$u_{i}$ and for test functions of the form $(\psi, \e \psi, 0) \in \A(0)$ for $\psi \in H^{1} (\Om; \R^{n})$ with $\psi = 0$ on~$\Gamma_{D}$:
\begin{equation}\label{e.EL2}
\int_{\Om} b_{z, \gamma, \p_{i-1}} (x, \e u_{i}) {\, \cdot\,} \e \psi \, \di x = \int_{\Om}  \ell(z)  f^{k}_{i} {\, \cdot\, } \psi \, \di x + \int_{\Gamma_{N}} g^{k}_{i} {\, \cdot\,} \psi \, \di \HH^{n-1} \,.
\end{equation}

In view of Lemma~\ref{l.2}, the nonlinear operator~$B_{z, \gamma, \p_{i-1}} \colon W^{1, p}( \Om; \R^{n}) \to W^{-1, p}( \Om; \R^{n})$ defined as $B_{z, \gamma,\p_{i-1}} (u) \coloneq b_{z, \gamma, \p_{i-1}} (x, \e u)$ satisfies the hypotheses of~\cite[Theorem~1.1]{Herzog}. Since $\Om \cup \Gamma_{N}$ is Gr\"oger regular, $ p \in (2, +\infty)$, $f^{k}_{i} \in L^{p} (\Om ; \R^{n})$, $g^{k}_{i} \in L^{p}(\Gamma_{N}; \R^{n})$, and $w^{k}_{i} \in W^{1, p}(\Om; \R^{n})$, we infer from~\cite[Theorem~1.1]{Herzog} applied to equation~\eqref{e.EL2} that there exist~$\tilde{p} \in (2, p)$ and a constant~$C>0$  (both independent of~$i$ and~$k$) such that
 \begin{equation}\label{e.24}
 \| u_{i} \|_{W^{1, q}} \leq C ( \| f^{k}_{i} \|_{p} + \| g_{i}^{k} \|_{p} + \| w^{k}_{i} \|_{W^{1, p}})
 \end{equation}
 for every $q \in (2, \tilde{p}]$. In particular, $C$ is independent of~$z \in L^{\infty}(\Om)$, of~$\gamma \in (0,+\infty)$, of~$k\in \mathbb{N}$, and of~$q \in (2, \tilde{p}]$. Inequality~\eqref{e.21}  can be deduced by combining~\eqref{e.15} and~\eqref{e.24}. Indeed, we have that
 \begin{align}\label{e.25}
\| \p_{i} \|_{q} & = \| F^{-1}_{z, \gamma, \p_{i-1}} (\Pi_{\M^{n}_{D}} (\C(z) \e u_{i})) \|_{q} 
\\
&
\leq \| F^{-1}_{z, \gamma, \p_{i-1}} (\Pi_{\M^{n}_{D}} (\C(z) \e u_{i})) - F^{-1}_{z, \gamma, \p_{i-1}} (0) \|_{q} + \| F^{-1}_{z, \gamma, \p_{i-1}} (0) \|_{q}  \nonumber
\\
&
\leq \widetilde{C} \| \Pi_{\M^{n}_{D}} (\C(z) \e u_{i}) \|_{q} + \| F^{-1}_{z, \gamma, \p_{i-1}} (0) \|_{q} \leq \overline{C} \| u_{i} \|_{W^{1, q}} +  \| F^{-1}_{z, \gamma, \p_{i-1}} (0) \|_{q}\,. \nonumber
 \end{align}
 To conclude the estimate, we notice that if $\qq \coloneq F^{-1}_{z, \gamma, \p_{i-1}} (0)$, we have that
 \begin{displaymath}
 \C(z) \qq + \ha(z) \qq = -d(z) \nabla_{\q} h_{\gamma} (\qq - \p_{i-1}) = - d(z) \frac{ \qq - \p_{i-1} }{ \sqrt{ | \qq - \p_{i-1} |^{2} + \frac{1}{\gamma^{2}}}} \,.
 \end{displaymath} 
 Multiplying the previous expression by~$\qq$ and using~\eqref{e.C}--\eqref{e.H} we deduce that 
 \begin{displaymath}
 (\alpha_{\C} + \alpha_{\ha}) | \qq |^{2}  \leq d(z) | \qq |
 \end{displaymath}
a.e.~in~$\Om$.
 Hence,~$  \| F^{-1}_{z, \gamma, \p_{i-1}} (0) \|_{q}$ is bounded uniformly w.r.t.~$i$,~$k$,~$\gamma$, and~$z$. Thus, combining~\eqref{e.24}--\eqref{e.25} we infer~\eqref{e.21} by the triangle inequality.
 
 In order to prove~\eqref{e.22}, we first rewrite the Euler-Lagrange equation~\eqref{e.EL2} satisfied by~$u^{2}_{i}$, $i=1, \ldots, k$. Namely, for every~$\psi \in W^{1, \tilde{p}'}(\Om; \R^{n})$ with~$\psi = 0 $ on~$\Gamma_{D}$ we have, after a simple algebraic manipulation,
 \begin{align}\label{e.27}
 \int_{\Om} & B_{z_{1}, \gamma, \p^{1}_{i-1} } (u^{2}_{i}) {\, \cdot \, } \e \psi \, \di x 
 \\
 &
 =  \int_{\Om} \big( \C (z_{1}) - \C(z_{2}) \big) \big( \e u^{2}_{i} - F^{-1}_{ z_{1} , \gamma, \p^{1}_{i-1}} (\Pi_{\M^{n}_{D}} ( \C ( z_{1} ) \e u^{2}_{i})) \big) \cdot \e \psi \, \di x \nonumber
 \\
 &
\qquad  + \int_{\Om} \C( z_{2}  ) \big(  F_{z_{2}, \gamma, \p^{2}_{i-1}}^{-1} ( \Pi_{\M^{n}_{D}} (\C(z_{2}) \e u^{2}_{i}) ) - F^{-1}_{z_{1}, \gamma, \p^{1}_{i-1}} ( \Pi_{\M^{n}_{D}} (\C(z_{1}) \e u^{2}_{i}) ) \big) \cdot \e \psi \, \di x \nonumber
 \\
 &
\qquad  + \int_{\Om} \ell(z_{2}) f^{k}_{i} {\, \cdot \, } \psi \, \di x  + \int_{\Gamma_{N}} g^{k}_{i} {\, \cdot\,} \psi \, \di \HH^{n-1} \,. \nonumber
 \end{align}

Comparing~\eqref{e.27} with~\eqref{e.EL2} written for~$z_{1}$ and
$(u^{1}_{i}, \strain^{1}_{i}, \p^{1}_{i})$, we deduce that~$u^{1}_{i}$
and~$u^{2}_{i}$ solve the same kind of equation, with a different
right-hand side, which however always belongs to $W^{-1,\tilde{p}}(\Om; \R^{n})$. Thus, applying once more~\cite[Theorem~1.1]{Herzog}, we infer that there exists~$C>0$ independent of~$z_{1}, z_{2}$, of~$\gamma$, of~$i$, and of~$k$, such that for every $q \in (2, \tilde{p}]$
\begin{align}\label{e.28}
\| u^{1}_{i} - u^{2}_{i} \|_{W^{1, q}} & \leq C \Big(  \left \| \big( \C (z_{1}) - \C (z_{2}) \big) \big( \e u^{2}_{i} - F^{-1}_{z_{2}, \gamma, \p^{2}_{i-1}} ( \Pi_{\M^{n}_{D}} ( \C (z_{2}) \e u^{2}_{i} )) \big) \right \|_{W^{- 1, q}}
\\
&
\quad + \left\| \C(z_{1}) \big( F_{z_{2}, \gamma, \p^{2}_{i-1}}^{-1} ( \Pi_{\M^{n}_{D}} (\C(z_{2}) \e u^{2}_{i} ) ) - F^{-1}_{z_{1}, \gamma, \p^{1}_{i-1}} ( \Pi_{\M^{n}_{D}} (\C(z_{1}) \e u^{2}_{i} ) ) \big) \right \|_{W^{-1, q}}  \nonumber
\\
&
\quad + \left \| \big(\ell(z_{1}) - \ell(z_{2}) \big) f^{k}_{i} \right \|_{W^{-1, q}}\Big)
=: C ( I_{1} + I_{2} + I_{3}) \,. \nonumber
\end{align}
By the Lipschitz continuity of~$\C$, by the identification $\p^{2}_{i} = F^{-1}_{z_{2}, \gamma, \p^{2}_{i-1}} ( \Pi_{\M^{n}_{D}} (\C (z_{2} ) \e u^{2}_{i}) )$, by the H\"older inequality, and by~\eqref{e.21} we deduce that
\begin{align}\label{e.30}
I_{1} & \leq C \| z_{1} - z_{2} \|_{\infty} \big( \| u_{i}^{2}\|_{W^{1, \tilde{p}}} + \| \p^{2}_{i} \|_{\tilde{p}} \big)
\\
&
 \leq C \| z_{1} - z_{2} \|_{\infty} \big( 1 + \| f\|_{L^{\infty}(0,T; L^{p} ( \Om; \R^{n}) )} + \| g\|_{L^{\infty}(0,T; L^{p} ( \Gamma_{N}; \R^{n}))}  \nonumber
 \\
 &
 \qquad + \| w \|_{L^{\infty}( [0 , T] ; W^{ 1 , p } (\Om; \R^{n}) )}\big) \nonumber \,.
\end{align}


Rewriting~\eqref{e.14} for $\pp_{j} = \p^{j}_{i-1}$ and  $\q_{j} = F^{-1}_{z_{j} , \gamma, \p^{j}_{i-1}} ( \Pi_{\M^{n}_{D}} (\C (z_{j}) \e u^{2}_{i}) ) $ we get that for a.e.~$x \in \Om$
\begin{align}\label{e:smth}
C_{1} |&  F^{-1}_{z_{1}, \gamma, \p^{1}_{i-1}}  ( \Pi_{\M^{n}_{D}}  (\C (z_{1}) \e u^{2}_{i}) ) - F^{-1}_{z_{2}, \gamma, \p^{2}_{i-1}} (\Pi_{\M^{n}_{D}} (\C (z_{2}) \e u^{2}_{i}) ) | 
\\
&
\qquad \leq  {\rm Lip}(\C) | \e u^{2}_{i} | | z_{1} - z_{2}| + C_{2} \big( | F^{-1}_{z_{2}, \gamma, \p^{2}_{i-1}} (\Pi_{\M^{n}_{D}} (\C (z_{2}) \e u^{2}_{i}) ) | + 1 \big) \, |z_{1} - z_{2} | \nonumber
\\
&
\qquad\qquad + C_{3} \gamma  | \p^{1}_{i-1} - \p^{2}_{i-1} |\,. \nonumber
\end{align}
The  identification $\p^{2}_{i} = F^{-1}_{z_{2}, \gamma, \p^{2}_{i-1}} ( \Pi_{\M^{n}_{D}} (\C (z_{2} ) \e u^{2}_{i}) )$ and inequalities~\eqref{e.21} and~\eqref{e:smth} imply that
\begin{align}\label{e.31}
I_{2} & \leq C ( \| u^{2}_{i}\|_{W^{1, \tilde{p}}} + \| \p^{2}_{i} \|_{\tilde{p}} + 1) \| z_{1} - z_{2} \|_{\infty} + C_{3} \gamma \| \p^{1}_{i-1} - \p^{2}_{i-1} \|_{q}
\\
&
 \leq C ( \| f \|_{L^{\infty}(0,T; L^{p}(\Om; \R^{n}))} + \| g \|_{L^{\infty}(0,T; L^{p}(\Gamma_{N}; \R^{n}))} \nonumber
 \\
 &
 \qquad   + \| w \|_{L^{\infty}(0,T; W^{1, p}(\Om; \R^{n}))} + 1 ) \| z_{1} - z_{2} \|_{\infty}  + C_{3} \gamma \| \p^{1}_{i-1} - \p^{2}_{i-1}\|_{q} \nonumber\,.
\end{align}

Finally, by the Lipschitz continuity of~$\ell$ we conclude that 
\begin{equation}\label{e.32}
I_{3} \leq C \| f \|_{L^{\infty}(0,T; L^{p}(\Om; \R^{n}))} \| z_{1} - z_{2} \|_{\infty} \,.
\end{equation}
Combining inequalities~\eqref{e.28}--\eqref{e.32} we infer that
\begin{align}\label{e.33}
\| u^{1}_{i} - u^{2}_{i} \|_{W^{1,q}} \leq & \ C \Big( \| f \|_{L^{\infty}(0,T; L^{p}(\Om; \R^{n}))} + \| g \|_{L^{\infty}(0,T; L^{p}(\Gamma_{N}; \R^{n}))} 
\\
&
+ \| w \|_{L^{\infty}(0,T; W^{1, p}(\Om; \R^{n}))} + 1 \Big) \| z_{1} - z_{2} \|_{\infty}  
  + C_{3} \gamma \| \p^{1}_{i-1} - \p^{2}_{i-1}\|_{q}  \nonumber\,.
\end{align}
We notice that inequality~\eqref{e.14} tested with
\begin{displaymath}
\q_{j} = F^{-1}_{z_{j}(x), \gamma, \p^{j}_{i-1}(x)} \big( \Pi_{\M^{n}_{D}} ( \C(z_{j}(x)) \e u^{j}_{i}(x)) \big) = \p^{j}_{i}(x) \qquad \text{for a.e.~$x \in \Om$}
\end{displaymath}
and integrated over~$\Om$ implies
\begin{align}\label{e.34}
 \| \p^{1}_{i} - \p^{2}_{i} \|_{q} & \leq C \big( \| u^{1}_{i} - u^{2}_{i} \|_{W^{1, q}} + (\| u^{2}_{i}\|_{W^{1, \tilde{p}}} + \| \p^{2}_{i}\|_{L^{\tilde{p}}} + 1) \| z_{1} - z_{2} \|_{\infty}  \big) 
 \\
 &
 \qquad + C_{3} \gamma \| \p^{1}_{i-1} - \p^{2}_{i-1}\|_{q} \nonumber
 \\
 &
 \leq C \| u^{1}_{i} - u^{2}_{i} \|_{W^{1, q}}  + C_{3} \gamma \| \p^{1}_{i-1} - \p^{2}_{i-1}\|_{q} \nonumber
 \\
 &
 \qquad + C ( \| f \|_{L^{\infty}(0,T; L^{p}(\Om; \R^{n}))} + \| g \|_{L^{\infty}(0,T; L^{p}(\Gamma_{N}; \R^{n}))}  \nonumber
\\
&
\qquad + \| w \|_{L^{\infty}(0,T; W^{1, p}(\Om; \R^{n}))} + 1 )  \| z_{1} - z_{2} \|_{\infty}   \nonumber
\\
&
\leq  C ( \| f \|_{L^{\infty}(0,T; L^{p}(\Om; \R^{n}))} + \| g \|_{L^{\infty}(0,T; L^{p}(\Gamma_{N}; \R^{n}))}  \nonumber
\\
&
\qquad + \| w \|_{L^{\infty}(0,T; W^{1, p}(\Om; \R^{n}))} + 1 )  \| z_{1} - z_{2} \|_{\infty} + C \gamma \| \p^{1}_{i-1} - \p^{2}_{i-1}\|_{q} \,. \nonumber
 \end{align}
 By the triangle inequality, an estimate similar to~\eqref{e.34} holds for $\strain^{1}_{i} - \strain^{2}_{i}$, for every $i=1, \ldots, k$.
%
 Iterating the inequalities~\eqref{e.33}--\eqref{e.34} for $l = 1, \ldots, i$ and taking into account that $( u^{1}_{0}, \strain^{1}_{0}, \p^{1}_{0})  = ( u^{2}_{0}, \strain^{2}_{0}, \p^{2}_{0} ) = ( 0, 0, 0 )$, we obtain~\eqref{e.22}. This concludes the proof of the lemma.
\end{proof}

We are now ready to prove Theorem~\ref{t.1}. The proof follows the lines of the proofs of~\cite[Theorem~3.1]{Alm-Ste_20} and of~\cite[Theorem~3.3]{Blank2}. The main difference is that, as in~\cite{Wachsmut2}, the forward problem is now time dependent and not static.

\begin{proof}[Proof of Theorem~\ref{t.1}]
Let us  fix $k \in \mathbb{N}$,~$\gamma \in(0,+\infty)$, and~$z, \varphi \in L^{\infty}(\Om)$. For $t \in \R$, let~$z_{t} \coloneq z + t\varphi$,~$(u^{k}_{i,t}, \strain^{k}_{i,t}, \p^{k}_{i,t})_{i=1}^{k} \coloneq S_{k, \gamma}( z_{t})$. The solution for~$t=0$ will be simply denoted by~$(u^{k}_{i}, \strain^{k}_{i}, \p^{k}_{i})_{i=1}^{k}$. Moreover, let~$(v^{k, \varphi}_{i}, \eeta^{k, \varphi}_{i}, \qq^{k, \varphi}_{i})$ be the solution of the recursive minimization problem~\eqref{e.der} and set
\begin{displaymath}
\overline{v}^{k}_{i, t} \coloneq u^{k}_{i, t} - u_{i}^{k} - t  v^{k, \varphi}_{i}  \,, \qquad \overline{\eeta}^{k}_{i, t} \coloneq \strain^{k}_{i, t} - \strain^{k}_{i} - t \eeta^{k, \varphi}_{i}  \,, \qquad \overline{\qq}^{k}_{i, t} \coloneq \p^{k}_{i, t} - \p^{k}_{i} - t \qq^{k, \varphi}_{i} \,.
\end{displaymath} 
We want to show that
\begin{equation}\label{e.35}
\| ( \overline{v}^{k}_{i, t}, \overline{\eeta}^{k}_{i, t}, \overline{\qq}^{k}_{i, t}) \|_{H^{1} \times L^{2} \times L^{2}} = o ( t ) \,,
\end{equation}
uniformly w.r.t.~$i = 1, \ldots, k$. In particular,~\eqref{e.35} implies the Frech\'et differentiability of the control-to-state map~$S_{k, \gamma}$.

We prove~\eqref{e.35} by induction on~$i$. For $i=1$,~\eqref{e.35}
follows from~\cite[Theorem~3.1]{Alm-Ste_20}, as the initial value
is~$(u^{k}_{0}, \strain^{k}_{0}, \p^{k}_{0}) = (0,0,0)$ by the assumptions on the data~$f(0) = g(0) = w(0) = 0$. For $i>1$, assume that $\| ( \overline{v}^{k}_{i-1, t}, \overline{\eeta}^{k}_{i-1, t}, \overline{\qq}^{k}_{i-1, t}) \|_{H^{1} \times L^{2} \times L^{2}} = o ( t )$. Writing the Euler-Lagrange equations satisfied by~$(u^{k}_{i, t}, \strain^{k}_{i, t}, \p^{k}_{i, t})$, $(u^{k}_{i}, \strain^{k}_{i}, \p^{k}_{i} )$, and $(v^{k, \varphi}_{i}, \eeta^{k, \varphi}_{i}, \qq^{k, \varphi}_{i})$ and subtracting the second and the third from the first one, we obtain, for every $(v, \eeta, \qq) \in \A(0)$,
\begin{align*}
\int_{\Om}  \C(z_{t})  \strain^{k}_{i, t} {\, \cdot\, } \eeta \, \di x & - \int_{\Om} \C (z) \strain^{k}_{i} {\, \cdot\,} \eeta \, \di x - t \int_{\Om} \C(z) \eeta^{k, \varphi}_{i}  {\, \cdot\,} \eeta \, \di x - t \int_{\Om} (\C'(z) \varphi) \strain^{k}_{i} {\, \cdot \, } \eeta \, \di x 
\\
&
+ \int_{\Om} \ha (z_{t}) \p^{k}_{i, t} {\, \cdot\,} \qq \, \di x
 - \int_{\Om} \ha (z) \p^{k}_{i} {\, \cdot\,} \qq \, \di x 
 - t \int_{\Om} \ha(z)  \qq^{k, \varphi}_{i}  {\, \cdot \,} \qq \, \di x  
 \\
 &
 - t  \int_{\Om} (\ha'(z) \varphi) \p^{k}_{i} {\, \cdot\,} \qq \, \di x + \int_{\Om} d(z_{t} ) \nabla_{\q} h_{\gamma}( \p^{k}_{i, t} - \p^{k}_{i-1, t} ) {\,\cdot\,} \qq \, \di x 
 \\
 &
 - \int_{\Om} d (z ) \nabla_{\q} h_{\gamma} (\p^{k}_{i} - \p^{k}_{i-1} ) {\, \cdot\,} \qq \, \di x - t \int_{\Om} \varphi \, d'(z) \nabla_{\q} h_{\gamma} (\p^{k}_{i} - \p^{k}_{i-1} ) {\, \cdot\,} \qq \, \di x 
 \\
 &
 - t \int_{\Om} d(z)  \nabla^{2}_{\q} h_{\gamma} (\p^{k}_{i} - \p^{k}_{i-1} ) (\qq^{k, \varphi}_{i} - \qq^{k, \varphi}_{i-1}) { \, \cdot \,} \qq \, \di x 
 - \int_{\Om} \ell(z_{t}) f^{k}_{i} {\, \cdot\,} v \, \di x 
 \\
 &
 + \int_{\Om} \ell(z) f^{k}_{i} {\, \cdot\,} v \, \di x + t \int_{\Om} \varphi\, \ell'(z)  f^{k}_{i} {\, \cdot\,} v \, \di x = 0 \,.
\end{align*}
By a simple algebraic manipulation, we rewrite the previous equality as
\begin{align}\label{e.36}
&0  =  \bigg(\int_{\Om}  \C(z_{t}) \strain^{k}_{i, t} {\, \cdot\, } \eeta \, \di x - \int_{\Om} \C (z) \strain^{k}_{i} {\, \cdot\,} \eeta \, \di x - t \int_{\Om} \C(z)  \eeta^{k, \varphi}_{i}  {\, \cdot\,} \eeta \, \di x
\\
&
\qquad  - t \int_{\Om} (\C'(z) \varphi) \strain^{k}_{i} {\, \cdot \, } \eeta \, \di x \bigg) \nonumber
\\
&
\qquad + \bigg( \int_{\Om} \ha (z_{t}) \p^{k}_{i, t} {\, \cdot\,} \qq \, \di x  - \int_{\Om} \ha (z) \p^{k}_{i} {\, \cdot\,} \qq \, \di x -  t  \int_{\Om} \ha(z)  \qq^{k, \varphi}_{i}  {\, \cdot \,} \qq \, \di x \nonumber 
\\
&
\qquad  -  t \int_{\Om} (\ha'(z) \varphi) \p^{k}_{i} {\, \cdot\,} \qq \, \di x  \bigg) \nonumber
\\
&
\qquad + \bigg( \int_{\Om}  \big( d(z_{t}) - d(z) - t \varphi \, d' (z) \big) \nabla_{\q} h_{\gamma} (\p^{k}_{i} - \p^{k}_{i-1} ) {\, \cdot\,} \qq \, \di x  \nonumber 
\\
&
\qquad + \int_{\Om} \big( d(z_{t}) - d( z) \big) \big( \nabla_{\q} h_{\gamma} (\p^{k}_{i, t} - \p^{k}_{i-1, t} ) - \nabla_{\q} h_{\gamma} ( \p^{k}_{i} - \p^{k}_{i-1} ) \big) \cdot \qq \, \di x  \bigg) \nonumber
\\ 
&
\qquad + \bigg( \int_{\Om} d(z) \big(\nabla_{\q} h_{\gamma} (\p^{k}_{i, t} - \p^{k}_{i-1, t} ) - \nabla_{\q} h_{\gamma} (\p^{k}_{i} - \p^{k}_{i-1} ) \nonumber
\\
&
\qquad - t \nabla_{\q}^{2} h_{\gamma} (\p^{k}_{i} - \p^{k}_{i-1}) (\qq^{k, \varphi}_{i} - \qq^{k, \varphi}_{i-1} )\big) {\, \cdot \,} \qq \, \di x \bigg) \nonumber
 \\
&
\qquad  - \bigg(  \int_{\Om} \big( \ell(z_{t}) - \ell(z) - t \ell'(z) \varphi \big)  f^{k}_{i} {\, \cdot\,} v \, \di x \bigg)\nonumber
\\
&
= : I_{t, 1} + I_{t, 2} + I_{t, 3} + I_{t, 4}  + I_{t, 5} \, . \nonumber
\end{align}

Let us now rewrite $I_{t, 1}$, $I_{t, 2}$, and~$I_{t, 4}$ from~\eqref{e.36}. For $I_{t, 1}$ we have that
\begin{displaymath}
I_{t, 1} =  \int_{\Om}  \! \C(z) \overline{\eeta}^{k}_{i, t} {\, \cdot\, } \eeta \, \di x +\! \int_{\Om} \! (\C(z_{t}) - \C(z)) ( \strain^{k}_{i, t} - \strain^{k}_{i} ) {\, \cdot \, } \eeta \, \di x
+ \! \int_{\Om} \! \big(\C(z_{t}) - \C(z) - t (\C'(z) \varphi)\big) \strain^{k}_{i} {\, \cdot\,} \eeta \, \di x .
\end{displaymath}
In a similar way, we have that
\begin{displaymath}
I_{t, 2} =  \int_{\Om}  \! \ha(z) \overline{\qq}^{k}_{i, t} {\, \cdot\, } \qq \, \di x + \! \int_{\Om} \! (\ha(z_{t}) - \ha(z)) ( \p^{k}_{i, t} - \p^{k}_{i} ) {\, \cdot \, } \qq \, \di x
+ \! \int_{\Om} \! \big(\ha(z_{t}) - \ha(z) - t (\ha'(z) \varphi)\big) \p^{k}_{i}  {\, \cdot\,} \qq \, \di x .
\end{displaymath}
As for~$I_{t, 4}$, since~$h_{\gamma} \in C^{\infty}( \M^{n}_{D})$, for every $t>0$ there exists $\boldsymbol{\xi}_{t}$ on the segment~$[\p^{k}_{i} - \p^{k}_{i-1}, \p^{k}_{i, t} - \p^{k}_{i-1, t}]$ such that
\begin{displaymath}
\begin{split}
I_{t, 4} & = \int_{\Om} d(z) \big(\nabla^{2}_{\q} h_{\gamma} (\boldsymbol{\xi}_{t}) ( \p^{k}_{i, t} - \p^{k}_{i-1, t} - \p^{k}_{i} + \p^{k}_{i-1} ) - t \nabla_{\q}^{2} h_{\gamma} (\p^{k}_{i} - \p^{k}_{i-1})( \qq^{k, \varphi}_{i} - \qq^{k, \varphi}_{i-1})  \big) {\, \cdot \,} \qq \, \di x
\\
&
= \int_{\Om} d (z)  \nabla_{\q}^{2} h_{\gamma} (\p^{k}_{i} - \p^{k}_{i-1} )  (\overline{\qq}^{k}_{i, t} - \overline{\qq}^{k}_{i-1, t})  {\, \cdot \, } \qq \, \di x 
\\
&
\qquad + \int_{\Om} d( z ) \big( \nabla^{2}_{\q} h_{\gamma} (\boldsymbol\xi_{t}) - \nabla^{2}_{\q} h_{\gamma} (\p^{k}_{i} - \p^{k}_{i-1} )\big) ( \p^{k}_{i, t} - \p^{k}_{i-1, t} - \p^{k}_{i} + \p^{k}_{i-1} ) {\, \cdot\,} \qq \, \di x  \,. 
\end{split}
\end{displaymath}

Inserting the previous equalities in~\eqref{e.36}, choosing the test function $(v, \eeta, \qq) = ( \overline{v}^{k}_{i, t}, \overline{\eeta}^{k}_{i, t}, \overline{\qq}^{k}_{i, t}) \in \A(0)$, using~\eqref{e.C}--\eqref{e.H}, the Lipschitz continuity of~$\C(\cdot)$,~$\ha( \cdot)$,~$d(\cdot)$, and~$\nabla_{\q} h_{\gamma}$, the convexity of~$h_{\gamma}$, and Lemma~\ref{l.3}, we obtain the estimate
\begin{align}  \label{e.37}
\| ( \overline{v}^{k}_{i, t}, & \overline{\eeta}^{k}_{i, t}, \overline{\qq}^{k}_{i, t} ) \|_{H^{1}\times L^{2} \times L^{2}}^{2} \leq  \ C_{\gamma, k} \, t^{2} \| \varphi \|_{\infty}^{2} \| ( \overline{v}^{k}_{i, t}, \overline{\eeta}^{k}_{i, t}, \overline{\qq}^{k}_{i, t} ) \|_{H^{1} \times L^{2} \times L^{2}} 
\\
&
-  \int_{\Om} \big(\C(z_{t}) - \C(z) - t (\C'(z) \varphi)\big) \strain^{k}_{i} {\, \cdot\,} \overline{\eeta}^{k}_{i, t} \, \di x \nonumber
\\
&
-  \int_{\Om} \big(\ha(z_{t}) - \ha(z) - t (\ha'(z) \varphi)\big) \p^{k}_{i} {\, \cdot\,} \overline{\qq}^{k}_{i, t} \, \di x \nonumber
\\
&
- \int_{\Om}  \big( d(z_{t}) - d(z) - t \varphi \, d' (z)\big) \nabla_{\q} h_{\gamma} (\p^{k}_{i} - \p^{k}_{i-1}) {\, \cdot\,} \overline{\qq}^{k}_{i, t} \, \di x \nonumber
\\
&
 + \int_{\Om} d (z)  \nabla_{\q}^{2} h_{\gamma} (\p^{k}_{i} - \p^{k}_{i-1} )  \overline{\qq}^{k}_{i-1, t}  {\, \cdot \, } \overline{\qq}^{k}_{i, t} \, \di x \nonumber
\\
&
 - \int_{\Om} d( z ) \big( \nabla^{2}_{\q} h_{\gamma} (\boldsymbol\xi_{t}) - \nabla^{2}_{\q} h_{\gamma} (\p^{k}_{i} - \p^{k}_{i-1} )\big) ( \p^{k}_{i, t} - \p^{k}_{i-1, t} - \p^{k}_{i} + \p^{k}_{i-1} ) {\, \cdot\,} \overline{\qq}^{k}_{i, t} \, \di x  \nonumber
\\
&
 +   \int_{\Om} \big( \ell(z_{t}) - \ell(z) - t \varphi\, \ell'(z)  \big)  f^{k}_{i} {\, \cdot\,} \overline{v}^{k}_{i, t} \, \di x \,,\nonumber
\end{align}
for some positive constant~$C_{\gamma,k}$ dependent on~$\gamma \in (0,+\infty)$ and~$k \in \mathbb{N}$. In view of the regularity of~$\C( \cdot)$,~$\ha( \cdot)$,~$d(\cdot)$, and~$\ell(\cdot)$, of the bounds~$|\nabla_{\q} h_{\gamma} ( \q) | \leq 1$ and~$| \nabla^{2}_{\q} h_{\gamma} (\q)| \leq 2\gamma$, and of Lemma~\ref{l.3}, we can continue in~\eqref{e.37} with
\begin{align}\label{e.38}
\vphantom{\int} \| ( & \overline{v}^{k}_{i, t},  \overline{\eeta}^{k}_{i, t}, \overline{\qq}^{k}_{i,t} ) \|_{H^{1} \times L^{2} \times L^{2}}^{2} 
\\
&
\vphantom{\int} \leq  \widetilde{C}_{\gamma, k} \, t^{2} \| \varphi\|_{\infty}^{2}\big(  \| ( u^{k}_{i}, \strain^{k}_{i} , \p^{k}_{i} ) \|_{H^{1} \times L^{2} \times L^{2}} + \| f^{k}_{i}\|_{2} +  1 \big) \| (\overline{v}^{k}_{i, t}, \overline{\eeta}^{k}_{i, t}, \overline{\qq}^{k}_{i, t} ) \|_{H^{1} \times L^{2} \times L^{2}} \nonumber 
\\
&
\vphantom{\int} \qquad + \widetilde{C}_{\gamma} \| ( \overline{v}^{k}_{i-1, t}, \overline{\eeta}^{k}_{i-1, t}, \overline{\qq}^{k}_{i-1 ,t} ) \|_{H^{1} \times L^{2} \times L^{2}}\| ( \overline{v}^{k}_{i, t}, \overline{\eeta}^{k}_{i, t}, \overline{\qq}^{k}_{i,t} ) \|_{H^{1} \times L^{2} \times L^{2}} \nonumber
\\
&
\vphantom{\int} \qquad  - \int_{\Om} d( z ) \big( \nabla^{2}_{\q} h_{\gamma} (\boldsymbol\xi_{t}) - \nabla^{2}_{\q} h_{\gamma} (\p^{k}_{i} - \p^{k}_{i-1} )\big) ( \p^{k}_{i, t}  - \p^{k}_{i} ) {\, \cdot\,} \overline{\qq}^{k}_{i, t} \, \di x  \nonumber
\\
&
\vphantom{\int} \qquad +   \int_{\Om} d( z ) \big( \nabla^{2}_{\q} h_{\gamma} (\boldsymbol\xi_{t}) - \nabla^{2}_{\q} h_{\gamma} (\p^{k}_{i} - \p^{k}_{i-1} )\big) (  \p^{k}_{i-1, t} - \p^{k}_{i-1} ) {\, \cdot\,} \overline{\qq}^{k}_{i, t} \, \di x  \nonumber
\\
&
\vphantom{\int} \leq \overline{C}_{\gamma, k } \, t \| \varphi\|_{\infty} \Big( t \| \varphi \|_{\infty} \big( \| ( u^{k}_{i}, \strain^{k}_{i} , \p^{k}_{i} ) \|_{H^{1} \times L^{2} \times L^{2}} + \| f^{k}_{i}\|_{2} +  1 \big)  \nonumber
\\
&
\vphantom{\int} \qquad + \|  \nabla^{2}_{\q} h_{\gamma} (\boldsymbol\xi_{t}) - \nabla^{2}_{\q} h_{\gamma} (\p^{k}_{i} - \p^{k}_{i-1} ) \|_{\nu} \Big) \| ( \overline{v}^{k}_{i, t}, \overline{\eeta}^{k}_{i, t}, \overline{\qq}^{k}_{i,t} ) \|_{H^{1} \times L^{2} \times L^{2}} \nonumber
\\
&
\vphantom{\int} \qquad + \widetilde{C}_{\gamma} \| ( \overline{v}^{k}_{i-1, t}, \overline{\eeta}^{k}_{i-1, t}, \overline{\qq}^{k}_{i-1 ,t} ) \|_{H^{1} \times L^{2} \times L^{2}} \| ( \overline{v}^{k}_{i, t}, \overline{\eeta}^{k}_{i, t}, \overline{\qq}^{k}_{i,t} ) \|_{H^{1} \times L^{2} \times L^{2}} \nonumber
\end{align}
for some positive constants~$\widetilde{C}_{\gamma, k} ,
\overline{C}_{\gamma, k} $ depending on~$\gamma$ and~$k$, for
$\widetilde{C}_{\gamma}$ depending only on~$\gamma$, and for some $\nu
\in (1, +\infty)$. In order to conclude for~\eqref{e.35} we are
left to show that 
\begin{equation}\label{e.39}
\lim_{t \to 0} \, \| \nabla^{2}_{\q} h_{\gamma} (\boldsymbol\xi_{t}) - \nabla^{2}_{\q} h_{\gamma} (\p^{k}_{i} - \p^{k}_{i-1} ) \|_{\nu} = 0 \,.
\end{equation}
Arguing as in Proposition~\ref{p.2} we get that $ \p^{k}_{j,t} \to \p^{k}_{j} $ in $L^{2} (\Om; \M^{n}_{D})$ as $t\to0$ for every $j = 1, \ldots, k$, Hence, up to a subsequence we may assume that $\p^{k}_{j, t} \to \p^{k}_{j}$ a.e.~in~$\Om$ for $j=1, \ldots, k$, which implies that~$\boldsymbol\xi_{t} \to \p^{k}_{i} - \p^{k}_{i-1}$ and~$\nabla^{2}_{\q} h_{\gamma} (\boldsymbol\xi_{t}) \to \nabla^{2}_{\q} h_{\gamma}(\p^{k}_{i} - \p^{k}_{i-1} )$ a.e.~in~$\Om$. In view of the bound $| \nabla^{2}_{\q}  h_{\gamma} (\boldsymbol\xi_{t}) | \leq 2\gamma$ in~$\Om$ by the Dominated Convergence Theorem we get~\eqref{e.39}. This, together with~\eqref{e.38}, concludes the proof of~\eqref{e.35}. In particular, estimate~\eqref{e.35} can be made uniform in~$i$ as we have to control a finite number of norms $\| ( \overline{v}^{k}_{i, t}, \overline{\eeta}^{k}_{i, t}, \overline{\qq}^{k}_{i,t} ) \|_{H^{1} \times L^{2} \times L^{2}}$ for $i=1, \ldots, k$.
\end{proof}


\section{Optimality conditions}\label{s.optimality}

 The aim of this section is to  provide first-order optimality conditions for
the TO problem~\eqref{min2}, see Theorem~\ref{t.continuous-optimality}. This will
be obtained by passing to the limit in the corresponding optimality
conditions for the time-discrete TO problem~\eqref{min3}. Since we
believe this to be of independent interest, also in view of a possible
numerical implementation of this TO perspective, we analyse the
time-discrete problem in detail in Subsection
\ref{sub.time-discrete-optimality}.

\subsection{Optimality of the time-discrete problem}
\label{sub.time-discrete-optimality}
 In the following we give the first-order optimality conditions for the time-discrete problem~\eqref{min3} by passing to the limit as~$\gamma \to +\infty$ in~\eqref{e.8}--\eqref{e.9}. We start by proving a uniform bound for the adjoint variables~$(\overline{u}^{k}_{i, \gamma}, \overline{\strain}^{k}_{i, \gamma}, \overline{\p}^{k}_{i, \gamma}) _{i=0}^{k+1} \in \big( H^{1}(\Om; \R^{n}) \times L^{2}(\Om; \M^{n}_{S}) \times L^{2}(\Om; \M^{n}_{D}) \big)^{k+1}$ satisfying~\eqref{e.8}--\eqref{e.9}. From now on, we will use the notation
\begin{align}\label{e.norm-notation}
& \| \strain \|_{\C(z)}^{2} := \int_{\Om} \C(z) \strain{\, \cdot\,} \strain \, \di x  \qquad \| \p \|_{\ha(z)}^{2} := \int_{\Om} \ha(z) \p {\, \cdot\,} \p \, \di x 
\end{align}
for every $z \in L^{\infty}(\Om)$, every $\strain \in L^{2}(\Om; \M^{n}_{S})$, and every $\p \in L^{2}(\Om; \M^{n}_{D})$. In view of~\eqref{e.C}--\eqref{e.H},~$\| \cdot\|_{\C(z)}$ and~$\| \cdot\|_{\ha(z)}$ are two norms in $L^{2}(\Om; \M^{n}_{S})$ and $L^{2}(\Om; \M^{n}_{D})$, respectively, and are both equivalent to the usual $L^{2}$-norm, uniformly w.r.t.~$z \in L^{\infty}(\Om) $.

We now state the main result of this section.

\begin{theorem}[Optimality for the time-discrete TO problem]\label{t.discrete-optimality}
Let $k \in \mathbb{N}$ and $f^{k}_{i}, g^{k}_{i}, w^{k}_{i}$ be
defined as in~\eqref{e.fgw}. For~$\gamma \in (0,+\infty)$, let $z_{k,
  \gamma} \in H^{1}(\Om; [0,1])$ be a solution of the approximate
time-discrete \emph{TO} problem~\eqref{min4}. Assume that~$z_{k, \gamma} \rightharpoonup z_{k}$ weakly in~$H^{1}(\Om)$ as $\gamma \to +\infty$. Then, $z_{k} \in H^{1}(\Om; [0,1])$  solves~\eqref{min3}  and, denoted with~$(u^{k}_{i}, \strain^{k}_{i}, \p^{k}_{i})_{i=0}^{k}$ the corresponding time-discrete quasistatic evolution, there exist $(\brho^{k}_{i})_{i=0}^{k}, (\bpi^{k}_{i})_{i=1}^{k+1} \in \big( L^{2}(\Om; \M^{n}_{D}) \big)^{k+1}$, and $( \overline{u}^{k}_{i}, \overline{\strain}^{k}_{i}, \overline{\p}^{k}_{i})_{i=1}^{k+1} \in \big( H^{1}(\Om; \R^{n}) \times L^{2}(\Om; \M^{n}_{S}) \times L^{2}(\Om; \M^{n}_{D}) \big)^{k+1}$ such that for every every $i=0, \ldots, k$, every $(v, \eeta, \qq) \in \A(0)$, and $\varphi \in H^{1}(\Om) \cap L^{\infty}(\Om)$:
\begin{align}
& \int_{\Om} \C(z_{k}) \strain^{k}_{i}{\, \cdot\,} \eeta\, \di x + \int_{\Om} \ha(z_{k}) \p^{k}_{i}{\, \cdot\,} \qq \, \di x  \label{e.discrete-equilibrium1}
\\
&
\qquad  \qquad + \int_{\Om} \brho^{k}_{i}{\, \cdot\,} \qq \, \di x - \int_{\Om} \ell(z_{k}) f^{k}_{i} {\, \cdot\,} v \, \di x - \int_{\Gamma_{N}} g^{k}_{i} {\, \cdot\, } v \, \di \HH^{n-1} = 0\,, \nonumber
\\[2mm]
& 
\brho^{k}_{i} {\, \cdot\,} ( \p^{k}_{i} - \p^{k}_{i-1} ) = d(z) | \p^{k}_{i} - \p^{k}_{i-1} |\quad \text{in $\Om$}\,,\qquad  \p^{k}_{i} - \p^{k}_{i-1} = 0 \quad \text{in $\{ |\brho^{k}_{i}| < d(z) \}$}\,, \label{e.discrete-equilibrium2}
\\[2mm]
&
\int_{\Om} \C(z_{k}) \overline{\strain}^{k}_{i}{\, \cdot\,} \eeta \, \di x + \int_{\Om} \ha(z_{k}) \overline{\p}^{k}_{i} {\, \cdot\,} \qq \, \di x \label{e.discrete-optimality1} 
\\
&
\qquad \qquad + \int_{\Om} \bpi^{k}_{i}{\, \cdot\,} \qq\, \di x  - \int_{\Om} \ell(z_{k}) f^{k}_{i} {\, \cdot\,} v \, \di x - \int_{\Gamma_{N}} \!\! g^{k}_{i} {\, \cdot\,} v \, \di \HH^{n-1} = 0\,, \nonumber 
\\[2mm]
& 
 \sum_{j=1}^{k} \bigg( \int_{\Om} \varphi \, \ell'(z_{k}) ( f^{k}_{j} - f^{k}_{j-1}) {\, \cdot\,} (u^{k}_{j} + \overline{u}^{k}_{j} )\, \di x 
 - \int_{\Om} \big( \C'(z_{k}) \varphi \big) ( \strain^{k}_{j} - \strain^{k}_{j-1}) {\, \cdot\,} \overline{\strain}^{k}_{j} \, \di x  \label{e.discrete-optimality2}
 \\
 &
 \qquad \qquad  - \int_{\Om} \big( \ha'(z_{k}) \varphi \big) ( \p^{k}_{j} - \p^{k}_{j-1}) {\, \cdot\,} \overline{\p}^{k}_{j} \, \di x - \int_{\Om} \varphi \, \frac{d'(z_{k})}{d'(z_{k})} \, ( \brho^{k}_{j} - \brho^{k}_{j-1} ) {\, \cdot\,} \overline{\p}^{k}_{j} \, \di x \bigg) \,, \nonumber
\\
&
\qquad \qquad  + \int_{\Om} \delta \nabla{z_{k}}{\, \cdot\,} \nabla{\varphi} + \frac{\varphi}{\delta} ( z_{k} (1 - z_{k})^{2} - z_{k}^{2} ( 1 - z_{k} )) \, \di x = 0 \,. \nonumber
\\[2mm]
&
\hphantom{---------} \vphantom{\int} \bpi^{k}_{i} {\,\cdot \,} (\p^{k}_{i} - \p^{k}_{i-1}) = 0 \qquad \text{in $\Om$}\,, \label{e.discrete-optimality3}
\\[2mm]
&
\hphantom{---------} \vphantom{\int} \overline{\p}^{k}_{i} - \overline{\p}^{k}_{i+1} = 0 \qquad \text{in $\{ | \brho^{k}_{i} | < d(z_{k}) \}$}
\label{e.discrete-optimality4}
\end{align}
\end{theorem}

In order to prove Theorem~\ref{t.discrete-optimality}, we need to establish some uniform bounds for the adjoint system~\eqref{e.9} of Corollary~\ref{c.approx-optimality}. This is the content of the following proposition.

\begin{proposition}[Uniform bounds]\label{p.bound-adjoint}
For every $\gamma \in (0,+\infty)$ and~$k \in \mathbb{N}$,
let~$z_{k}^{\gamma} \in H^{1}(\Om;[0,1])$ be a solution of the
approximate time-discrete \emph{TO} problem~\eqref{min4} with corresponding approximate time-discrete quasistatic evolution~$(u^{k, \gamma}_{i}, \strain^{k, \gamma}_{i}, \p^{k, \gamma}_{i}) _{i=0}^{k+1} \in \big( H^{1}(\Om; \R^{n}) \times L^{2}(\Om; \M^{n}_{S}) \times L^{2}(\Om; \M^{n}_{D}) \big)^{k+1}$. Furthermore, let $(\overline{u}^{k, \gamma}_{i}, \overline{\strain}^{k, \gamma}_{i}, \overline{\p}^{k, \gamma}_{i}) _{i=1}^{k+1} \in \big( H^{1}(\Om; \R^{n}) \times L^{2}(\Om; \M^{n}_{S}) \times L^{2}(\Om; \M^{n}_{D}) \big)^{k+1}$ be the adjoint variables introduced in Corollary~\ref{c.approx-optimality}. Then, $(\overline{u}^{k, \gamma}_{i}, \overline{\strain}^{k, \gamma}_{i}, \overline{\p}^{k, \gamma}_{i})$ is bounded in~$H^{1}(\Om; \R^{n}) \times L^{2}(\Om; \M^{n}_{S}) \times L^{2}(\Om; \M^{n}_{D})$ uniformly w.r.t.~$i$,~$k$, and~$\gamma$.
\end{proposition}

\begin{proof}
We test the equation~\eqref{e.8} for $i=k, \ldots, 1$ with the triple $(\overline{u}^{k, \gamma}_{i} - \overline{u}^{k, \gamma}_{i+1} , \overline{\strain}^{k, \gamma}_{i} - \overline{\strain}^{k, \gamma}_{i+1} , \overline{\p}^{k, \gamma}_{i} - \overline{\p}^{k, \gamma}_{i+1}) \in \A(0)$. Since the function~$h_{\gamma}$ is convex, we have that
\begin{align}\label{e.40}
\int_{\Om} & \C(z_{k}^{\gamma}) \overline{\strain}^{k, \gamma}_{i} {\, \cdot\,} ( \overline{\strain}^{k, \gamma}_{i} - \overline{\strain}^{k, \gamma}_{i+1} ) \, \di x + \int_{\Om} \ha(z^{\gamma}_{k}) \overline{\p}^{k, \gamma}_{i} {\, \cdot\,} ( \overline{\p}^{k, \gamma}_{i} - \overline{\p}^{k, \gamma}_{i+1}) \, \di x 
\\
&
 - \int_{\Om} \ell(z_{k}^{\gamma}) f^{k}_{i}{\, \cdot\,} (\overline{u}^{k, \gamma}_{i} - \overline{u}^{k, \gamma}_{i+1} ) \, \di x - \int_{\Gamma_{N}} g^{k}_{i} {\, \cdot\,} (\overline{u}^{k, \gamma}_{i} - \overline{u}^{k, \gamma}_{i+1} ) \, \di \HH^{n-1} \leq 0\,. \nonumber
\end{align}
We rewrite the first term in~\eqref{e.40} as
\begin{align*}
\int_{\Om} & \C(z_{k}^{\gamma}) \overline{\strain}^{k, \gamma}_{i} {\, \cdot\,} ( \overline{\strain}^{k, \gamma}_{i} - \overline{\strain}^{k, \gamma}_{i+1} ) \, \di x = \frac12 \| \overline{\strain}^{k, \gamma}_{i} \|_{\C(z_{k}^{\gamma})}^{2} - \frac12 \| \overline{\strain}^{k, \gamma}_{i+1} \|_{\C(z_{k}^{\gamma})}^{2} + \frac12 \| \overline{\strain}^{k, \gamma}_{i} - \overline{\strain}^{k, \gamma}_{i+1} \|_{\C(z_{k}^{\gamma})}^{2}\,.
\end{align*}
In a similar way we can rewrite the second term in~\eqref{e.40}, obtaining
\begin{align}\label{e.41}
& \frac12 \| \overline{\strain}^{k, \gamma}_{i} \|_{\C(z_{k}^{\gamma})}^{2} - \frac12 \| \overline{\strain}^{k, \gamma}_{i+1} \|_{\C(z_{k}^{\gamma})}^{2} + \frac12 \| \overline{\strain}^{k, \gamma}_{i} - \overline{\strain}^{k, \gamma}_{i+1} \|_{\C(z_{k}^{\gamma})}^{2} 
\\
&
\qquad + \frac12 \| \overline{\p}^{k, \gamma}_{i} \|_{\ha(z_{k}^{\gamma}), 2}^{2} - \frac12 \| \overline{\p}^{k, \gamma}_{i+1} \|_{\ha(z_{k}^{\gamma}), 2}^{2} + \frac12 \| \overline{\p}^{k, \gamma}_{i} - \overline{\p}^{k, \gamma}_{i+1} \|_{\ha(z_{k}^{\gamma}), 2}^{2} \nonumber
\\
&
\qquad  - \int_{\Om} \ell(z_{k}^{\gamma}) f^{k}_{i}{\, \cdot\,} (\overline{u}^{k, \gamma}_{i} - \overline{u}^{k, \gamma}_{i+1} ) \, \di x - \int_{\Gamma_{N}} g^{k}_{i} {\, \cdot\,} (\overline{u}^{k, \gamma}_{i} - \overline{u}^{k, \gamma}_{i+1} ) \, \di \HH^{n-1} \leq 0\,. \nonumber
\end{align}
For every $j \in \{1, \ldots, k \}$ we sum up~\eqref{e.41} over $i=k, \ldots, j$ and use that $\overline{u}^{k}_{k+1, \gamma} = 0$, so that
\begin{align}
 \frac12  \| \overline{\strain}^{k, \gamma}_{j} \|_{\C(z_{k}^{\gamma})}^{2} &  + \frac12 \| \overline{\p}^{k, \gamma}_{j} \|_{\ha(z_{k}^{\gamma})}^{2} 
\\
&
\leq \sum_{i=k}^{j+1} \bigg( \int_{\Om} \ell(z_{k}^{\gamma}) (f^{k}_{i} - f^{k}_{i-1}) {\, \cdot\,} \overline{u}^{k, \gamma}_{i} \, \di x + \int_{\Gamma_{N}} ( g^{k}_{i} - g^{k}_{i-1} ) {\, \cdot\,} \overline{u}^{k, \gamma}_{i} \, \di \HH^{n-1} \bigg) \nonumber
\\
&
\qquad + \int_{\Om} \ell(z_{k}^{\gamma}) f^{k}_{j} {\, \cdot\,} \overline{u}^{k, \gamma}_{j} \, \di x + \int_{\Gamma_{N}} g^{k}_{j}{\, \cdot\,} \overline{u}^{k, \gamma}_{j} \, \di \HH^{n-1} \,. \nonumber
\end{align}
By Cauchy inequality and by the regularity of~$f$ and~$g$ we deduce that
\begin{align*}
& \sup_{i = 1, \ldots, k}   \big( \| \overline{\strain}^{k, \gamma}_{i} \|_{\C(z_{k}^{\gamma})} + \| \overline{\p}^{k, \gamma}_{i} \|_{\ha(z_{k}^{\gamma})} \big)^{2}
\\
&
\leq C \sup_{i= 1, \ldots, k}  \big( \| \overline{\strain}^{k, \gamma}_{i} \|_{\C(z_{k}^{\gamma})} + \| \overline{\p}^{k, \gamma}_{i} \|_{\ha(z_{k}^{\gamma})} \big) \bigg(  \int_{0}^{T} (\| \dot{f}_{k}(t) \|_{2} + \| \dot{g}_{k}(t) \|_{2} ) \, \di t + \sup_{i = 1, \ldots, k} ( \| f^{k}_{i} \|_{2} + \| g^{k}_{i}\|_{2}) \bigg)
\\
&
\leq C \sup_{i= 1, \ldots, k}  \big( \| \overline{\strain}^{k, \gamma}_{i} \|_{\C(z_{k}^{\gamma})} + \| \overline{\p}^{k, \gamma}_{i} \|_{\ha(z_{k}^{\gamma})} \big) \big( \| f \|_{H^{1}(0,T; H^{1}(\Om; \R^{n}))}  +  \| g \|_{H^{1}(0,T; L^{2}(\Gamma_{N}; \R^{n}))} \big)\,,
\end{align*}
for some positive constant~$C$ independent of~$i$, $k$, and~$\gamma$. The above inequality implies the boundedness of~$(\overline{u}^{k, \gamma}_{i}, \overline{\strain}^{k, \gamma}_{i}, \overline{\p}^{k, \gamma}_{i})$ in~$H^{1}(\Om; \R^{n}) \times L^{2}(\Om; \M^{n}_{S}) \times L^{2}(\Om; \M^{n}_{D})$ uniformly w.r.t.~$i$,~$k$, and~$\gamma$.
\end{proof}

We now prove Theorem~\ref{t.discrete-optimality}.

\begin{proof}[Proof of Theorem~\ref{t.discrete-optimality}]
Let~$z_{k}^{\gamma}, z_{k} \in H^{1}(\Om;[0,1])$ be as in the statement of the Theorem, and let~$(u^{k, \gamma}_{i}, \strain^{k, \gamma}_{i}, \p^{k, \gamma}_{i})_{i=0}^{k}, (u^{k}_{i}, \strain^{k}_{i}, \p^{k}_{i})_{i=0}^{k} \in \big( H^{1}(\Om; \R^{n}) \times L^{2}(\Om; \M^{n}_{S}) \times L^{2}(\Om; \M^{n}_{D}) \big) ^{k+1}$ be the corresponding approximate time-discrete and time-discrete evolutions, respectively. By Proposition~\ref{p.5}, we know that $(u^{k, \gamma}_{i}, \strain^{k, \gamma}_{i}, \p^{k, \gamma}_{i})_{i=0}^{k}$ converges to $(u^{k}_{i}, \strain^{k}_{i}, \p^{k}_{i})_{i=0}^{k}$ in $\big( H^{1}(\Om; \R^{n}) \times L^{2}(\Om; \M^{n}_{S}) \times L^{2}(\Om; \M^{n}_{D}) \big) ^{k+1}$ as $\gamma \to + \infty$. 

Equations~\eqref{e.discrete-equilibrium1}--\eqref{e.discrete-equilibrium2} are equivalent to the equilibrium condition~\eqref{e.3} of Definition~\ref{d.discrete-evolution}. In particular, we have that
\begin{displaymath}
\brho^{k}_{i} = \Pi_{\M^{n}_{D}} \big( \C(z_{k}) \strain^{k}_{i} \big)  - \ha(z_{k}) \p^{k}_{i} \qquad \text{in $L^{2}(\Om; \M^{n}_{D})$}\,.  
\end{displaymath}
Furthermore, setting $\brho^{k, \gamma}_{i} := d(z_{k}^{\gamma}) \nabla_{\q} h_{\gamma} (\p^{k, \gamma}_{i} - \p^{k, \gamma}_{i-1})$ we have that $\brho^{k, \gamma}_{i} = \Pi_{\M^{n}_{D}} \big( \C(z_{k}^{\gamma}) \strain^{k, \gamma}_{i} \big)  - \ha(z_{k}) \p^{k, \gamma}_{i}$ and $\brho^{k, \gamma}_{i} \to \brho^{k}_{i}$ in $L^{2}(\Om; \M^{n}_{D})$ as $\gamma \to + \infty$ for every $i$ and~$k$.

Denoting by $(\overline{u}^{k, \gamma}_{i}, \overline{\strain}^{k,
  \gamma}_{i}, \overline{\p}^{k, \gamma}_{i})_{i=1}^{k+1} \in  \A(0)
^{k+1}$ the adjoint variables  introduced in Corollary~\ref{c.approx-optimality}, we have by Proposition~\ref{p.bound-adjoint} that $(\overline{u}^{k, \gamma}_{i}, \overline{\strain}^{k, \gamma}_{i}, \overline{\p}^{k, \gamma}_{i})$ are bounded in $H^{1}(\Om; \R^{n}) \times L^{2}(\Om; \M^{n}_{S}) \times L^{2}(\Om; \M^{n}_{D})$ uniformly w.r.t.~$i$,~$k$, and~$\gamma$. Thus, we may assume that, up to a subsequence, $(\overline{u}^{k, \gamma}_{i}, \overline{\strain}^{k, \gamma}_{i}, \overline{\p}^{k, \gamma}_{i}) \rightharpoonup (\overline{u}^{k}_{i}, \overline{\strain}^{k}_{i}, \overline{\p}^{k}_{i})$ weakly in $H^{1}(\Om; \R^{n}) \times L^{2}(\Om; \M^{n}_{S}) \times L^{2}(\Om; \M^{n}_{D})$ as $\gamma \to +\infty$ for every~$i$ and~$k$.

In order to prove that $ (\overline{u}^{k}_{i}, \overline{\strain}^{k}_{i}, \overline{\p}^{k}_{i})$ satisfies~\eqref{e.discrete-optimality1}--\eqref{e.discrete-optimality2}, we first rewrite the optimality conditions~\eqref{e.8}--\eqref{e.9} for $\gamma \in (0,+\infty)$ in a form similar to~\eqref{e.discrete-optimality1}. To this aim, we define for $i=1, \ldots, k$
\begin{displaymath}
\bpi^{k, \gamma}_{i} := d(z_{k}^{\gamma}) \nabla^{2}_{\q} h_{\gamma} (\p^{k, \gamma}_{i} - \p^{k, \gamma}_{i-1}) (\overline{\p}^{k, \gamma}_{i} - \overline{\p}^{k, \gamma}_{i+1}) \qquad \text{in $L^{2}(\Om; \M^{n}_{D})$}\,.
\end{displaymath}
Hence, we rewrite~\eqref{e.8} as
\begin{align}\label{e.8.2}
  \int_{\Om} \C ( z_{k}^{\gamma} )  \overline{\strain}^{k, \gamma}_{i} {\, \cdot\,} \eeta \, \di x  & + \int_{\Om} \ha ( z_{k}^{\gamma } ) \overline{\p}^{k, \gamma}_{i} {\, \cdot\,} \qq \, \di x 
+ \int_{\Om} \bpi^{k, \gamma}_{i} {\, \cdot\,} \qq \, \di x  
\\
&
- \int_{\Om} \ell ( z_{k}^{\gamma}) f^{k}_{i} {\, \cdot\,} v \, \di x - \int_{\Gamma_{N}} g^{k}_{i} {\, \cdot\,} v \, \di \HH^{n-1} = 0 \,, \nonumber
\end{align}
for $(v, \eeta, \qq) \in \A(0)$. From~\eqref{e.8.2} tested against $(0, -\qq, \qq)$ with~$\qq \in L^{2}(\Om; \M^{n}_{D})$ we deduce that
\begin{displaymath}
\bpi^{k, \gamma}_{i} = \Pi_{\M^{n}_{D}} \big( \C (z_{k}^{\gamma}) \overline{\strain}^{k, \gamma}_{i} \big) - \ha (z_{k}^{\gamma}) \overline{\p}^{k, \gamma}_{i} \,.
\end{displaymath}
Thus, setting
\begin{displaymath}
\bpi^{k}_{i}:= \Pi_{\M^{n}_{D}} \big( \C (z_{k}) \overline{\strain}^{k}_{i} \big) - \ha (z_{k}) \overline{\p}^{k}_{i} \qquad \text{for $i=1, \ldots, k+1$}
\end{displaymath}
we infer that $\bpi^{k, \gamma}_{i} \rightharpoonup \bpi^{k}_{i}$ weakly in $L^{2}(\Om; \M^{n}_{D})$ as~$\gamma \to +\infty$, for every $k$ and~$i$. Passing to the limit as~$\gamma \to +\infty$ in~\eqref{e.8.2} we deduce~\eqref{e.discrete-optimality1}.

As for~\eqref{e.discrete-optimality2}, we rewrite~\eqref{e.9} as
\begin{align}\label{e.9.2}
& 
 \sum_{j=1}^{k} \bigg( \int_{\Om} \varphi \, \ell'(z_{k}^{\gamma}) ( f^{k}_{j} - f^{k}_{j-1}) {\, \cdot\,} (u^{k, \gamma}_{j} + \overline{u}^{k, \gamma}_{j} )\, \di x 
 - \int_{\Om} \big( \C'(z_{k}^{\gamma}) \varphi \big) ( \strain^{k, \gamma}_{j} - \strain^{k, \gamma}_{j-1}) {\, \cdot\,} \overline{\strain}^{k, \gamma}_{j} \, \di x 
 \\
 &
 \qquad \qquad  - \int_{\Om} \big( \ha'(z_{k}^{\gamma}) \varphi \big)
   ( \p^{k, \gamma}_{j} - \p^{k, \gamma}_{j-1}) {\, \cdot\,}
   \overline{\p}^{k, \gamma}_{j} \, \di x - \int_{\Om} \varphi \,
   \frac{d'(z_{k}^{\gamma})}{d(z_{k}^{\gamma} )} \, ( \brho^{k, \gamma}_{j} - \brho^{k, \gamma}_{j-1} ) {\, \cdot\,} \overline{\p}^{k, \gamma}_{j} \, \di x \bigg)  \nonumber
\\
&
\qquad \qquad  + \int_{\Om} \delta \nabla{z_{k}^{\gamma}}{\, \cdot\,} \nabla{\varphi} + \frac{\varphi}{\delta} ( z_{k}^{\gamma} (1 - z_{k}^{\gamma})^{2} - (z_{k}^{\gamma})^{2} ( 1 - z_{k}^{\gamma} )) \, \di x = 0 \,. \nonumber
\end{align}
Owing to the convergences discussed above, we again infer~\eqref{e.discrete-optimality2} by passing to the limit in~\eqref{e.9.2} as~$\gamma \to +\infty$.

Finally, the proof of~\eqref{e.discrete-optimality3}--\eqref{e.discrete-optimality4} can be obtained by repeating step by step the proof of~\cite[formula~(4.5), Theorem~4.1]{Alm-Ste_20}.
\end{proof}

\begin{remark}\label{r.bound-adjoint}
We notice that $(\overline{u}^{k}_{i}, \overline{\strain}^{k}_{i}, \overline{\p}^{k}_{i})$ is bounded in~$H^{1}(\Om; \R^{n}) \times L^{2}(\Om; \M^{n}_{S}) \times L^{2}(\Om, \M^{n}_{D})$ uniformly w.r.t.~$i$ and~$k$ as a consequence of Proposition~\ref{p.bound-adjoint}.
\end{remark}


\subsection{Optimality of the time-continuous problem}

We conclude with the first-order optimality conditions for the TO problem~\eqref{min2}. Most of the conditions follow directly from those computed in Theorem~\ref{t.discrete-optimality} by passing to the limit as the time step~$\tau_{k}$ tends to~$0$. The only difficulty is to find the time-continuous condition corresponding to~\eqref{e.discrete-optimality4}, since the adjoint variable~$\overline{\p}^{k}_{i}$ can be bounded in~$L^{2}(\Om; \M^{n}_{D})$ uniformly w.r.t.~$i$ and~$k$ (see Remark~\ref{r.bound-adjoint}), but no time regularity is expected.

\begin{theorem}[Optimality for the TO problem]\label{t.continuous-optimality}
Let $z_{k} \in H^{1}(\Om; [0,1])$ be the sequence of solutions of the
time-discrete \emph{TO} problem~\eqref{min3} found in Theorem~\ref{t.discrete-optimality}. Then, there exists $z \in H^{1}(\Om; [0,1])$ solving~\eqref{min2} such that, up to a subsequence, $z_{k} \rightharpoonup z$ weakly in~$H^{1}(\Om)$. Denoting by $(u( \cdot), \strain(\cdot), \p(\cdot))$ the quasistatic evolution corresponding to~$z$, there exists~$\brho \in H^{1}(0,T; L^{2}(\Om; \M^{n}_{D}))$ such that for every $(v, \eeta, \qq) \in \A(0)$ and every $t \in [0,T]$ the following holds:
\begin{align}
& \int_{\Om} \C(z) \strain(t) {\, \cdot\,} \eeta \, \di x + \int_{\Om} \ha(z) \p(t) {\, \cdot\,} \qq \, \di x + \int_{\Om} \brho(t) {\, \cdot\,} \qq \, \di x \label{e.continuous-equilibrium1}
\\
&
\qquad\qquad\qquad\quad\,\,\, - \int_{\Om} \ell(z) f(t) {\, \cdot\,} v \, \di x - \int_{\Gamma_{N}} g(t) {\, \cdot\,} v \, \di x =  0 \,, \nonumber
\\[2mm]
&
\brho(t){\, \cdot\,} \dot{\p}(t) = d(z) | \dot{\p} (t) | \quad \text{in~$\Om$}\,, \qquad |\dot{\p}(t) | = 0 \quad \text{in~$\{ |\brho(t)| < d(z) \}$}\,. \label{e.continuous-equilibrium2}
\end{align}

Furthermore, there exist the adjoint variables $\overline{\p}_{0} \in L^{2}(\Om; \M^{n}_{D})$, $\bpi \in L^{\infty}(0,T; L^{2}(\Om; \M^{n}_{D}))$, and $(\overline{u}, \overline{\strain}, \overline{\p}) \in L^{\infty}(0,T; H^{1}(\Om; \R^{n}) \times L^{2}(\Om; \M^{n}_{S}) \times L^{2}(\Om; \M^{n}_{D}))$,  such that for every~$(v, \eeta, \qq) \in \A(0)$, for every $\varphi \in H^{1}(\Om) \cap L^{\infty}(\Om)$, and for a.e.~$t \in [0,T]$ we have
\begin{align}
&\int_{\Om} \C(z) \overline{\strain}(t) {\, \cdot\,} \eeta \, \di x + \int_{\Om} \ha(z) \overline{\p}(t) {\, \cdot\,} \qq , \di x + \int_{\Om} \bpi(t) {\, \cdot\,} \qq \,\di x  \label{e.continuous-optimality1}
\\
&
\qquad\qquad\qquad \quad \,\,\, - \int_{\Om} \ell(z) f(t) {\, \cdot\,} v \, \di x - \int_{\Gamma_{N}} g(t) {\, \cdot\,} v \, \di \HH^{n-1} =0 \,, \nonumber
\\[2mm]
&
 \int_{0}^{T} \bigg(  \int_{\Om} \varphi \, \ell' ( z ) \dot{f}(t) {\, \cdot\,}  ( u(t) + \overline{u}(t)) \, \di x - \int_{\Om} \big( \C'(z) \varphi \big) \dot{\strain}(t) {\, \cdot\,} \overline{\strain}(t) \, \di x    \label{e.continuous-optimality2}
\\
&
\qquad\qquad    - \int_{\Om} \big( \ha'(z) \varphi\big) \dot{\p}(t) {\, \cdot\,} \overline{\p}(t) \, \di x - \int_{\Om} \varphi \, \frac{d'(z)}{d(z)} \,  \dot{\brho}(t) {\, \cdot\,} \overline{\p}(t) \, \di x \bigg) \di t  \nonumber
\\
&
\qquad \qquad + \int_{\Om} \delta \nabla{z}{\, \cdot\,} \nabla{\varphi} + \frac{\varphi}{\delta} ( z (1 - z)^{2} - z^{2} ( 1 - z )) \, \di x = 0 \,. \nonumber
\\[2mm]
&
\hphantom{-----------}\vphantom{\int}\bpi(t) {\, \cdot\,} \dot{\p}(t) = 0 \quad \text{in~$\Om$}\,, \label{e.continuous-optimality3}
\\[2mm]
&
\int_{\Om} d^{2} (z) \,  \overline{\p}_{0}\, \di x  - 2 \int_{0}^{T} \int_{\Om} ( \brho(t) {\, \cdot\,} \dot{\brho}(t)) \, \overline{\p}(t)  \, \di x \, \di t = 0 \,. \label{e.continuous-optimality4}
\end{align}
\end{theorem}

\begin{proof}
Conditions~\eqref{e.continuous-equilibrium1}--\eqref{e.continuous-equilibrium2} are a direct consequence of Definition~\ref{d.quasistatic}. In particular, $\brho(t) = \Pi_{\M^{n}_{D}} \big( \C(z) \strain(t) \big) - \ha(z) \p(t)$ and $\brho(t) \in d(z) \partial| \cdot | (\dot{\p}(t))$ for a.e.~$t \in [0,T]$. 

Let us consider the time-discrete quasistatic evolution ~$(u^{k}_{i}, \strain^{k}_{i}, \p^{k}_{i})_{i=0}^{k} \in \big( H^{1}(\Om; \R^{n}) \times L^{2}(\Om; \M^{n}_{S}) \times L^{2}(\Om, \M^{n}_{D}) \big)^{k+1}$ associated with~$z_{k}$ and let $(\overline{u}^{k}_{i}, \overline{\strain}^{k}_{i}, \overline{\p}^{k}_{i})_{i=1}^{k+1} \in \big( H^{1}(\Om; \R^{n}) \times L^{2}(\Om; \M^{n}_{S}) \times L^{2}(\Om, \M^{n}_{D}) \big)^{k+1}$, $(\brho^{k}_{i})_{i=0}^{k}, (\bpi^{k}_{i})_{i=1}^{k+1} \in \big( L^{2}(\Om; \M^{n}_{D})\big)^{k+1}$ be the corresponding adjoint variables introduced in Theorem~\ref{t.discrete-optimality}. We further define the interpolation functions
\begin{align*}
&\overline{u}_{k}(t) := \overline{u}^{k}_{i} \,, \qquad \overline{\strain}_{k}(t) := \overline{\strain}^{k}_{i} \,, \qquad \overline{\p}_{k}(t) := \overline{\p}^{k}_{i} \,,
\\
&
 \tu_{k}(t):= u^{k}_{i}\,,\qquad  \tstrain_{k}(t) := \strain^{k}_{i}\,, \qquad \tp_{k}(t):= \p^{k}_{i}\,,
 \\
 &
  \trho_{k}(t) := \brho^{k}_{i} \,, \qquad \tpi_{k}(t) := \bpi^{k}_{i} \,,
 \\
 &
 \tilde{f}_{k}(t) := f^{k}_{i}\,, \qquad \tilde{g}_{k}(t) := g^{k}_{i}\,, \qquad \tilde{w}_{k}(t) := w^{k}_{i}\,,
 \\
 &\brho_{k}(t) := \brho^{k}_{i-1} + \frac{(t - t^{k}_{i-1})}{\tau_{k}} \, ( \brho^{k}_{i} - \brho^{k}_{i-1})
\end{align*}
for $t \in (t^{k}_{i-1}, t^{k}_{i}]$. We recall that the piecewise affine interpolation functions~$f_{k}$, $g_{k}$, $w_{k}$, $u_{k}$, $\strain_{k}$, and~$\p_{k}$ have been introduced in~\eqref{e.fgw2} and~\eqref{e.affine-interpolant}. As a consequence of Proposition~\ref{p.3}, we have that $(\tilde{u}_{k}, \tstrain_{k}, \tp_{k}) \to (u, \strain, \p)$ in $L^{\infty}(0,T; H^{1}(\Om; \R^{n}) \times L^{2}(\Om; \M^{n}_{S}) \times L^{2}(\Om; \M^{n}_{D}))$ and  $\brho_{k} \to \brho$ in~$H^{1}(0,T; L^{2}(\Om; \M^{n}_{D}))$ (see also Lemma~\ref{l.A1}). By the equilibrium conditions~\eqref{e.discrete-equilibrium1}--\eqref{e.discrete-equilibrium2} we also infer that $\trho_{k}(t) \in d(z_{k}) \partial | \cdot | (\dot{\p}_{k}(t) ) $, which implies that $\trho_{k}, \brho_{k}$ are bounded in~$L^{\infty}(0,T; L^{\infty}(\Om ; \M^{n}_{D}))$. Moreover, by Proposition~\ref{p.bound-adjoint} we have that~$\tpi_{k}$ and~$(\overline{u}_{k}, \overline{\strain}_{k}, \overline{\p}_{k})$ are bounded in~$L^{\infty} ( 0,T ; L^{2} ( \Om ; \M^{n}_{D}) )$ and in $L^{\infty} ( 0,T ; H^{1} ( \Om ; \R^{n} ) \times L^{2} (\Om ;  \M^{n}_{S} ) \times L^{2} ( \Om ; \M^{n}_{D}))$, respectively. Therefore, we may assume that, up to a subsequence, $\tpi \xrightharpoonup{*} \bpi$ weakly$^*$ in $L^{\infty}(0,T; L^{2} ( \Om ; \M^{n}_{D}))$ and ~$(\overline{u}_{k}, \overline{\strain}_{k}, \overline{\p}_{k}) \xrightharpoonup{*} (\overline{u}, \overline{\strain}, \overline{\p})$ weakly$^*$ in $L^{\infty}(0,T; H^{1}(\Om; \R^{n}) \times L^{2}(\Om; \M^{n}_{S}) \times L^{2}(\Om; \M^{n}_{D}))$.

Let us show that~$\bpi$ and~$(\overline{u}, \overline{\strain},
\overline{\p})$
satisfy~\eqref{e.continuous-optimality1}--\eqref{e.continuous-optimality4}. We
start with~\eqref{e.continuous-optimality1}. Let  us  fix an at most
countable and dense subset~$D$ of~$\A(0)$. For every $(v, \eeta, \qq)
\in D$, every $\psi \in C^{\infty}_{c}(0,T)$, and every $t \in[0,T]$,
we consider the test function $(\psi(t) v, \psi(t) \eeta, \psi(t) \p)
\in \A(0)$ and rewrite the optimality
condition~\eqref{e.discrete-optimality1} as
\begin{align}\label{e.42}
 \int_{\Om} \C(z_{k}) \overline{\strain}_{k} (t) {\, \cdot\,} \psi(t) \eeta\, \di x & + \int_{\Om} \ha(z_{k}) \overline{\p}_{k} (t) {\, \cdot\,} \psi(t) \qq \, \di x   + \int_{\Om} \tpi_{k} (t) {\, \cdot\,} \psi(t) \qq \, \di x
\\
&
 - \int_{\Om} \ell(z_{k}) \tilde{f}_{k} (t)  {\, \cdot\,} \psi(t) v \, \di x - \int_{\Gamma_{N}} \tilde{g}_{k} (t) {\, \cdot\, } \psi(t) v \, \di \HH^{n-1} = 0\,. \nonumber
\end{align}
We integrate~\eqref{e.42} over~$[0,T]$ and pass to the limit as~$k \to \infty$. In view of the above convergences, we infer that for every $\psi \in C_{c}^{\infty}(0,T)$
\begin{align}\label{e.43}
\int_{0}^{T} \psi(t) \bigg( \int_{\Om} \C(z) \overline{\strain} (t) {\, \cdot\,}  \eeta\, \di x & + \int_{\Om} \ha(z) \overline{\p} (t) {\, \cdot\,}  \qq \, \di x   + \int_{\Om} \bpi (t) {\, \cdot\,}  \qq \, \di x
\\
&
 - \int_{\Om} \ell(z )f (t)  {\, \cdot\,} v \, \di x - \int_{\Gamma_{N}} g (t) {\, \cdot\, } v \, \di \HH^{n-1} \bigg) \, \di t = 0 \,. \nonumber
\end{align}
Since $D$ is at most coutable, we deduce from~\eqref{e.43} that~\eqref{e.continuous-optimality1} holds for a.e.~$t \in [0,T]$ and for every $(v, \eeta, \qq) \in D$. By density we extend the equality to~$\A(0)$.

Arguing in the same way, we can also prove that~\eqref{e.continuous-optimality3} holds for a.e.~$t \in [0,T]$, as the corresponding time-discrete condition~\eqref{e.discrete-optimality3} holds for every $t \in [0,T]$ and only the time derivative~$\dot{\p}_{k}$ is involved, which converges to~$\dot{\p}$ in $L^{2}(0,T; L^{2} (\Om; \M^{n}_{D}))$.

As for~\eqref{e.continuous-optimality2}, for every $\varphi \in H^{1} ( \Om ) \cap L^{\infty} ( \Om )$ we rewrite~\eqref{e.discrete-optimality2} as
\begin{align}\label{e.44}
& \int_{0}^{T} \Bigg( \int_{\Om} \varphi \, \ell' ( z_{k} )\dot{f}_{k}(t)  {\, \cdot\,} ( \tilde{u}_{k} (t) + \overline{u}_{k}(t))  \, \di x - \int_{\Om} \big( \C'(z_{k}) \varphi \big) \dot{\strain}_{k}(t) {\, \cdot\,} \overline{\strain}_{k}(t) \, \di x 
\\
&
\qquad  \qquad - \int_{\Om} \big( \ha'(z_{k}) \varphi\big) \dot{\p}_{k}(t) {\, \cdot\,} \overline{\p}_{k}(t) \, \di x - \int_{\Om} \varphi \, \frac{ d'(z_{k})}{d(z_{k})} \, \dot{\brho}_{k}(t) {\, \cdot\,} \overline{\p}_{k}(t) \, \di x \bigg) \di t   \nonumber
\\
&
\qquad \qquad  + \int_{\Om} \delta \nabla{z_{k}}{\, \cdot\,} \nabla{\varphi} + \frac{\varphi}{\delta} ( z_{k} (1 - z_{k})^{2} - z_{k}^{2} ( 1 - z_{k} )) \, \di x = 0 \,. \nonumber
\end{align}
Thus, condition~\eqref{e.continuous-optimality2} is obtained by passing to the limit in~\eqref{e.44} as $k\to \infty$ relying on the continuity of~$\ell'$, $d'$, $\C'$, and~$\ha'$, and on the convergences discussed above.

We conclude with~\eqref{e.continuous-optimality4}. First we notice that, thanks to~\eqref{e.discrete-equilibrium2},~\eqref{e.discrete-optimality4} can be equivalently expressed as
\begin{align}\label{e.commento}
\sum_{i=1}^{k} \int_{\Om} ( d^{2}(z_{k}) - | \brho^{k}_{i} |^{2} )  | \overline{\p}^{k}_{i} - \overline{\p}^{k}_{i+1} |\, \di x = 0 \,,
\end{align}
which,  owing to the fact that $|\brho^k_i| \leq d(z^k_i)$, implies
\begin{align}\label{e.equivalent}
\sum_{i=1}^{k} \int_{\Om} ( d^{2}(z_{k}) - | \brho^{k}_{i} |^{2} ) ( \overline{\p}^{k}_{i} - \overline{\p}^{k}_{i+1} )  \, \di x = 0\,.
\end{align}
Recalling that $\overline{\p}^{k}_{k+1}= 0$, we rewrite~\eqref{e.equivalent} as follows:
\begin{align}\label{e.equivalent2}
0 & =  \int_{\Om} (d^{2}(z_{k}) - |\brho^{k}_{1}|^{2} ) \, \overline{\p}^{k}_{1}\, \di x  - \sum_{i=2}^{k} \int_{\Om} (\brho^{k}_{i-1} + \brho^{k}_{i}) {\, \cdot\,} ( \brho^{k}_{i} - \brho^{k}_{i-1})\, \overline{\p}^{k}_{i} \, \di x  
\\
&
=  \int_{\Om} ( d^{2} ( z_{k}) - | \brho^{k}_{1} |^{2} ) \, \overline{\p}^{k}_{1} \, \di x - \int_{ t^{k}_{1}}^{T} \int_{\Om} ( \trho_{k}(t) + \trho_{k}(t - \tau_{k})) {\, \cdot\,} \dot{\brho}_{k}(t) \, \overline{\p}_{k}(t) \, \di x \, \di t \,. \nonumber
\end{align}
Since~$\overline{\p}^{k}_{1}$ is bounded in~$L^{2}(\Om; \M^{n}_{D})$, there exists~$\overline{\p}_{0} \in L^{2}(\Om; \M^{n}_{D})$ such that, up to a subsequence, $\overline{\p}_{1}^{k} \rightharpoonup \overline{\p}_{0}$ weakly in~$L^{2}(\Om; \M^{n}_{D})$. Since $\brho_{k} \to \brho$ in $H^{1}(0,T; L^{2}(\Om; \M^{n}_{D}))$, the function~$\trho_{k}(\cdot - \tau_{k})$ converges to~$\brho$ in $L^{2}(0,T; L^{2}(\Om; \M^{n}_{D}))$ and $\brho^{k}_{1} \to \brho(0) = 0$ in $L^{2}(\Om; \M^{n}_{D})$. As $\overline{\p}_{k} \xrightharpoonup{*} \overline{\p}$ weakly$^*$ in $L^{\infty}(0,T; L^{2}(\Om; \M^{n}_{D}))$ and~$\trho_{k}$ is bounded in $L^{\infty}(0,T; L^{\infty}(\Om; \M^{n}_{D}))$, we can pass to the limit in~\eqref{e.equivalent2} and deduce~\eqref{e.continuous-optimality4}
\end{proof}


\appendix

\section{Proof or Proposition~\ref{p.3}}\label{appendixA}

We start by recalling that by the definition of quasistatic evolution
(see Definition~\ref{d.quasistatic}), there exists $\brho \in
L^{\infty}((0,T)\times\Om; \M^{n}_{D}))$ such that $\brho(t) \in
d(z) \partial | \cdot | (\dot{\p} (t))$ almost everywhere (see, e.g.,~\cite{Han}), where the
symbol~$\partial$ denotes here the subdifferential, and such that
for $t \in [0,T]$ the equilibrium condition~\eqref{e.1} is equivalent to
\begin{align}\label{e.continuous-equilibrium}
\int_{\Om}  \C(z) \strain(t){\, \cdot\,} \eeta \, \di x & + \int_{\Om} \ha(z) \p(t) {\, \cdot\,} \qq \, \di x + \int_{\Om} \brho(t) {\, \cdot\,} \qq \, \di x
\\
& - \int_{\Om} \ell(z) f(t) {\, \cdot\,} v \, \di x - \int_{\Gamma_{N}} g(t) {\, \cdot\,} v \, \di \HH^{n-1} = 0 \nonumber
\end{align}
for every $(v, \eeta, \qq) \in \A(0)$. 

In the next lemma we prove that the piecewise affine functions defined in~\eqref{e.affine-interpolant} converge in $L^{\infty}(0,T ; H^{1}(\Om; \R^{n}) \times L^{2}(\Om; \M^{n}_{S}) \times L^{2}(\Om; \M^{n}_{D}))$ to a quasistatic evolution.

\begin{lemma}\label{l.A1}
Let~$z_{k}, z \in H^{1}(\Om; [0,1])$, $(u_{k}, \strain_{k}, \p_{k})$, and~$(u(\cdot), \strain(\cdot), \p(\cdot))$ be as in Proposition~\ref{p.3}. Then, $(u_{k}, \strain_{k}, \p_{k})$ converges to $(u(\cdot), \strain(\cdot), \p(\cdot))$ in $L^{\infty}(0,T ; H^{1}(\Om; \R^{n}) \times L^{2}(\Om; \M^{n}_{S}) \times L^{2}(\Om; \M^{n}_{D}))$.
\end{lemma}

\begin{proof}
We first show that $(u_{k}, \strain_{k}, \p_{k})$ is bounded in~$H^{1}(0,T; H^{1}(\Om; \R^{n}) \times L^{2}(\Om; \M^{n}_{S}) \times L^{2}(\Om;\M^{n}_{D}))$. We notice that the~$L^{\infty}$-boundedness is a consequence of the energy inequality~\eqref{e.4}.

By the uniform convexity of the functional~$\E_{k} (z_{k},
t^{k}_{i}, \cdot, \cdot, \cdot) + \D( z_{k},  \cdot - \p^{k}_{ i-1})$ and by testing the minimality of~$(u^{k}_{i}, \strain^{k}_{i}, \p^{k}_{i})$ at time $t^{k}_{i}$ with $(u^{k}_{i-1} + w^{k}_{i} - w^{k}_{i-1}, \strain^{k}_{i-1} + \e w^{k}_{i} - \e w^{k}_{i-1}, \p^{k}_{i-1}) \in \A(w^{k}_{i})$ and the minimality of~$(u^{k}_{i-1}, \strain^{k}_{i-1}, \p^{k}_{i-1})$ at time~$t^{k}_{i-1}$ with $(u^{k}_{i} - w^{k}_{i} + w^{k}_{i-1}, \strain^{k}_{i} - \e w^{k}_{i} + \e w^{k}_{i-1} , \p^{k}_{i}) \in \A(w^{k}_{i-1})$, we have that
\begin{align*}
& c \big( \| \strain^{k}_{i} - ( \strain^{k}_{i-1} - \e w^{k}_{i-1} + \e w^{k}_{i}) \|_{2}^{2} + \| \p^{k}_{i} - \p^{k}_{i-1} \|_{2}^{2} \big) 
\\
&
\qquad \leq  \int_{\Om} \C(z_{k}) (\e w^{k}_{i} - \e w^{k}_{i-1}) {\, \cdot\,} (\e w^{k}_{i} - \e w^{k}_{i-1}) \, \di x \nonumber
\\
&
\qquad \quad - \int_{\Om} \C(z_{k}) (\strain^{k}_{i} -  \strain^{k}_{i-1}) {\, \cdot\,} (\e w^{k}_{i} - \e w^{k}_{i-1}) \, \di x \nonumber
\\
&
\qquad \quad + \int_{\Om} \ell(z_{k}) (f^{k}_{i} - f^{k}_{i-1}){\, \cdot\,} (u^{k}_{i} - (u^{k}_{i-1} - w^{k}_{i-1} + w^{k}_{i})) \, \di x \nonumber
\\
&
\qquad \quad + \int_{\Gamma_{N}} ( g^{k}_{i} - g^{k}_{i-1}) {\, \cdot\,} (u^{k}_{i} - (u^{k}_{i-1} - w^{k}_{i-1} + w^{k}_{i}) ) \, \di \HH^{n-1}  \nonumber 
\\
&
\qquad \quad + \D(z_{k} , \p^{k}_{i} - \p^{k}_{i-2}) - \D(z_{k} , \p^{k}_{i-1} - \p^{k}_{i-2}) - \D (z_{k} , \p^{k}_{i} - \p^{k}_{i-1})  \nonumber
\\
&
\qquad \leq  C \big( \| \strain^{k}_{i} - ( \strain^{k}_{i-1} - \e w^{k}_{i-1} + \e w^{k}_{i}) \|_{2} + \| u^{k}_{i} - (u^{k}_{i-1} - w^{k}_{i-1} + w^{k}_{i}) \|_{H^{1}} \big) \big( \| f^{k}_{i} - f^{k}_{i-1} \|_{2} \nonumber 
\\
&
\qquad \quad + \| g^{k}_{i} - g^{k}_{i-1} \|_{2} + \| w^{k}_{i} - w^{k}_{i-1} \|_{H^{1}} \big) \,, \nonumber
\end{align*}
from which we deduce the bound in~$H^{1}(0,T; H^{1}(\Om; \R^{n}) \times L^{2}(\Om; \M^{n}_{S}) \times L^{2}(\Om;\M^{n}_{D}))$. In particular, this implies that $(u_{k} , \strain_{k}, \p_{k})$ converges to $(u, \strain, \p)$ weakly in~$H^{1}(0,T; H^{1}(\Om; \R^{n}) \times L^{2}(\Om; \M^{n}_{S}) \times L^{2}(\Om; \M^{n}_{D}))$, and the derivatives exist a.e.~in~$(0,T)$. 

Defining~$\brho^{k}_{i}:=  \Pi_{\M^{n}_{D}}\big(\C(z_{k}) \strain^{k}_{i} \big) - \ha(z_{k}) \p^{k}_{i}$, we have that the stability condition~\eqref{e.discrete-equilibrium1} holds. Setting
\begin{displaymath}
\brho_{k} (t) := \brho^{k}_{i} + \frac{(t - t^{k}_{i})}{\tau_{k}} \, ( \brho^{k}_{i+1} - \brho^{k}_{i}) \qquad \text{for $t \in [t^{k}_{i}, t^{k}_{i+1})$}
\end{displaymath}
we have that $\brho_{k}$ is bounded in $H^{1}(0,T; L^{2}(\Om;
\M^{n}_{D}))$ as well.

We proceed now by proving the uniform convergence. To this end, we need to introduce the piecewise constant interpolants
\begin{align*}
& \tu_{k}(t):= u^{k}_{i}\,,\qquad  \tstrain_{k}(t) := \strain^{k}_{i}\,, \qquad \tp_{k}(t):= \p^{k}_{i}\,, \qquad \trho_{k}(t) := \brho^{k}_{i},
 \\
 &
 \tilde{f}_{k}(t) := f^{k}_{i}\,, \qquad \tilde{g}_{k}(t) := g^{k}_{i}\,, \qquad \tilde{w}_{k}(t) := w^{k}_{i}
\end{align*}
for $t \in (t^{k}_{i-1}, t^{k}_{i}]$. In particular, $\| \tu_{k} (t) -
u_{k}(t) \|_{H^{1}} \leq \tau_{k} \| \dot{u}_{k}(t) \|_{H^1}$,
and similar inequalities hold for~$\tstrain_{k}$ and~$\tp_{k}$
for the $L^2$-norm. 

For a.e.~$t \in (t^{k}_{i-1}, t^{k}_{i}]$, we test the equilibrium conditions~\eqref{e.discrete-equilibrium1} and \eqref{e.continuous-equilibrium} with the triple
\begin{displaymath}
( \dot{u}(t) - \dot{u}_{k}(t) + \dot{w}_{k}(t) - \dot{w}(t) , \dot{\strain}(t) - \dot{\strain}_{k}(t) + \e \dot{w}_{k}(t) - \e \dot{w}(t) , \dot{\p}(t) - \dot{\p}_{k}(t)) \in \A(0)
\end{displaymath}
and we subtract one from the other, obtaining
\begin{align}\label{e.A1}
& \int_{\Om} \C(z_{k}) \tstrain_{k}(t) {\, \cdot\,} ( \dot{\strain}(t) - \dot{\strain}_{k}(t) + \e \dot{w}_{k}(t) - \e \dot{w}(t))  \, \di x 
\\
&
\qquad -  \int_{\Om} \C(z) \strain (t) {\, \cdot\,} ( \dot{\strain}(t) - \dot{\strain}_{k}(t) + \e \dot{w}_{k}(t) - \e \dot{w}(t))  \, \di x \nonumber
\\
&
\qquad + \int_{\Om} \ha(z_{k}) \tp_{k}(t) {\, \cdot\,} ( \dot{\p}(t) - \dot{\p}_{k}(t) ) \, \di x -  \int_{\Om} \ha(z) \p(t) {\, \cdot\,} ( \dot{\p}(t) - \dot{\p}_{k}(t) ) \, \di x \nonumber
\\
&
\qquad + \int_{\Om} \trho_{k}(t) {\, \cdot\,} (  \dot{\p}(t) - \dot{\p}_{k}(t) ) \, \di x - \int_{\Om} \brho(t) {\, \cdot\,} (  \dot{\p}(t) - \dot{\p}_{k}(t) ) \, \di x \nonumber
\\
&
\qquad - \int_{\Om} \ell(z_{k}) \tilde{f}_{k}(t) {\, \cdot\,} (\dot{u}(t) - \dot{u}_{k}(t) + \dot{w}_{k}(t) - \dot{w}(t)) \, \di x  \nonumber
\\
& 
\qquad + \int_{\Om} \ell(z) f(t) {\, \cdot\,}  (\dot{u}(t) - \dot{u}_{k}(t) + \dot{w}_{k}(t) - \dot{w}(t)) \, \di x  \nonumber
\\
&
\qquad - \int_{\Gamma_{N}} (\tilde{g}_{k}(t) - g(t)) {\, \cdot\,} (\dot{u}(t) - \dot{u}_{k}(t) + \dot{w}_{k}(t) - \dot{w}(t)) \, \di \HH^{n-1} = 0 \,. \nonumber
\end{align}
We notice that, being $\brho(t) \in d(z) \partial | \cdot |
(\dot{\p}(t))$ and $\trho_{k}(t) \in d(z_{k}) \partial | \cdot |
(\dot{\p}_{k}(t))$ almost everywhere in $\Om$, it holds
\begin{equation}\label{e.A2}
\int_{\Om} \bigg( \frac{d(z_{k})}{d(z)} \,  \brho(t) - \trho_{k}(t) \bigg) \cdot ( \dot{\p}(t) - \dot{\p}_{k}(t))  \, \di x \geq 0\,.
\end{equation}
Hence, adding and subtracting in~\eqref{e.A1} the terms
\begin{align*}
& \int_{\Om} \C(z_{k}) ( \strain_{k}(t)  + \strain(t)) {\, \cdot\,} ( \dot{\strain}(t) - \dot{\strain}_{k}(t) + \e \dot{w}_{k}(t) - \e \dot{w}(t)) \, \di x \,,
\\
&
\int_{\Om} \ha(z_{k}) (\p_{k}(t) + \p(t)) {\, \cdot\,} ( \dot{\p}(t) - \dot{\p}_{k}(t))\, \di x \,,
\\
&
\int_{\Om} \frac{d(z_{k})}{d(z)} 	\, \brho(t) {\, \cdot\,} ( \dot{\p}(t) - \dot{\p}_{k}(t) ) \,\di x \,,
\end{align*}
and using~\eqref{e.A2}, we obtain, after a simple algebraic manipulation,
\begin{align}\label{e.A3}
& \int_{\Om} \C(z_{k}) ( \strain(t) - \strain_{k} (t)) {\, \cdot\,} ( \dot{\strain}(t) - \dot{\strain}_{k}(t)) \, \di x + \int_{\Om} \ha(z_{k}) ( \p(t) - \p_{k}(t)) {\, \cdot\,} ( \dot{\p}(t) - \dot{\p}_{k}(t)) \, \di x 
\\
&
\leq \int_{\Om} \C(z_{k}) ( \strain_{k}(t) - \strain ( t) )  {\, \cdot\,} ( \e \dot{w}_{k}(t) - \e \dot{w}(t) ) \, \di x   \nonumber
\\
&
\qquad + \int_{\Om} \C(z_{k} ) ( \tstrain_{k}(t) - \strain_{k}(t)) {\, \cdot\,} ( \dot{\strain}(t) - \dot{\strain}_{k}(t) + \e \dot{w}_{k}(t) - \e \dot{w}(t)) \, \di x \nonumber
\\
&
\qquad + \int_{\Om}  (\C(z_{k}) - \C(z)) \strain(t) {\, \cdot \,} ( \dot{\strain}(t) - \dot{\strain}_{k}(t) + \e \dot{w}_{k}(t) - \e \dot{w}(t)) \, \di x \nonumber
\\
&
\qquad + \int_{\Om} \ha(z_{k} ) ( \tp_{k}(t) - \p_{k}(t)) {\, \cdot\,} ( \dot{\p}(t) - \dot{\p}_{k}(t) ) \, \di x \nonumber
\\
&
\qquad + \int_{\Om}  (\ha(z_{k}) - \ha(z)) \p(t) {\, \cdot \,} ( \dot{\p}(t) - \dot{\p}_{k}(t) ) \, \di x \nonumber
\\
&
\qquad + \int_{\Om} \frac{d(z_{k}) - d(z)}{d(z)} \, \brho(t) {\, \cdot\,} ( \dot{\p}(t) - \dot{\p}_{k}(t)) \, \di x  \nonumber
\\
&
\qquad -  \int_{\Om} ( \ell(z_{k}) \tilde{f}_{k}(t) - \ell(z) f(t)) {\, \cdot\,} (\dot{u}(t) - \dot{u}_{k}(t) + \dot{w}_{k}(t) - \dot{w}(t)) \, \di x \nonumber
\\
&
\qquad - \int_{\Gamma_{N}} ( \tilde{g}_{k}(t) - g(t)) {\, \cdot\,} (\dot{u}(t) - \dot{u}_{k}(t) + \dot{w}_{k}(t) - \dot{w}(t)) \, \di \HH^{n-1} \nonumber\,.
\end{align}
Integrating~\eqref{e.A3} w.r.t.~$t$ on the interval $[0,s]$, for $s \in [0,T]$, recalling~\eqref{e.hp-bound}--\eqref{e.H} and that $u_{k}(0) = u(0) = 0$ and $\strain_{k}(0) = \strain(0) = \p_{k}(0) = \p(0) = 0$, we further estimate
\begin{align}\label{e.A4}
& \frac{\alpha_{\C}}{2} \| \strain(s)  - \strain_{k}(s) \|_{2}^{2}  + \frac{\alpha_{\ha}}{2} \| \p(s)  - \p_{k} (s) \|_{2}^{2}
\\
&
\qquad \leq \beta_{\C} \Big( \| \strain - \strain_{k}\|_{L^{2}(0,T; L^{2}(\Om; \M^{n}_{S}))} + \tau_{k} \| \dot{\strain}_{k} \|_{L^{2}(0,T; L^{2}(\Om; \M^{n}_{S}))}  \nonumber
\\
&
\qquad \qquad +  \| ( \C ( z_{k} ) - \C ( z ) ) \strain \|_{L^{2}(0,T; L^{2}(\Om; \M^{n}_{S}))} \Big) \| w_{k} - w \|_{H^{1}(0,T; H^{1}(\Om; \R^{n}))} \nonumber
\\
&
\qquad \qquad + \Big( \beta_{\C} \tau_{k} \| \dot{\strain}_{k} \|_{L^{2} ( [ 0 , T ] ; L^{2} ( \Om ; \M^{n}_{S} ) ) } \nonumber
\\
&
\qquad \qquad  + \| ( \C ( z_{k} ) - \C ( z ) ) \strain \|_{L^{2}(0,T; L^{2}(\Om; \M^{n}_{S}))} \Big) \| \dot{\strain} - \dot{\strain}_{k}\|_{L^{2}(0,T; L^{2}(\Om; \M^{n}_{S}))} \nonumber
\\
&
\qquad \qquad + \beta_{\ha} \Big( \tau_{k} \| \dot{\p}_{k} \|_{L^{2}(0,T; L^{2}(\Om; \M^{n}_{D}))} + \| (\ha(z_{k}) - \ha(z)) \p \|_{L^{2}(0,T; L^{2}(\Om; \M^{n}_{D}))}  \nonumber
\\
&
\qquad \qquad + \frac{1}{ \alpha } \| (d(z_{k}) - d(z)) \brho \|_{L^{2}(0,T; L^{2}(\Om; \M^{n}_{D}))} \Big) \| \dot{\p} - \dot{\p}_{k}\|_{L^{2}(0,T; L^{2}(\Om; \M^{n}_{D}))} \nonumber
\\
&
\qquad \qquad + C \Big( \| \ell(z_{k})\tilde{f}_{k} - \ell(z) f \|_{L^{2}(0,T; L^{2}(\Om; \R^{n}))} \nonumber
\\
&
\qquad \qquad  + \| \tilde{g}_{k} - g \|_{L^{2}(0,T; L^{2}(\Gamma_{N}; \R^{n}))} \Big) \| u - u_{k} + w_{k} - w \|_{H^{1}(0,T; H^{1}(\Om; \R^{n}))}\,, \nonumber
\end{align}
for some positive constant~$C$ independent of~$k$. Since $(u_{k},
\strain_{k}, \p_{k})$ is bounded in $H^{1}(0,T; H^{1}(\Om; \R^{n})
\times L^{2}(\Om; \M^{n}_{S}) \times L^{2}(\Om; \M^{n}_{D}))$,
$\tilde{f}_{k} \to f$ in $L^{2}(0,T; L^{2}(\Om; \R^{n}))$,
$\tilde{g}_{k} \to g$ in $L^{2}(0,T; L^{2}(\Gamma_{N}; \R^{n}))$, and
$z_{k} \rightharpoonup z$ in $H^{1}(\Om)$ with $0 \leq z_{k}, z \leq
1$ almost everywhere, we deduce from~\eqref{e.A4} that $\strain_{k} \to \strain$ in $L^{\infty}(0,T; L^{2}(\Om; \M^{n}_{S}))$ and $\p_{k} \to \p$ in $L^{\infty}(0,T; L^{2}(\Om; \M^{n}_{D}))$. By Korn's inequality and by the convergence of~$w_{k}$ to $w$ in $H^{1}(0,T; H^{1}(\Om; \R^{n}))$, we infer that $u_{k} \to u$ in $L^{\infty}(0,T; H^{1}(\Om; \R^{n}))$. This concludes the proof of the lemma.
\end{proof}

We are now in a position to conclude the proof of Proposition~\ref{p.3}. We follow here the lines of~\cite[Theorem~3.3]{Wachsmut1}.

\begin{proof}[Proof of Proposition~\ref{p.3}]
In view of Lemma~\ref{l.A1}, it remains to show that $(\dot{u}_{k}, \dot{\strain}_{k}, \dot{\p}_{k})$ converges to~$(\dot{u}, \dot{\strain}, \dot{\p})$ in $L^{2}(0,T; H^{1}(\Om; \R^{n}) \times L^{2}(\Om; \M^{n}_{S}) \times L^{2}(\Om; \M^{n}_{D}))$. To this end, we define the auxiliary triples
\begin{align}
&( \om^{k}_{i}, \bxi^{k}_{i}, \btheta^{k}_{i}) = \argmin \bigg\{  \frac{1}{2} \int_{\Om} \C(z_{k}) \Big( \strain + \frac{\e w^{k}_{i} - \e w^{k}_{i-1}}{\tau_{k}} \Big)  {\, \cdot\,} \Big( \strain + \frac{\e w^{k}_{i} - \e w^{k}_{i-1}}{\tau_{k}} \Big) \, \di x  \label{e.A5}
\\
&
\qquad \qquad\qquad \qquad  \qquad + \frac12 \int_{\om} \ha(z_{k}) \p {\, \cdot\,} \p \, \di x - \int_{\Om} \ell(z_{k}) \frac{f^{k}_{i} - f^{k}_{i-1}}{\tau_{k}} {\, \cdot\,} u \, \di x \nonumber
\\
&
\qquad\qquad\qquad \qquad  \qquad - \int_{\Gamma_{N}} \frac{g^{k}_{i} - g^{k}_{i-1}}{\tau_{k}} {\, \cdot\,} u\, \di \HH^{n-1} : 
(u, \strain, \p) \in \A (0 ) \bigg\} \nonumber
\\[2mm]
&
(\om(t), \bxi(t), \btheta(t)) = \argmin \bigg\{  \frac{1}{2} \int_{\Om} \C(z) ( \strain + \e \dot{w}(t) ) {\, \cdot\,} ( \strain + \e \dot{w}(t) )  \, \di x + \frac12 \int_{\om} \ha(z) \p {\, \cdot\,} \p \, \di x  \label{e.A6}
\\
&
\qquad  \qquad \qquad \qquad \qquad \qquad - \int_{\Om} \ell(z) \dot{f}(t) {\, \cdot\,} u \, \di x - \int_{\Gamma_{N}} \dot{g}(t) {\, \cdot\,} u\, \di \HH^{n-1} : (u, \strain, \p) \in \A ( 0) \bigg\} \,. \nonumber
\end{align}
Since $(f_{k}, g_{k}, w_{k})$ converges to $(f, g, w)$ in $H^{1}(0,T; H^{1}(\Om; \R^{n}) \times L^{2}(\Gamma_{N}; \R^{n}) \times H^{1}(\Om; \R^{n}))$, we deduce that the piecewise constant function
\begin{displaymath}
(\om_{k}(t) , \bxi_{k}(t), \btheta_{k}(t)) := ( \om^{k}_{i} , \bxi^{k}_{i}, \btheta^{k}_{i}) \qquad \text{for $t \in (t^{k}_{i-1}, t^{k}_{i}]$}
\end{displaymath}
converges to $(\om, \bxi, \btheta)$ in $L^{2}(0,T; H^{1} (\Om; \R^{n}) \times L^{2}(\Om; \M^{n}_{S})\times L^{2}(\Om; \M^{n}_{D}))$.

By the minimality of~$(u^{k}_{i}, \strain^{k}_{i}, \p^{k}_{i})$ in~\eqref{e.3}, we have that
\begin{align}\label{e.A15}
\int_{\Om} \C(z_{k}) & \strain^{k}_{i} {\, \cdot\,} \eeta \, \di x + \int_{\Om} \ha(z_{k}) \p^{k}_{i} {\, \cdot\,} \qq \, \di x + \int_{\Om} d(z_{k}) | \qq - ( \p^{k}_{i-1} - \p^{k}_{i}) | \, \di x 
\\
&
- \int_{\Om} d(z_{k}) | \p^{k}_{i} - \p^{k}_{i-1} | \, \di x - \int_{\Om} \ell(z_{k}) f^{k}_{i} {\, \cdot\,} v \, \di x - \int_{\Gamma_{N}} g^{k}_{i} {\, \cdot\,} v \, \di \HH^{n-1} \geq 0 \nonumber
\end{align} 
for every $(v, \eeta, \qq) \in \A(0)$. Testing~\eqref{e.A15} with the triple
\begin{displaymath}
( u^{k}_{i-1} - u^{k}_{i} - w^{k}_{i-1} + w^{k}_{i}, \strain^{k}_{i-1} - \strain^{k}_{i} - \e w^{k}_{i-1} + \e w^{k}_{i} , \p^{k}_{i-1} - \p^{k}_{i}) \in \A(0)
\end{displaymath}
combined with the equilibrium condition (at time $t^{k}_{i-1}$)
\begin{align*}
- \int_{\Om} d(z_{k}) | \p^{k}_{i} - \p^{k}_{i-1} | \, \di x \leq &  \int_{\Om} \C(z_{k}) \strain^{k}_{i-1} {\, \cdot \,} ( \strain^{k}_{i} - \strain^{k}_{i-1} - \e w^{k}_{i} + \e w^{k}_{i-1}) \, \di x 
\\
&
+ \int_{\Om} \ha(z_{k}) \p^{k}_{i-1} {\, \cdot\,} ( \p^{k}_{i} - \p^{k}_{i-1}) \, \di x 
\\
&
- \int_{\Om} \ell(z_{k}) f^{k}_{i-1} {\, \cdot\,} (u^{k}_{i} - u^{k}_{i-1} - w^{k}_{i} + w^{k}_{i-1}) \, \di x 
\\
&
- \int_{\Gamma_{N}} g^{k}_{i-1} {\, \cdot\,}   (u^{k}_{i} - u^{k}_{i-1} - w^{k}_{i} + w^{k}_{i-1}) \, \di \HH^{n-1}\,,
\end{align*}
we deduce that
\begin{align} \label{e.A7}
\int_{\Om}  \C(z_{k})  & \dot{\strain}_{k}(t)  {\, \cdot\,} ( \dot{\strain}_{k}(t) -  \e \dot{w}_{k}(t) ) \, \di x +  \int_{\Om} \ha(z_{k})  \dot{\p}_{k} (t) {\, \cdot\,}  \dot{\p}_{k}(t) \, \di x 
\\
&
\quad - \int_{\Om} \ell(z_{k}) \dot{f}_{k}(t) {\, \cdot\,} ( \dot{u}_{k}(t) - \dot{w}_{k}(t) ) \, \di x - \int_{\Gamma_{N}} \dot{g}_{k}(t) {\, \cdot\,} ( \dot{u}_{k}(t) - \dot{w}_{k}(t)) \, \di \HH^{n-1} \leq 0 \,. \nonumber
\end{align}
Testing the Euler-Lagrange equation of~\eqref{e.A5} with the test $(\dot{u}_{k}(t) - \dot{w}_{k}(t), \dot{\strain}_{k} (t) - \e \dot{w}_{k}(t), \dot{\p}_{k}(t)) \in \A(0)$ we also get
\begin{align}\label{e.A8}
\int_{\Om} \C(z_{k}) & ( \bxi_{k}(t) + \e \dot{w}_{k}(t) ) {\, \cdot\,} (\dot{\strain}_{k}(t) -  \e \dot{w}_{k}(t) )\, \di x + \int_{\Om} \ha(z_{k}) \btheta_{k}(t){\, \cdot\,} \dot{\p}_{k}(t) \, \di x 
\\
&
- \int_{\Om} \ell(z_{k}) \dot{f}_{k}(t) {\, \cdot\,} ( \dot{u}_{k}(t) - \dot{w}_{k}(t) ) \, \di x - \int_{\Gamma_{N}} \dot{g}_{k}(t) {\, \cdot\,} ( \dot{u}_{k}(t) - \dot{w}_{k}(t) ) \, \di \HH^{n-1} = 0\,. \nonumber
\end{align}
We subtract~\eqref{e.A8} from~\eqref{e.A7} and obtain the inequality
\begin{align*}
\int_{\Om} &  \C(z_{k}) \big ( ( \dot{\strain}_{k}(t) - \e
             \dot{w}_{k}(t))  - \bxi_{k}(t) \big) {\,
             \cdot\,}(\dot{\strain}_{k}(t) -  \e \dot{w}_{k}(t) )\,
             \di x \\
&\quad+ \int_{\Om} \ha(z_{k}) ( \dot{\p}_{k}(t) - \btheta_{k}(t)) {\, \cdot\,} \dot{\p}_{k}(t) \, \di x \leq 0\,,
\end{align*}
which in turn implies
\begin{align}\label{e.A9}
\int_{\Om}   \C(z_{k}) \big ( \bxi_{k}(t) & - 2 ( \dot{\strain}_{k}(t) - \e \dot{w}_{k}(t)) \big) {\,\cdot\,} \big( \bxi_{k}(t) -  2 ( \dot{\strain}_{k}(t) - \e \dot{w}_{k}(t)) \big) \, \di x 
\\
&
+ \int_{\Om} \ha(z_{k}) ( \btheta_{k}(t) - 2\dot{\p}_{k}(t)){\, \cdot\,}  ( \btheta_{k}(t) - 2\dot{\p}_{k}(t)) \, \di x  \nonumber
\\
&
 \qquad \leq \int_{\Om} \C(z_{k})  \bxi_{k}(t) {\, \cdot\,}  \bxi_{k}(t) \, \di x + \int_{\Om} \ha(z_{k}) \btheta_{k}(t) {\, \cdot\,} \btheta_{k}(t) \, \di x \,. \nonumber
\end{align}

By the equilibrium condition~\eqref{e.1} of $(u(t), \strain(t), \p(t))$ and by the energy balance~\eqref{e.2}, we have that for a.e.~$t \in [0,T]$
\begin{align}\label{e.A10}
\int_{\Om} \C(z) &  \strain(t) {\, \cdot\,} ( \dot{\strain}(t) - \e \dot{w}(t) ) \, \di x + \int_{\Om} \ha(z) \p(t) {\, \cdot\,}  \dot{\p}(t) \, \di x + \int_{\Om} d(z) \brho(t) {\, \cdot\,} \dot{\p}(t) \, \di x 
\\
&
- \int_{\Om} \ell(z) f(t) {\, \cdot\,} ( \dot{u}(t) - \dot{w}(t)) \, \di x - \int_{\Gamma_{N}} g(t) {\, \cdot\,} ( \dot{u}(t) - \dot{w}(t)) \, \di x =  0\,. \nonumber
\end{align}
Since $\brho(t) \cdot \dot{\p}(t) = d(z) | \dot{\p}(t)|$ almost
everywhere in~$\Om$ and, by the equilibrium~\eqref{e.1} at time~$t+h$, 
\begin{align*}
 \int_{\Om} d(z) | \dot{\p}(t) | \, \di x \geq&  \int_{\Om} \C(z) \strain(t+h){\, \cdot\,} ( \e \dot{w}(t) - \dot{\strain}(t) ) \, \di x - \int_{\Om} \ha(z) \p(t+h) {\, \cdot\,}  \dot{\p}(t) \, \di x
\\
&
+ \int_{\Om} \ell(z) f(t+h) {\, \cdot\,} ( \dot{u}(t) - \dot{w}(t)) \, \di x + \int_{\Gamma_{N}} g(t+h) {\, \cdot\,} ( \dot{u}(t) - \dot{w}(t)) \, \di x \,,
\end{align*}
we infer from~\eqref{e.A10} that for $h \in \R\setminus\{0\}$,
\begin{align}\label{e.A11}
 \int_{\Om}   \C(z)  ( \strain(t+h) &  -  \strain(t) ){\, \cdot\,} ( \dot{\strain}(t) - \e \dot{w}(t) ) \, \di x + \int_{\Om} \ha(z) ( \p(t+h) - \p(t) ) {\, \cdot\,}  \dot{\p}(t) \, \di x
\\
&
- \int_{\Om} \ell(z) ( f(t+h) - f(t) )  {\, \cdot\,} ( \dot{u}(t) - \dot{w}(t)) \, \di x  \nonumber 
\\
&
- \int_{\Gamma_{N}} ( g(t+h) - g(t) )  {\, \cdot\,} ( \dot{u}(t) - \dot{w}(t)) \, \di x \geq  0\,. \nonumber
\end{align}
Dividing~\eqref{e.A11} by $h$ (positive or negative) and passing to the limit as~$h \to 0$, we deduce that for a.e.~$t \in [0,T]$ there holds
\begin{align}\label{e.A12}
 \int_{\Om}   \C(z)  \dot{\strain}(t) & {\, \cdot\,} ( \dot{\strain}(t) - \e \dot{w}(t) ) \, \di x + \int_{\Om} \ha(z) \dot \p(t) {\, \cdot\,}  \dot{\p}(t) \, \di x
\\
&
- \int_{\Om} \ell(z)  \dot{f}(t)  {\, \cdot\,} ( \dot{u}(t) - \dot{w}(t)) \, \di x 
- \int_{\Gamma_{N}}  \dot{g}(t)   {\, \cdot\,} ( \dot{u}(t) - \dot{w}(t)) \, \di x =  0\,. \nonumber
\end{align}
Testing the Euler-Lagrange equation relative to~\eqref{e.A6} with the triple $(\dot{u}(t) - \dot{w}(t), \dot{\strain}(t) - \e \dot{w}(t) , \dot{\p}(t)) \in \A(0)$ we get
\begin{align}\label{e.A13}
\int_{\Om} \C(z) &  (  \bxi(t) + \e \dot{w}(t) )   {\, \cdot\,} (\dot{\strain}(t)  - \e \dot{w}(t)) \, \di x + \int_{\Om} \ha(z) \btheta(t) {\, \cdot\,} \dot{\p}(t) \, \di x 
\\
&
= \int_{\Om} \ell(z) \dot{f}(t) {\, \cdot\,} ( \dot{u}(t) - \dot{w}(t)) \, \di x + \int_{\Gamma_{N}} \dot{g}(t) {\, \cdot\,} ( \dot{u}(t) - \dot{w}(t)) \, \di \HH^{n-1}\,. \nonumber
\end{align}
Subtracting~\eqref{e.A13} from \eqref{e.A12} and arguing as in~\eqref{e.A9} we finally obtain that
\begin{align}\label{e.A14}
 \int_{\Om}    \C(z) &  \big( \bxi(t)  - 2 (  \dot{\strain}(t)) - \e \dot{w}(t) ) \big)  {\,\cdot\,} ( \bxi(t) - 2 ( \dot{\strain}(t)) - \e \dot{w}(t) ) \big)  \, \di x 
\\
&
 + \int_{\Om} \ha(z) ( \btheta(t) - 2\dot{\p}(t)){\, \cdot\,}  ( \btheta(t) - 2\dot{\p}(t)) \, \di x   \nonumber
\\
&
 \qquad = \int_{\Om} \C(z)  \bxi(t) {\, \cdot\,}  \bxi(t) \, \di x + \int_{\Om} \ha(z) \btheta(t) {\, \cdot\,} \btheta(t) \, \di x \,. \nonumber
\end{align}

Let us now set 
\begin{align*}
& r_{k}(t) := \int_{\Om} \C(z_{k})  \bxi_{k}(t) {\, \cdot\,}  \bxi_{k}(t) \, \di x + \int_{\Om} \ha(z_{k}) \btheta_{k}(t) {\, \cdot\,} \btheta_{k}(t) \, \di x \,,
\\
&
r(t) : = \int_{\Om} \C(z)  \bxi(t) {\, \cdot\,}  \bxi(t) \, \di x + \int_{\Om} \ha(z) \btheta(t) {\, \cdot\,} \btheta(t) \, \di x \,.
\end{align*}
By the convergence of $$(\om_{k}, \bxi_{k}, \btheta_{k}) \to ( \om,
\bxi, \btheta)\ \ \text{in} \ \ L^{2}(0,T; H^{1}(\Om; \R^{n}) \times L^{2}(\Om; \M^{n}_{S}) \times L^{2}(\Om; \M^{n}_{D}))$$ we have that $r_{k} \to r$ in $L^{1}(0,T)$. In view of~\eqref{e.A9} and~\eqref{e.A14} we estimate
\begin{align}\label{e.A19}
& \limsup_{k\to\infty} \bigg( \int_{0}^{T}  \int_{\Om} \Big( \C( z_{k} ) \big( \bxi_{k}(t) - 2( \dot{\strain}_{k} (t) - \e \dot{w}_{k} (t)) \big) - \C(z) \big( \bxi(t) 
\\
&
\quad - 2( \dot{\strain}(t) - \e \dot{w}(t)) \big) \Big) \cdot
\big( \bxi_{k}(t) - 2( \dot{\strain}_{k}(t) - \e \dot{w}_{k}(t)) - \bxi(t) + 2 ( \dot{\strain}(t) - \e \dot{w}(t)) \big) \, \di x \, \di t \nonumber
\\
&
\quad + \int_{0}^{T} \int_{\Om} \Big(\ha(z_{k}) ( \btheta_{k}(t) - 2 \dot{\p}_{k}(t) ) - \ha(z) ( \btheta(t) - 2 \dot{\p}(t) )\Big)  \cdot \nonumber
\\
&
\quad \big( \btheta_{k}(t) - 2 \dot{\p}_{k}(t) - \btheta (t) + 2  \dot{\p}(t)  \big) \, \di x \, \di t \bigg) \nonumber
\\
&
= \limsup_{k\to\infty} \bigg(  \int_{0}^{T} \int_{\Om}  \C( z_{k} )
  \big( \bxi_{k}(t) - 2( \dot{\strain}_{k} (t) - \e \dot{w}_{k} (t))
  \big){\, \cdot\,}  \nonumber
\\
&
\quad {\, \cdot\,} \big( \bxi_{k}(t) - 2( \dot{\strain}_{k}(t) - \e \dot{w}_{k}(t)) \big) \, \di x \, \di t  + \int_{0}^{T} \int_{\Om} \ha(z_{k}) ( \btheta_{k}(t) - 2 \dot{\p}_{k}(t) ) {\, \cdot\,} ( \btheta_{k}(t) - 2 \dot{\p}_{k}(t) ) \, \di x \, \di t \bigg)  \nonumber
\\
&
\quad -  \int_{0}^{T} \int_{\Om}  \C( z ) \big( \bxi(t) - 2( \dot{\strain} (t) - \e \dot{w} (t)) \big){\, \cdot\,} \big( \bxi (t) - 2( \dot{\strain} (t) - \e \dot{w} (t)) \big) \, \di x \, \di t \nonumber
\\
&
\quad - \int_{0}^{T} \int_{\Om} \ha(z) ( \btheta (t) - 2 \dot{\p} (t) ) {\, \cdot\,} ( \btheta (t) - 2 \dot{\p} (t) ) \, \di x \, \di t \nonumber
\\
&
\leq \limsup_{k \to \infty} \int_{0}^{T} (r_{k}(t) - r(t) ) \, \di t = 0\,. \nonumber
\end{align}
Using~\eqref{e.A19}, we further estimate
\begin{align}\label{e.A20}
& \limsup_{k\to \infty} \bigg( \int_{0}^{T} \int_{\Om} \C(z_{k}) \big( \bxi_{k}(t)  - 2 ( \dot{\strain}_{k}(t) - \e \dot{w}_{k}(t)) - \bxi(t) + 2 ( \dot{\strain}(t) - \e \dot{w}(t) ) \big) {\,\cdot\,}
\\
&
\quad {\,\cdot\,} \big( \bxi_{k}(t) -  2 ( \dot{\strain}_{k}(t) - \e \dot{w}_{k}(t))  - \bxi(t) + 2 ( \dot{\strain}(t) - \e \dot{w}(t) ) \big) \, \di x \, \di t  \nonumber
\\
&
\quad + \int_{0}^{T} \int_{\Om} \ha(z_{k}) ( \btheta_{k}(t) - 2 \dot{\p}_{k}(t) - \btheta(t) + 2 \dot{\p} (t)) {\, \cdot\,} ( \btheta_{k}(t) - 2 \dot{\p}_{k}(t) \nonumber
\\
&
\quad \vphantom{\int}  - \btheta(t) + 2 \dot{\p} (t)) \, \di x \, \di t  \bigg) \nonumber
\\
&
= \limsup_{k \to \infty} \bigg( \int_{0}^{T} \int_{\Om} \C(z_{k}) \big( \bxi_{k}(t)  - 2 ( \dot{\strain}_{k}(t) - \e \dot{w}_{k}(t)) \big) {\, \cdot\,} \big( \bxi_{k}(t)  - 2 ( \dot{\strain}_{k}(t) - \e \dot{w}_{k}(t)) \nonumber
\\
&
\quad \vphantom{\int} - \bxi(t) + 2 ( \dot{\strain}(t) - \e \dot{w}(t) ) \big)\, \di x \, \di t  \nonumber
\\
&
\quad - \int_{0}^{T}\int_{\Om} \C(z) \big( \bxi (t)  - 2 ( \dot{\strain} (t) - \e \dot{w} (t)) \big) {\, \cdot\,} \big( \bxi_{k}(t)  - 2 ( \dot{\strain}_{k}(t) - \e \dot{w}_{k}(t)) \nonumber
\\
&
\quad \vphantom{\int} - \bxi(t) + 2 ( \dot{\strain}(t) - \e \dot{w}(t) ) \big)\, \di x \, \di t \nonumber
\\
&
\quad + \int_{0}^{T}\int_{\Om} \big( \C(z) - \C(z_{k}) \big) \big( \bxi (t)  - 2 ( \dot{\strain} (t) - \e \dot{w} (t)) \big) {\, \cdot\,} \big( \bxi_{k}(t)  - 2 ( \dot{\strain}_{k}(t) - \e \dot{w}_{k}(t)) \nonumber
\\
&
\quad \vphantom{\int} - \bxi(t) + 2 ( \dot{\strain}(t) - \e \dot{w}(t) ) \big)\, \di x \, \di t  \nonumber
\\
&
\quad + \int_{0}^{T}\int_{\Om} \ha(z_{k}) ( \btheta_{k}(t) - 2\dot{\p}_{k}(t)) {\, \cdot\,} ( \btheta_{k}(t) - 2\dot{\p}_{k}(t) - \btheta(t) + 2\dot{\p}(t)) \, \di x  \nonumber
\\
&
\quad - \int_{0}^{T}\int_{\Om} \ha(z) ( \btheta(t) - 2\dot{\p}(t)) {\, \cdot\,} ( \btheta_{k}(t) - 2\dot{\p}_{k}(t) - \btheta(t) + 2\dot{\p}(t)) \, \di x \, \di t  \nonumber
\\
&
\quad +\int_{0}^{T} \int_{\Om} \big( \ha(z) - \ha(z_{k}) \big)   ( \btheta(t) - 2\dot{\p}(t)) {\, \cdot\,} ( \btheta_{k}(t) - 2\dot{\p}_{k}(t) - \btheta(t) + 2\dot{\p}(t)) \, \di x \, \di t \bigg) \nonumber
\\
&
\leq \limsup_{k \to \infty} \bigg(  \int_{0}^{T}\int_{\Om} \big( \C(z) - \C(z_{k}) \big) \big( \bxi (t)  - 2 ( \dot{\strain} (t) - \e \dot{w} (t)) \big) {\, \cdot\,} \big( \bxi_{k}(t)  \nonumber
\\
&
\quad \vphantom{\int} - 2 ( \dot{\strain}_{k}(t) - \e \dot{w}_{k}(t)) \big)  \, \di x \, \di t \nonumber
\\
&
\quad + \int_{0}^{T} \int_{\Om} \big( \ha(z) - \ha(z_{k}) \big)   ( \btheta(t) - 2\dot{\p}(t)) {\, \cdot\,} ( \btheta_{k}(t) - 2\dot{\p}_{k}(t) - \btheta(t) + 2\dot{\p}(t)) \, \di x \, \di t \bigg) = 0\,.\nonumber
\end{align}
From~\eqref{e.C}--\eqref{e.H} and~\eqref{e.A20} we infer that
\begin{align*}
& \bxi_{k} - 2 ( \dot{\strain}_{k} - \e \dot{w}_{k})  \to  \bxi - 2(  \dot{\strain} - \e \dot{w})  \qquad \text{and} \qquad  \btheta_{k} - 2 \dot{\p}_{k}  \to  \btheta - 2 \dot{\p}
\end{align*}
in $L^{2}(0,T;  L^{2}(\Om; \M^{n}))$. Since $\bxi_{k} \to \bxi$ and $\btheta_{k} \to \btheta$ in $L^{2}(0,T; L^{2}(\Om; \M^{n}))$, we immediately deduce that $\dot{\strain}_{k} \to \dot{\strain}$ and $\dot{\p}_{k} \to \dot{\p}$ in $L^{2}(0,T; L^{2}(\Om; \M^{n}))$. Finally, the convergence of~$\dot{u}_{k}$ to $\dot{u}$ in $L^{2}(0,T; H^{1}(\Om; \R^{n}))$ is a consequence of the convergences of~$\dot{\strain}_{k}$, $\dot{\p}_{k}$, and~$\dot{w}_{k}$, and of Korn's inequality. This concludes the proof of Proposition~\ref{p.3}.
\end{proof}

\section*{ Acknowledgements}
 This work is partially supported 
 by the Austian Science Fund (FWF) projects F\,65, W\,1245, I\,4354, I\,5149, and P\,32788 and by the OeAD-WTZ project CZ 01/2021.

\bibliographystyle{siam}

\begin{thebibliography}{10}

  
\bibitem{Allaire2002}
{\sc G.~Allaire}, {\em Shape optimization by the homogenization
  method}. Applied Mathematical Sciences, 146. Springer-Verlag, New
York, 2002.

\bibitem{Alm-Ste_20}
{\sc S.~Almi and U. Stefanelli}, {\em Topology optimization for incremental
  elastoplasticity: a phase-field approach}, SIAM J. Control Optim.,
59 (2021), pp.~339--364.

\bibitem{Bensdsoe}
{\sc M.~P.~Bends\o e and O.~Sigmund}, {\it Topology optimization}, Springer-Verlag, Berlin, 2003. Theory, methods and applications.



\bibitem{Blank1}
{\sc L.~Blank, H.~Garcke, C.~Hecht, and C.~Rupprecht}, {\em Sharp interface
  limit for a phase field model in structural optimization}, SIAM J. Control
  Optim., 54 (2016), pp.~1558--1584.


\bibitem{Blank2}
{\sc L.~Blank, H.~Garcke, M.~H. Farshbaf-Shaker, and V.~Styles}, {\em Relating
  phase field and sharp interface approaches to structural topology
  optimization}, ESAIM Control Optim. Calc. Var., 20 (2014),
pp.~1025--1058.


\bibitem{Boissier}
{\sc M.~Boissier, J.~Deaton, P.~Beran, and N.~Vermaak}, {\em Elastoplastic Topology Optimization and cyclically loaded structures via direct methods for shakedown}, Struct. Multidisc. Optim. (2021). https://doi.org/10.1007/s00158-021-02875-6
 

\bibitem{Bourdin}
{\sc B.~Bourdin and A.~Chambolle}, {\em Design-dependent loads in topology
  optimization}, ESAIM Control Optim. Calc. Var., 9 (2003),
pp.~19--48.
 
\bibitem{Burger}
{\sc M.~Burger and R.~Stainko}, {\em Phase-field relaxation of topology
  optimization with local stress constraints}, SIAM J. Control Optim., 45
  (2006), pp.~1447--1466.


\bibitem{Carraturo}
{\sc M.~Carraturo, E.~Rocca, E.~Bonetti, D.~H\"{o}mberg, A.~Reali, and
  F.~Auricchio}, {\em Graded-material design based on phase-field and topology
  optimization}, Comput. Mech., 64 (2019), pp.~1589--1600.


  
  \bibitem{delosReyes}
{\sc J.~C. de~los Reyes, R.~Herzog, and C.~Meyer}, {\em Optimal control of
  static elastoplasticity in primal formulation}, SIAM J. Control Optim., 54
  (2016), pp.~3016--3039.
  
  \bibitem{Groeger}
{\sc K.~Gr\"{o}ger}, {\em A {$W^{1,p}$}-estimate for solutions to mixed
  boundary value problems for second order elliptic differential equations},
  Math. Ann., 283 (1989), pp.~679--687.
  
  \bibitem{Han}
{\sc W.~Han and B.~D.~Reddy}, {\em Plasticity}, Interdisciplinary
Applied Mathematics, Springer, New York, 2013.


\bibitem{haslinger1}
{\sc J.~Haslinger and P.~Neittaanm\"aki}, {\em On the existence of optimal shapes in contact problems--perfectly plastic bodies}, Comput. Mech., 1 (1986), pp.~293--299.

\bibitem{haslinger2}
{\sc J.~Haslinger, P.~Neittaanm\"aki, and T.~Tiihonen}, {\em Shape optimization in contact problems. 1. Design of an elastic body. 2. Design of an elastic perfectly plastic body}, Analysis and Optimization of Systems, Springer, 1986, pp. 29--39.




\bibitem{Hlavacek1}
{\sc I.~Hlav\'{a}\v{c}ek}, {\em Shape optimization of elastoplastic bodies obeying {H}encky's
              law}, Apl.~Mat., 31 (1986), pp.~486--499.

\bibitem{Hlavacek2}
{\sc I.~Hlav\'{a}\v{c}ek}, {\em Shape optimization of an elastic-perfectly plastic body}, Apl.~Mat., 32 (1987), pp.~381--400.
	
\bibitem{Hlavacek3}
{\sc I.~Hlav\'{a}\v{c}ek}, {\em Shape optimization of elastoplastic axisymmetric bodies}, Appl.~Math., 36 (1991), pp.~469--491.
  
 

  
\bibitem{Herzog}
{\sc R.~Herzog, C.~Meyer, and G.~Wachsmuth}, {\em Integrability of displacement
  and stresses in linear and nonlinear elasticity with mixed boundary
  conditions}, J. Math. Anal. Appl., 382 (2011), pp.~802--813.
  
  \bibitem{Herzog2}
{\sc R.~Herzog, C.~Meyer, and G.~Wachsmuth}, {\em C-stationarity for
  optimal control of static plasticity with linear kinematic hardening}, SIAM
  J. Control Optim., 50 (2012), pp.~3052--3082.



 \bibitem {Karkauskas}
 {\sc R.~Karkauskas}, {\em Optimization of elastic-plastic geometrically non-linear lightweight structures under stiffness and stability constraints}, J.~Civil Engrg.~Manag., 10 (2004), pp.~97--106.
  
\bibitem{Khanzadi}
{\sc M.~Khanzadi and M.~Tavakkoli}, {\em Optimal plastic design of frames using evolutionary structural optimization}, Int.~J.~Civil Engrg., 9 (2011), pp.~175--170.



  
  \bibitem{Krejci}
  {\sc P.~Krej\v{c}\'{i}}, {\em Evolution variational inequalities and multidimensional hysteresis operators},
Technical Report 432, Weierstrass Institute for Applied Analysis and
Stochastics (WIAS), 1998.





\bibitem{Maury}
{\sc A.~Maury, G.~Allaire, and F.~Jouve}, {\em Elasto-plastic shape
  optimization using the level set method}, SIAM J.~Control Optim., 56
(2018), pp.~556--581.


\bibitem{Mielke05}
{\sc A. Mielke}, {\em Evolution in rate-independent systems (ch. 6)},
in C. Dafermos and E. Feireisl, editors, Handbook of Differential
Equations, Evolutionary Equations, 2, 461-559. Elsevier B.V., 2005.



\bibitem{Mielke-Roubicek}
{\sc A.~Mielke and T.~Roub\'{\i}\v{c}ek}, {\em Rate-independent systems},
  vol.~193 of Applied Mathematical Sciences, Springer, New York, 2015.
\newblock Theory and application.

\bibitem{mrs} 
{\sc A.~Mielke, T.~Roub\'i\v cek, and U.~Stefanelli},
{\em $\Gamma$-limits and relaxations for rate-independent evolutionary
problems}, Calc. Var. Partial Differential Equations, 31
(2008), pp.~387--416. 


\bibitem{MR0445362}
{\sc L.~Modica and S.~Mortola}, {\em Un esempio di {$\Gamma
  ^{-}$}-convergenza}, Boll. Un. Mat. Ital. B (5), 14 (1977),
pp.~285--299.


\bibitem{Pedersen}
{\sc C.~B.~W.~Pedersen}, {\em Topology optimization of 2D-frame structures with path-dependent response}, Internat.~J.~Numer.~Methods Engrg., 57 (2003), pp.~1471--1501.




\bibitem{Penzler}
{\sc P.~Penzler, M.~Rumpf, and B.~Wirth}, {\em A phase-field model for
  compliance shape optimization in nonlinear elasticity}, ESAIM Control Optim.
  Calc. Var., 18 (2012), pp.~229--258.

\bibitem{Rindler}
{\sc F.~Rindler}, {\em Optimal control for nonconvex rate-independent evolution
  processes}, SIAM J. Control Optim., 47 (2008), pp.~2773--2794.

\bibitem{Rindler2}
{\sc F.~Rindler}, {\em Approximation of rate-independent optimal control
  problems}, SIAM J. Numer. Anal., 47 (2009), pp.~3884--3909.



  
  
  
  \bibitem{Sokolowski2}
{\sc J.~Soko{\l}owski and J.-P.~Zolesio}, {\em Introduction to Shape Optimization. Shape Sensitivity
Analysis}, Springer Ser.~Comput.~Math.~16, Springer-Verlag, Berlin, 1992.


  
  \bibitem{Wachsmut1}
  {\sc G.~Wachsmuth}, {\em Optimal control of quasi-static plasticity with linear
  kinematic hardening, {P}art {I}: {E}xistence and discretization in time},
  SIAM J. Control Optim., 50 (2012), pp.~2836--2861 + loose erratum.
  
  \bibitem{Wachsmut2}
{\sc G.~Wachsmuth}, {\em Optimal control of
  quasistatic plasticity with linear kinematic hardening {II}: {R}egularization
  and differentiability}, Z. Anal. Anwend., 34 (2015), pp.~391--418.

\bibitem{Wachsmut3}
{\sc G.~Wachsmuth}, {\em Optimal control of
  quasistatic plasticity with linear kinematic hardening {III}: {O}ptimality
  conditions}, Z. Anal. Anwend., 35 (2016), pp.~81--118.



\end{thebibliography}

\end{document}